\documentclass[dvips,imslayout]{imsart}

\RequirePackage[OT1]{fontenc}
\RequirePackage{amsthm,amsmath,natbib}
\usepackage{amsfonts}

\usepackage{epsfig,array,a4}
\usepackage{amssymb}

\startlocaldefs

\def\thesection {\arabic{section}}

\theoremstyle{plain}
\newtheorem{cor}{Corollary}[section]

\newtheorem{lem}{Lemma}[section]
\newtheorem{rem}{Remark}[section]
\newtheorem{thr}{Theorem}[section]

\newenvironment{pro}{\bf Proof.\rm\ }{\protect\nopagebreak$\mbox{}\hfill\Box$
\\}

\newcommand{\bea}{\begin{eqnarray*}}
\newcommand{\eea}{\end{eqnarray*}}
\newcommand{\beq}{\begin{equation}}
\newcommand{\eeq}{\end{equation}}

\newcommand{\e}{\epsilon}

\newcommand{\s}{\sigma}

\newcommand{\NN}{\mathcal{N}}

\def \R {\mathchoice {\hbox{I\kern -0.1667em R}}{\hbox{I\kern -0.1667em R}}
                      {\small{I\kern -0.1667em R}}{\small{I\kern -0.1667em R}}}

\def \eins {\mathchoice {\hbox{1\kern -0.27em l}}{\hbox{1\kern -0.27em l}}
                     {\small{1\kern -0.27em l}}{\small{1\kern -0.27em l}}}
\endlocaldefs

\begin{document}

\begin{frontmatter}
\title{Asymptotic Bayes optimality under sparsity for generally distributed effect sizes under the alternative}
\runtitle{ABOS for general distribution}

\begin{aug}
\author{\fnms{Florian} \snm{Frommlet}\thanksref{1}\ead[label=e1]{Florian.Frommlet@meduniwien.ac.at}},
\author{\fnms{Arijit} \snm{Chakrabarti}\thanksref{2}\ead[label=e2]{arc@isical.ac.in}},
\author{\fnms{Magdalena} \snm{Murawska}\thanksref{3}\ead[label=e3]{m.murawska@erasmusmc.nl}}
\and
\author{\fnms{Ma{\l}gorzata} \snm{Bogdan}\thanksref{4}\ead[label=e4]{Malgorzata.Bogdan@pwr.wroc.pl}}

\runauthor{F. Frommlet et al.}

\affiliation{Medical University of Vienna\thanksmark{1},  Indian Statistical Institute\thanksmark{2},Erasmus University\thanksmark{3}, Wroc{\l}aw University of Technology\thanksmark{4} }

\address{
Department of Medical Statistics\\
 Medical University of Vienna\\
 Spitalgasse 23\\
A-1090 Vienna, Austria\\
\printead*{e1}}

\address{203 B.T.Road\\
Bayesian and Interdisciplinary Research Unit\\
Indian Statistical Institute \\
Kolkata 700108, West Bengal,India\\
\printead*{e2}}

\address{Department of Biostatistics\\
Erasmus University Medical Center\\
PO Box 2040, 3000 CA Rotterdam,the Netherlands\\
\printead*{e3}}

\address{Department of
Mathematics and Computer Science\\
Wroc{\l}aw University of Technology\\
ul. Janiszewskiego 14a\\
50-370 Wroclaw, Poland\\
\printead{e4}}

\end{aug}

\begin{abstract}

Recent results concerning asymptotic Bayes-optimality under
sparsity (ABOS) of  multiple testing procedures are extended to
fairly generally distributed effect sizes under the alternative. 
An asymptotic framework is considered where
both the number of tests $m$ and the sample size $n$ go to infinity, while
the fraction $p$ of true alternatives converges to zero. It is shown that under  mild restrictions on the loss function nontrivial asymptotic inference is possible only if $n$  increases to infinity at least at the rate of $\log m$. Based on this assumption precise conditions are given under which the Bonferroni correction with nominal Family Wise Error Rate (FWER) level $\alpha$ and the Benjamini-
Hochberg procedure (BH) at FDR level $\alpha$ are asymptotically optimal.
When $n \propto \log m$ then $\alpha$ can remain fixed, whereas when $n$ increases to infinity at a quicker rate, then $\alpha$ has to converge to zero roughly like $n^{-1/2}$. 
Under these conditions the Bonferroni correction is ABOS in case
of extreme sparsity ($p\propto m^{-1})$, while BH adapts well to the unknown level of
sparsity. 

In the second part of this article these optimality results are carried
over to model selection in the context of multiple regression with
orthogonal regressors. Several modifications of Bayesian Information
Criterion are considered, controlling either FWER or FDR, and conditions
are provided under which these selection criteria are ABOS.
Finally the performance of these criteria is examined in a brief simulation study.

\end{abstract}



\end{frontmatter}

\section{Introduction} \label{Sec:Intr}

Driven by a vast number of applications, over the last few years multiple hypothesis testing with sparse alternatives has become a topic of intensive research (see,  \cite{A}, \cite{CJ}, \cite{Djin}, \cite{Djin2}, \cite{JinCai} or \cite{MR}).  As a result of this interest many new multiple testing procedures have been proposed, which can be compared according to several different optimality criteria. In the classical context a multiple testing procedure is considered to be {\it optimal} if it maximizes the number of true discoveries, while keeping one of the type I error measures (like Family Wise Error Rate, False Dicovery Rate or the expected number of false positives) at a certain, fixed level (see, \cite{GR}, \cite{Lehmann}, \cite{Chi}, \cite{Roquain},  \cite{Pena}, \cite{Finner}, \cite{S3} or \cite{SCai}). A different notion of optimality is proposed in \cite{SB} and \cite{Bog2009}, which investigate multiple testing procedures in the context of minimizing the  Bayes risk.

 In many  applications of high-dimensional multiple testing it is assumed that the proportion $p$ of true alternative hypotheses among all tests is very small. In asymptotic analysis this is often expressed by the sparsity assumption, that $p$ decreases to 0 as the total number of tests $m$ increases to infinity.
Recently, substantial efforts have been made to understand the asymptotic properties of multiple testing  under  sparsity  (see,  \cite{Djin}, \cite{Djin2}, \cite{A}, \cite{Bog2009}).


Bogdan et al. \citep{Bog2009} consider the problem of testing hypotheses about means $\mu_i$ in normal populations $X_i \sim N(\mu_i, \sigma^2)$, $i=1,\ldots,m$. Their analysis is based on a two-groups model, which assumes that the unknown means are generated by the scale mixture of two normal distributions: null and alternative.  The classical case of testing $H_{0i}: \mu_i=0$ corresponds to the situation when the variance of the null distribution is equal to 0.   
In \citep{Bog2009} the ratio $u$ of  variances of the alternative distribution of $\mu_i$ and the null distribution of $X_i$, slowly increases to infinity as $p\rightarrow 0$, at a rate which  guarantees that the limiting power of the Bayes classifier is larger than 0 and smaller than 1. Such sequences of alternative distributions are considered to be ``on the verge of detectability''. The Bayes risk is computed assuming that losses generated by the type I and type II errors are the same for all tests, and the total loss is the sum of losses for individual tests. In case of known $p, \s^2$ and $u$ the risk is minimized by using Bayes classifiers for each individual test. This optimal rule, which is in practice unattainable, is referred to as the Bayes oracle.

Under the described asymptotic assumptions a multiple testing rule is classified as asymptotically Bayes optimal under sparsity (ABOS) if the ratio of the corresponding Bayes risk and the risk of the Bayes oracle converges to one.
 Bogdan et al. \citep{Bog2009} characterize the class of multiple testing rules with fixed threshold which are ABOS, and they provide conditions under which the Bonferroni correction and the popular Benjamini--Hochberg multiple testing procedure (BH, \citep{BH}) are asymptotically optimal.

In the first part of this paper we extend the results of \citep{Bog2009} concerned with testing $H_{0i}:\mu_i=0$ to the case when the distribution of $\mu_i$ under the alternative $\nu(\mu)$ is fixed and not necessarily normal, while the number of individuals $n$ used to calculate  the test statistics $\bar X_i=\frac{1}{n} \sum_{j=1}^n X_{ij}$ increases with $m$.  It turns out that, given $p\propto m^{-\beta}$, signals are at the verge of detectability exactly when $n \propto \log m$. This situation is notably relevant in the context of bioinformatics data, where $n$ is usually much smaller than $m$.  We show that in this case  BH and the Bonferroni correction are ABOS  under the same assumptions  as in \cite{Bog2009}.  In particular, we show that if $\nu(\mu)$ has a positive and bounded density on the real line then the Bonferroni correction at a fixed FWER $\alpha \in (0,1)$ is ABOS if $p\propto m^{-1}$ and the ratio of losses for the false positive and the false negative $\delta$ decreases to 0 at such a rate that $\log \delta=o(\log m)$.   In contrast BH at a fixed FDR level $\alpha \in (0,1)$ adapts very well to the unknown level of sparsity and is ABOS  whenever $p\propto m^{-\beta}$, $\beta \in (0,1]$. As explained in \citep{Bog2009} the assumption of decreasing $\delta$ is quite reasonable since the cost of missing a true signal usually increases when the total number of signal decreases.   We also show that if $p\propto m^{-\beta}$ with $\beta \in (0,1]$ then the step-down version of the FDR controlling procedure, SD, is ABOS under the same conditions as BH.

Unlike in \cite{Bog2009} we also consider the case where the power of the Bayes oracle converges to 1.  For $p\propto m^{-\beta}$ this relates to the case where $n$ increases to infinity at a quicker rate than $\log m$. We show that in this case BH and SD are ABOS for any $\beta \in (0,1]$ as long as FDR levels decrease to 0 approximately at the rate of $n^{-1/2}$, while $\delta$ is bounded from above  and such that $\log \delta=o(\log m)$. Similarly,  the Bonferroni  correction is ABOS if its FWER converges to zero at the rate $n^{-1/2}$ and $p \propto \frac{1}{m}$. In this case the only assumption on  $\nu$ is that it has a positive  and bounded density in a neighborhood of 0.
Extending the results of \citep{Bog2009} to a more general class of distributions is based on techniques introduced by \citep{JT}, where nontrivial modifications are required to deal with sparsity.

In the second part of the paper we use the results on multiple testing rules to prove asymptotic optimality of some model selection criteria for sparse least squares regression. Here we concentrate on the orthogonal design and study the two cases of known and unknown variance of the error term $\sigma^2$. As discussed in \citep{BGOT}, in case of  orthogonal design with known $\sigma$, penalized likelihood model selection criteria work analogously to  multiple testing procedures which verify individually the significance of each regression coefficient. Based on this analogy it is very easy to prove  that  popular model selection criteria, like AIC \cite{Ak} or BIC \cite{S}, are not consistent when $\frac{m}{\sqrt{n}}$ increases to infinity (see \citep{BGOT}). Specifically, under this scenario the expected number of false discoveries increases to infinity.

To solve this problem some modifications of  AIC   \citep{mAIC} and BIC (see, \citep{BGD, Chen1}) were recently proposed in the literature. In this article we will concentrate on modifications of BIC, which is more appropriate to consider when one aims at minimizing the misclassification rate, or in our context the Bayes risk based on a generalized 0-1 loss. The first of the considered criteria, mBIC, was derived in \citep{BGD} in a Bayesian setting using a prior on the model dimension which assumes that the expected number of true regressors does not depend on $m$.
In case of orthogonality and known $\sigma$ it was pointed out in \citep{BGOT}  that mBIC is controlling the FWER. Optimality results at a sparsity level $p\propto m^{-1}$ follow immediately from the analysis for multiple testing.

 In view of   results on multiple testing it would actually be of great interest to study model selection criteria which control the FDR.   In \cite{A}  penalized  model selection schemes are discussed which  have  exactly this property. Quite similar penalties have been discussed in \cite{FG} and \cite{GF}.  Starting from the penalty of  \cite{A} we will introduce several new modifications of BIC (mBIC1 - mBIC3), where mBIC2 has been shown already to perform very well in the application of genome wide association studies \cite{FRTB}.  In case of known $\sigma$ we prove that the FDR controlling criteria are ABOS  for a wide range of sparsity levels, satisfying for example $p=m^{-\beta}$, with $\beta \in (0,1]$.

In most applications it is much more realistic to assume that $\sigma$ is not known. Under sparsity  it is  rather difficult to get reliable estimates on $\sigma$, and for that reason optimality results on the corresponding model selection criteria under sparsity are very rare in the literature. In a Bayesian approach with normally distributed error terms, $\sigma$ is integrated out and in the corresponding version of BIC the residual sum of squares $RSS$ is replaced by $\log RSS$.
We will show that in this context mBIC is again ABOS  in case of extreme sparsity. The conditions we need for  unknown $\sigma$ are not much more restrictive than for known  $\sigma$. Our proof is technically rather involved, and cannot be easily extended to prove ABOS for mBIC1 - mBIC3. However, in analogy to the case of known variance we conjecture that these criteria should be ABOS for a wide range of sparsity levels. This conjecture is underpinned by simulations, which show good properties of the new versions of mBIC both for known and unknown $\sigma$.

The rest of the paper is organized as follows. In Section \ref{Sec:MultipleTest} we present results for multiple testing, whereas  Section \ref{Sec:Regression} focuses on linear regression models under orthogonality.
The main emphasis of Sections \ref{Sec:Fixed} and \ref{Sec:BFDR} is the generalization of results from \cite{Bog2009} to the situation of general distributions under the alternative.  Section \ref{Sec:Bonf} shows ABOS of Bonferroni correction in case of extreme sparsity. The most important theorems on multiple testing are given in Section \ref{Sec:FDR}, where ABOS of step-up and step-down FDR controlling procedures is proven. These results are needed in Section \ref{Sec:mBIC2} to show ABOS of the FDR-controlling model selection criteria, after ABOS of mBIC for known variance was shown in Section \ref{Sec:mBIC}. Optimality results of mBIC for unknown variance are proved in Section \ref{Sec:mBIC_su}. Finally in Section \ref{Simulation} different model selection criteria are compared in a small simulation study. Most proofs of technical results can be found in the Appendix.

\section{ABOS for multiple testing rules}   \label{Sec:MultipleTest}

Consider a set of $m$ normal populations $\NN(\mu_i, \sigma^2)$, $i=1,\ldots,m$.  We are interested in testing  point null hypotheses $H_{0 i}:\mu_i = 0$ against the alternatives $H_{Ai}:\mu_i\neq 0$, based on simple random samples $X_i=(X_{1i},\ldots, X_{ni})$  of size $n$ from each of these populations.
The effects under study $\mu_i$ are supposed to be independent and identically distributed according to a mixture distribution
\beq \label{mixture}
\nu_{mix} = (1-p) d_0 + p \nu \;,
\eeq
 where $d_{0}$ is the Dirac measure at 0, $\nu$ is a  probability measure on the real line describing the distribution of $\mu_i$ under the  alternative, and $p\in(0,1)$ is the proportion of alternatives among all tests. Since $\nu$ describes  the alternative distribution of the different $\mu_i$, we assume that $\nu(\{0\})=0$. Furthermore  both positive and negative values of $\mu_i$ should be possible, that is
 \begin{equation}\label{as_both}
 \nu(-\infty,0)>0\;\;\;\mbox{and}\;\;\;\nu(0,\infty)>0\;.
 \end{equation}

From (\ref{mixture}) it easily follows that the marginal distribution of the sample mean $\bar X_i=\frac{1}{n}\sum_{j=1}^n X_{ji}$  is the mixture
\beq \label{barX}
\bar X_i \sim (1-p) \NN(0,\sigma^2/n) + p\ \left(\nu * \NN(0,\sigma^2/n)\right) \; ,
\eeq
where the  pdf of the second measure is computed by convolution of $\nu$ and $\NN(0,\sigma^2/n)$.

Our decision theoretic framework for multiple testing is based on a generalization of the standard 0-1 loss. There are $m$ decisions to be made. For each false rejection (type I error) we assign a loss of $\delta_0$, and for missing a true signal (type II error) a loss of $\delta_A$. The total loss of a multiple testing procedure is then defined as the sum of losses for individual tests \citep{Leh1}. The total loss is clearly minimized by applying the Bayes  classifier to each individual test, the decision rule which was called Bayes oracle in \citep{Bog2009}.

Hence our first task is to determine the critical values $a_n$ and $b_n$ corresponding to the Bayes classifier for each individual test.  As noted in \cite{JT}, if $p\in(0,1)$ then for any measure $\nu$ satisfying (\ref{as_both}) and  sufficiently large $n$,  the Bayes classifier chooses $H_{0i}$ if
$\bar X_i \in (a_n, b_n)$, where the critical values $a_n$ and $b_n$ are uniquely defined by
\begin{eqnarray}\label{JT1}
&& a_n < 0 < b_n \nonumber \\
&& (1-p) \delta_0 = p\ \delta_A  \int\limits_{\mathcal{R}  } \exp\left(n(a_n \frac{\mu}{\sigma^2} -
\frac{\mu^2}{2 \sigma^2})\right)\ d\nu(\mu) \; , \\
&& (1-p) \delta_0 = p\ \delta_A  \int\limits_{\mathcal{R}  } \exp\left(n(b_n \frac{\mu}{\sigma^2} -
\frac{\mu^2}{2 \sigma^2})\right)\ d\nu(\mu) \; . \nonumber
\end{eqnarray}

Let $\delta = \delta_0/\delta_A$ denote the ratio of type I error and type II error losses, and let $f = (1-p)/p$ which serves as a measure of sparsity. In the forthcoming asymptotic analysis we will assume that $m\rightarrow \infty$ and that $n=n_m\rightarrow \infty$. Furthermore we will allow the parameters $\delta=\delta_m$ and $p=p_m$ to depend on $m$, whereas $\sigma$  and $\nu$ are kept fixed. For  simplicity of notation the index $m$ will be omitted for $n$, $\delta$ and $p$.  The most generic situation will be  $p \rightarrow 0$, in which case $f\rightarrow \infty$.  However,  theorems are formulated  in the more general setting under the following assumption:

\vspace{2mm}

\noindent {\bf Assumption (A)}:  $\ \ n \rightarrow\infty$, $\delta f  \rightarrow c \in (0,\infty]$, and $\frac{2\log (
\delta f)}{n}\rightarrow C$, where $0\leq C<\infty$.

\vspace{2mm}

\begin{rem}
{\rm

Under the model assumptions of \citep{Bog2009}  ``signals on the verge of detectability'' had to  satisfy $\frac{2\log (\delta f)}{u}\rightarrow C\in(0,\infty)$, which yielded asymptotic power of the Bayes oracle  within $(0,1)$. 
Here we are concerned with a  different situation, where the alternative distribution for $\mu_i$ is not necessarily normal and does not depend on $p$, but the number $n$ of individuals increases to infinity.   In this setting the role of $u$ is taken by $n$. Compared to the assumptions in \citep{Bog2009}  the  major  difference is that we additionally consider the case $\frac{2\log (\delta f)}{n}\rightarrow C= 0$,  which means that the asymptotic power of the Bayes oracle is equal to 1. This additional case covers the  interesting scenario where sparsity is of the form $p = m^{-\beta}, \ \beta > 0$, $\log \delta=o(\log m)$ and $n\in(m^{c_1}, m^{c_2})$, for any positive constants $c_1<c_2$. 
 }
\end{rem}

The generic situation will be concerned with sparsity and with the loss ratio $\delta$ having no dominating influence on the asymptotic results. We formalize this  in

\vspace{2mm}

\noindent {\bf Assumption (B)}:  $\ \  n \rightarrow\infty, p  \rightarrow 0$,  $\log \delta = o(\log p)$ and $\delta$  bounded from above.

\vspace{4mm}

\noindent If Assumption (B) holds, then  $-\frac{2 \log p}{n} \rightarrow C \geq 0$ is enough to guarantee that Assumption (A) is fulfilled. All theorems in Section \ref{Sec:Regression} are formulated under Assumption (B).

\vspace{2mm}

The following assumption imposes a restriction on the measure $\nu$, which will be used throughout this mansucript.

\vspace{2mm}

\noindent {\bf Assumption (C)}: $\ \ $
Let $T := \sigma \sqrt{C}$. We assume that there exists $\e > 0$ such that $\nu$ has a positive bounded density $\rho$ with respect to Lebesgue measure on $[-T - \e, -T +\e]$ and $[T - \e, T +\e]$. In case of $C = 0$ it is further assumed that  $\rho(0^-) := \lim\limits_{\mu \uparrow 0} \rho(\mu)$ and   $\rho(0^+) := \lim\limits_{\mu \downarrow 0} \rho(\mu)$ both exist and are  finite and positive.

\vspace{2mm}

The following Lemma provides  the asymptotic critical points of the  Bayes rule for distributions $\nu$ satisfying  Assumption (C).

\begin{lem} \label{LemF1}

Let Assumptions (A) and (C) hold. Then the critical values converge with limits
$$
a_n \rightarrow -T \mbox{ and } b_n \rightarrow T\ .
$$
\end{lem}

The proof is given in Appendix \ref{Sec:App0}.

\vspace{2mm}

{\bf Notation:} Throughout the paper we will make use of the following notation: Let $g_n$ and $h_n$ be two sequences. Then $g_n \sim h_n$ indicates that $\frac{g_n}{h_n}   \rightarrow 1$ as  $n \rightarrow \infty$. If $g_n \rightarrow 0$ we write $g_n = o_n$.

\vspace{2mm}

The following Lemma \ref{LemF2} specifies the rate at which $a_n$ and $b_n$ converge to zero in case of $C=0$.

\begin{lem} \label{LemF2}
Let Assumptions (A) and (C) hold. If $\ C=0$ then the critical values of the Bayes oracle fulfill
\begin{equation} \label{f0a}
 \sqrt{n}  e^{- \frac{n a_n^2}{2\sigma^2}} \sim \frac{\sqrt{2 \pi} \sigma}{f \delta}  \rho(0^-)
\end{equation}
and
\begin{equation} \label{f0b}
\sqrt{n}  e^{- \frac{n b_n^2}{2\sigma^2}} \sim  \frac{\sqrt{2 \pi} \sigma}{f \delta}   \rho(0^+)   \; .
\end{equation}
\end{lem}

The proof is given in Appendix \ref{Sec:App1}.

\begin{rem}
{\rm
As shown in the proof of Lemma \ref{LemF2}, the accuracy of the approximations provided in (\ref{f0a}) and (\ref{f0b})  depends on the asymptotic behavior of $\delta f$  and on the regularity of $\rho$ in a neighborhood of 0. Assuming for example that $\rho$ is one-sided Lipschitz (on both sides of 0) and that $\delta f$ is polynomially bounded one obtains that the ratio of the right and left-hand sides of (\ref{f0a}) and (\ref{f0b}) can be expressed as $1+z_n$ with   $z_n=o(n^{-1/2} \log n)$.
}\end{rem}
\begin{rem}\label{v}
{\rm The results of Lemmas \ref{LemF1} and \ref{LemF2} generalize the critical value of the Bayes rule specified  in  \citep{BCFG} and  \citep{Bog2009}. Note that for $\nu \sim {\cal N}(0,\tau^2)$ the ``magnitude'' of the true signal defined in \citep{Bog2009} is given by $u=\frac{n\tau2}{\sigma^2}$.  Thus, according to Lemma \ref{LemF1}, for $C>0$ the Bayes classifier rejects the null hypothesis if
$$ \frac{n \bar X_n^2}{\sigma^2}>  \log(u f^2 \delta^2)(1+o_n)\;\;,$$
which agrees with the results of \citep{Bog2009}. 

 Next consider the case $C=0$. For normal distribution $\mu_i \sim {\cal N}(0,\tau^2)$ it holds that
$\rho(0^-)  = \frac{1}{\sqrt{2 \pi} \tau}$. Taking logarithms of (\ref{f0a}) we  obtain the accurate approximation
$$
\frac{n a_n^2}{\sigma^2} =  2 \log\left( \frac{\delta f \sqrt n}{\sqrt{2 \pi}\  \sigma  \rho(0^-)}  \right)+o_n
= \log( u f^2 \delta^2)+o_n\;\;
$$
and because of $\rho(0^-) = \rho(0^+)$ the same relation holds for $b_n$.
}\end{rem}

\vspace{3mm}

To emphasize similarity with the results for normal scale mixture models from \citep{Bog2009} we introduce the notation
$$
v:= n \delta^2 f^2\;\;.
$$

Then according to Lemmas \ref{LemF1} and \ref{LemF2}  the Bayes oracle threshold values  satisfy
\begin{equation}\label{cv_general}
a_n = -\sigma \sqrt{\frac{\log v }{n}}(1+o_n)\mbox{ and }
b_n = \sigma \sqrt{\frac{\log v}{n}}(1+o_n)\ .
 \end{equation}


The risk for a multiple testing rule is computed under the additive loss of individual tests simply as the sum of the risks of individual tests. Note that for the specified mixture model (\ref{barX}) type I error $t_1$ and type II error $t_2$ of fixed threshold rules are identical for each individual test. The corresponding  risk is therefore defined as
\beq \label{risk1}
R = R_1 + R_2 = m(1-p) t_1 \delta_0 + m p t_2 \delta_A \; .
\eeq
In the following theorem we compute the asymptotic risk $R^B$ of the Bayes oracle.

\begin{thr}\label{TH1_general}
Under Assumptions (A) and (C) the risk obtained by the Bayes rule (\ref{JT1}) takes for $C=0$   the form
\begin{equation}\label{optrisk_general}
R^B=m p \delta_A \sigma  \sqrt{\frac{\log v}{n} }\left(\rho(0^-) +
\rho(0^+)\right) (1 + o_n)
\end{equation}
whereas for  $0 < C < \infty$
\begin{equation}\label{optrisk1_general}
R^B = m p \delta_A  \  \nu(-T,T) (1 + o_n) \ .
\end{equation}
\end{thr}

The proof is given in  Appendix \ref{Sec:App2}.

\ \\

{\bf Definition:}
A multiple testing rule is called asymptotically Bayes optimal under sparsity (ABOS) if its risk $R$  satisfies   $\frac{R}{R^B} \rightarrow 1$ under the conditions of Assumption (A).

\subsection{ABOS of fixed threshold rules} \label{Sec:Fixed}

The next theorem describes which multiple testing rules with fixed threshold are ABOS.

\begin{thr} \label{TH2_general}
Consider the testing rule which rejects $H_{0i}$ if $\bar X_i$ falls out of the interval $(\tilde a_n, \tilde b_n)$, with $\tilde a_n<0$ and $\tilde b_n>0$.  Under Assumptions (A) and (C) this rule is ABOS if and only if
\begin{equation}\label{optcv_general}
\frac{n\tilde a_n^2}{\sigma^2} = \log v + z_a \quad \mbox{and}\quad  \frac{n\tilde b_n^2}{\sigma^2} = \log v + z_b
\end{equation}
where
\begin{equation}\label{optcv1_general}
z_a = o(\log v) \; , \quad z_b = o(\log v) \; ,
\end{equation}
 and
\begin{equation}\label{optcv2_general}
\lim_{n\rightarrow \infty}  z_a +  2 \log \log v = \infty \; , \quad
\lim_{n\rightarrow \infty}  z_b +  2 \log  \log v = \infty \ .
\end{equation}
\end{thr}
The proof is given in  Appendix \ref{Sec:App3}.
\ \\

As a simple consequence of  Theorem \ref{TH2_general} we have
\begin{cor} \label{Cor_optcv}
Suppose that additional to the assumptions of Theorem \ref{TH2_general}  also Assumption (B) holds. If for $n = n_m$  the sparsity assumption
\beq\label{as_spars}
m p \rightarrow s \in (0,\infty], \quad  \frac{\log (m p)}{\log(n/p^2)} \rightarrow 0\;,
\eeq
is fulfilled, then  thresholds of the form
\begin{equation}\label{optcv_pm}
c_a^2 \sim c_b^2 = \log(n m^2) + \xi, \quad \xi = o(\log(n/ p^2))
\end{equation}
yield multiple testing rules which are ABOS, whenever $2 \xi \geq  -2\log(mp) + d$ for some arbitrary constant $d$. In particular this is the case when $\xi$ is a  constant.
\end{cor}
\begin{pro}
Simply observe that $z = \log(n m^2) + \xi - \log(n p^{-2} \delta^2)$ fulfills the requirements of Theorem \ref{TH2_general} under the assumption of the corollary.
\end{pro}
\begin{rem}
{\rm    \label{Rem:Cor_optcv}
Corollary \ref{Cor_optcv}   addresses the situation of extreme sparsity, where the number $m$ of tests increases to infinity, but the  expected number of true signals remains constant or increases only very slowly with $m$. If additionally $\log n = o(\log m)$ then Corollary \ref{Cor_optcv}  implies that   the universal threshold $2\log m$ of \cite{DJ1} is ABOS. This extends  Remark 3.4 of \citep{Bog2009} to the case where the distribution of $\mu_i$ under the alternative is not necessarily normal and does not change with $m$, while the number of individuals $n$ slowly increases with $m$.
}\end{rem}


\subsection{BFDR  controlling procedures}  \label{Sec:BFDR}


One of our main goals is to study ABOS of FDR controlling procedures like the popular Benjamini--Hochberg procedure (BH,\citep{BH}). As in \citep{Bog2009}  the main technical tool to prove ABOS is to approximate the random threshold of BH by the threshold from a rule controlling the Bayesian false discovery rate (BFDR, see \citep{ET}). For that reason we will start our discussion here with results on the asymptotic properties of BFDR rules for general distributions of $\mu_i$ under the alternative.
BFDR is defined as

\begin{equation}\label{BFDR}
BFDR=P(H_{0i}\;\; \mbox{is true}| H_{0i}\;\;\mbox{was
rejected})=\frac{(1-p)t_{1i}}{(1-p)t_{1i}+ p\ (1-t_{2i})}\;\;,
\end{equation}
where $t_{1i}$ and $t_{2i}$ are the probabilities of the corresponding type I and type II errors.
Consider a fixed threshold rule based on $\bar X_i$ with threshold values $a<0$ and $b>0$. Then $t_{1i} = t_1$, $t_{2i} = t_2$, and  under the mixture model (\ref{barX})
$$
t_1 = \Phi(\sqrt n a/\sigma)+1-\Phi(\sqrt n b/\sigma)
$$
and
$$
t_2 = 1-\int\limits_{\mathbb R}(\Phi(\sqrt n( a-\mu)/\sigma)+1-\Phi(\sqrt n (b-\mu)/\sigma))d \nu(\mu) \; .
$$
To obtain threshold values $a_n^B<0$ and $b_n^B>0$ with BFDR level $\alpha$ we have to solve  $\frac{(1-p)t_1}{(1-p)t_1+p(1-t_2)} = \alpha$,  or equivalently
$$\frac{\alpha}{f(1 -\alpha)} = \frac{\Phi(\sqrt n a^B_n/\sigma)+1-\Phi(\sqrt n b^B_n/\sigma)}
 {\int\limits_{\mathbb R} (\Phi(\sqrt n (a^B_n - \mu)/\sigma) + 1- \Phi(\sqrt n (b^B_n - \mu)/\sigma))d \nu(\mu)}\;\;.$$
 We will restrict our attention to  rules based on  symmetric thresholds, such that $a_n^B=-b_n^B$, and use 
\beq \label{BFDR_thresh}
c_B^2 = c_B^2(n) := \frac{n \left(a^B_n\right)^2}{\s^2}  = \frac{n \left(b^B_n\right)^2}{\s^2} \;  
\eeq
to denote the corresponding threshold for the scaled test statistics $Z_i=\frac{n \bar X_i^2}{\sigma^2}$.
Then $c_B$ satisfies the following equation 
\begin{equation}\label{BFDR_rule}
 \frac{\alpha}{f(1 -\alpha)} = \frac{ 2(1 - \Phi(c_B))}{ 2 - \int\limits_{\mathbb R} [\Phi(c_B + \sqrt n \mu/\sigma)+\Phi(c_B - \sqrt n \mu/\sigma)]d \nu (\mu)}\;\;.
\end{equation}
 As shown in Lemma \ref{final} in Appendix \ref{Sec:ApBFDR},  $\alpha \in (0,1-p)$ guarantees existence and uniqueness of a solution $c_B$ for (\ref{BFDR_rule}). 
The following theorem provides conditions on $\alpha$, for which the BFDR controlling rule specified in (\ref{BFDR_rule}) is ABOS.

\begin{thr} \label{TH_BFDR}
Additional to Assumptions (A) and (C) suppose that \\ $\alpha \in (0, 1 - p),\ \alpha \rightarrow \alpha_\infty<1$, and
\begin{equation}\label{con3_general}
f/\alpha \rightarrow \infty, \quad\frac{\log \left(\frac{f}{\alpha}\right)}{n}\rightarrow C_0<\infty\;\;,
\end{equation}
where $C_0$ is such that $\nu(-\sigma \sqrt{2 C_0}, \sigma \sqrt{2C_0})<1$ and $\nu$ has no atoms at $\pm \s \sqrt{2 C_0}$.
The threshold value $c_B$ of the rule controlling BFDR at level $\alpha$ is then given by
\begin{equation}\label{BFDR2_general}
c^2_B =  2\log \left(\frac{f}{\alpha}\right)
- \log \left(  2 \log \left(\frac{f}{\alpha}\right)\right) +
2 \log \left(\frac{\sqrt 2 \ (1 - \alpha_\infty)}{\sqrt \pi \ C_1}\right)  + o_n \;\;,
\end{equation}
where 
$$
C_1 = 1- \nu(-\sigma \sqrt{2 C_0}, \sigma \sqrt{2C_0})\;.
$$ 

The BFDR controlling rule is ABOS if and only if
\beq \label{BFDR_opt}
\frac{\log(f \delta \sqrt n)}{\log(f/\alpha)} \rightarrow 1, \ \
\mbox{ and } \ \ 2 \log(\alpha \delta \sqrt n) - \log \log(f/\alpha) \rightarrow -\infty \;\;.
\eeq
In that case $C_0 = C/2$ and therefore $C_1 = 1- \nu(-T, T)$.
\end{thr}

The proof is given in  Appendix \ref{Sec:App4}.
\ \\

\begin{cor} \label{Rem:sqrt_n}
{\rm
If in addition to the assumptions of Theorem \ref{TH_BFDR} also Assumption (B) holds  then  the fixed threshold rule with BFDR at the level  $\alpha \propto n^{-1/2}$ is ABOS.}
\end{cor}

\begin{cor} \label{Rem:alpha_const}
{\rm
If in addition to the assumptions of Theorem \ref{TH_BFDR} and Assumption (B) also  $\delta \rightarrow 0$ and $ n \propto -\log p$ then the fixed threshold rule with BFDR equal to $\alpha\in (0,1)$ is ABOS. It is not possible that a BFDR controlling rule is ABOS when both $\alpha$ and $\delta$ are constant.}
\end{cor}

\begin{rem}
{\rm Based on (\ref{BFDR2_general}) straight forward calculations yield the asymptotic type I error of the BFDR rule
\beq \label{BFDR_t1}
t_1^B = \frac{C_1\alpha}{(1 - \alpha_\infty) f} (1 + o_n) \; .
\eeq
}
\end{rem}

The BFDR controlling rules discussed in this section require the knowledge of some of the parameters  of the unknown mixture distribution and therefore they are not applicable in practice. However, the results on ABOS of the BFDR controlling rules can be used to prove  ABOS of some popularly used multiple testing rules, like the Bonferroni correction or the Benjamini--Hochberg procedure. Asymptotic optimality results of these rules will be presented in the following sections.

\subsection{Bonferroni correction} \label{Sec:Bonf}

In  applied sciences the most popular multiple testing procedure is still the fixed threshold rule of Bonferroni correction.  In our setting its critical value $c_{Bon}$ for the test statistic $\frac{\sqrt n |\bar X_i|}{\s}$ is defined  by
$$
1 - \Phi(c_{Bon}) = \frac{\alpha}{2 m}\; .
$$
The procedure controls the family wise error rate at level $\alpha$. The following lemma specifies the conditions for $\alpha$ under which
the Bonferroni procedure is ABOS.

\begin{lem}\label{L1}
Suppose Assumptions (A), (C) and sparsity condition (\ref{as_spars}) hold.
 The Bonferroni procedure  at  FWER level $\alpha_n$  is ABOS if $\alpha_n$ satisfies the assumptions of Theorem \ref{TH_BFDR}.
\end{lem}
\begin{proof}
If $m\rightarrow \infty$ then the threshold for the Bonferroni correction can be written as
$$
c^2_{Bon}=2\log\left(\frac{m}{\alpha}\right)-\log\left(2\log\left(\frac{m}{\alpha}\right)\right)+\log(2/ \pi)+o_n\;\;.
$$
 Comparison of this threshold with  the asymptotic approximation to an optimal BFDR control rule (\ref{BFDR_thresh}) and (\ref{BFDR2_general}) yields
$$
c^2_{Bon}=c^2_B+2\log m p + O_n(1) \;\;.
$$
From (\ref{as_spars}) it follows easily that $c^2_{Bon} = c^2_{B}(1+o_n)$.
By assumption, the rule based on the threshold $c^2_{B}$ is optimal, and hence $c^2_{Bon}$ satisfies condition (\ref{optcv1_general}) of Theorem \ref{TH2_general}. Condition (\ref{optcv2_general}) is satisfied, since by assumption $\log m p$ is bounded from below and thus ABOS of the Bonferroni correction follows.
\end{proof}


\subsection{FDR  controlling procedures}  \label{Sec:FDR}


The Benjamini--Hochberg rule \citep{BH}, which we will also call  step-up FDR controlling procedure,
is defined as follows:   For the square of the scaled test statistics
$Z_i^2 = \frac{n \bar X_i^2}{\sigma^2}$ one computes two-sided p-values $p_i = 2(1 - \Phi(|Z_i|))$ which are then ordered $p_{[1]} \leq  p_{[2]} \leq \dots \leq p_{[m]}$. For the step-up procedure at the FDR level $\alpha$  compute
\begin{equation}\label{step-up}
k_F := \max\left\{i:\; p_{[i]} \leq \frac{i \alpha}{m} \right\}
\end{equation}
and reject the $k_F$  hypothesis with p-values smaller or equal $p_{[k_F]}$. In view of the proof of ABOS for FDR controlling model selection criteria in Section \ref{Sec:mBIC2}  we will not only consider the step-up procedure, but also the corresponding step-down procedure at level $\alpha$. For this compute
\begin{equation}\label{step-down}
k_G := \min \left\{i:\; p_{[i]} > \frac{i \alpha}{m} \right\}
\end{equation}
and reject the $k_G - 1$  hypotheses with p-values smaller than $p_{[k_G]}$. It is well known, that in practice both procedures behave very similar (see \citep{A}).

Optimality results for the step-up FDR controlling rule were proven in \citep{Bog2009} under the assumption of $\mu_i$ being normally distributed. A crucial step was the definition of a random threshold for the BH rule
$$
c_{BH}=\min\{c_{Bon}, \tilde c_{BH}\}\;\;.
$$
with
\begin{equation}\label{cBH}
\tilde c_{BH}=\inf\left\{y:\frac{2(1-\Phi(y))}{1-\check F_m(y)}\leq \alpha\right\}\;\;.
\end{equation}
Here $1-\check F_m(y) = \# \{|Z_i|\geq y\}/m$.
Alternatively let us denote $1-\hat F_m(y) = \# \{|Z_i| > y\}/m$.
Similar as in case of BH it is easy to check  that SD rejects the null hypothesis $H_{0i}$ if and only if $Z_i^2\geq  c^2_{SD}$ where
\begin{equation}\label{cSD}
 c_{SD}=\sup\left\{y:\frac{2(1-\Phi(y))}{1-\hat F_m(y)+ 1/m} > \alpha\right\}\;\;.
\end{equation}

It was proven by Genovese and Wassermann (GW) in \citep{GW1} that for fixed $p$, as the number of tests increases, the random threshold $c_{BH}$ can be approximated by the non-random threshold
\begin{equation}\label{GW2}
c_{GW}: \frac{2(1-\Phi(c_{GW}))}{1-F(c_{GW})}=\alpha\;\;,
\end{equation}
where $F(y) = P(|Z_1| \leq y)$.

\begin{center}
\begin{figure}[t]
\begin{minipage}[t]{4.6cm}
\begin{center}
\leavevmode  \epsfxsize=6.5cm \epsffile{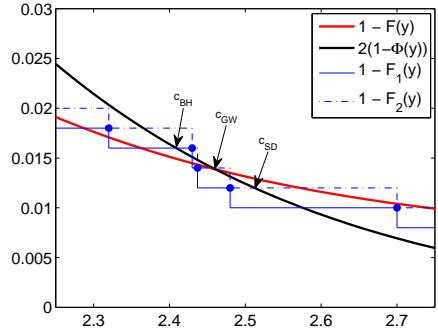}
\end{center}
\end{minipage}
\caption{Comparison of the  random thresholds $c_{BH}$ and $c_{SD}$ with the nonrandom threshold $c_{GW}$. In the legend $F_1$ refers to $\check F_m$ and $F_2$ refers to $\hat F$.} \label{Fig:thresholds}
\end{figure}
\end{center}

Figure \ref{Fig:thresholds} illustrates the thresholds $c_{BH}$,  $c_{SD}$
and $c_{GW}$.
Comparing $\tilde c_{BH}$ and $c_{SD}$ with $c_{GW}$ the major change is in replacing the cumulative distribution function of $|Z_i|$ by the corresponding empirical distribution function. In \citep{Bog2009} it was shown that also in case of sparsity
$c_{BH}$ can be well approximated by  $c_{GW}$, and in Lemma \ref{LSD2} of Appendix \ref{ApSec:Th2.4} we will see that the same is true for $c_{SD}$. A much simpler result is that under sparsity the difference between $c_{GW}$  and the corresponding BFDR controlling threshold $c_{B}$ becomes asymptotically negligible.
\begin{thr}\label{TH:GW}
Suppose Assumptions (A) and (C) are true and that $p\rightarrow 0$. Consider the rule rejecting the null hypothesis $H_{0i}$ if $\frac{n \bar X_i^2}{\sigma^2}\geq c^2_{GW}$. This rule is ABOS if and only if the corresponding BFDR controlling rule defined in (\ref{BFDR_rule}) (for the same $\alpha=\alpha_n$) is ABOS. In this case we have
 $$
 c^2_{GW}=c^2_{B}+o_n\;\;.
 $$
\end{thr}
\begin{pro}
The proof of this statement follows exactly as the proof of Theorem 4.2 of \citep{Bog2009}.
\end{pro}

The next theorem provides the optimality result of BH and SD for generally distributed effect sizes under the alternative.

\begin{thr} \label{TH:BH}
Apart from Assumptions (A) and (C)  assume that
\begin{equation}\label{pm}
m p\rightarrow s \in (0,\infty]
\end{equation}
 and
\begin{eqnarray}\label{alpha}
\alpha \mbox{ satisfies the conditions of Theorem \ref{TH_BFDR},  } \label{alpha2}
\\ \nonumber
\mbox{i.e the BFDR control rule at level $\alpha$ is asymptotically optimal.}
\end{eqnarray}
 For the denser case
\beq \label{dense}
p > \frac{\log^{\gamma_1} m}{m} \;\;,\mbox{ for some constant } \gamma_1 > 1
\eeq
the additional assumptions
\begin{equation}\label{ularge}
 n \leq m^{\gamma_2} ,\;\;\;\mbox{for some $\gamma_2 > 0$ and }\;\; \frac{\log \log m}{\log (p\ \alpha)} \rightarrow 0
\end{equation}
should hold.
Then both BH and SD are ABOS.
\end{thr}

\begin{pro}
BH is more liberal than SD, thus it is enough to control the risk contribution of Type 1 error  for BH, as well as the risk contribution of Type 2 error for SD.  Under the first condition in (\ref{ularge}) the proof for Type 1 error of BH follows  along the same lines as the proof of Lemma 5.4 in \citep{Bog2009}. Also, under the condition of extreme sparsity (\ref{as_spars}) according to Lemma \ref{L1} the Bonferroni procedure is ABOS. Therefore the optimality of the type II error component of the risk of SD in the extremely sparse case follows directly from a comparison with the
more conservative Bonferroni correction.
Finally, the necessary bound of the type II error component of the risk of SD  for the denser case (\ref{dense}) is provided in  Appendix \ref{ApSec:Th2.4}. This proof substantially relies on the second condition in (\ref{ularge}).
\end{pro}

\begin{rem}
{\rm The upper bound on $m$ provided in the second condition of (\ref{ularge}) is not very restrictive. Specifically, it is satisfied whenever $p\propto  m^{-\beta}$ with $\beta\in(0,1]$. For $p$ decreasing to 0 at a slower rate (for example like $(\log m)^{-1}$) one can replace this bound with the condition
\begin{equation}\label{sim_con}
n\geq  m^{\gamma_3}\;\;\mbox{ for some}\;\; \gamma_3>0\;\;.
\end{equation}
(It is easy to show that (\ref{sim_con}) implies the upper bound on $m$ in (\ref{ularge}) given the other assumptions of  Theorem \ref{TH:BH}) .
}
\end{rem}

The following Corollaries are easy consequences of Theorem \ref{TH:BH}.

\begin{cor}\label{cor1}
Suppose Assumptions (A) and (C) hold. If $p=m^{-\beta}$ with $\beta\in(0,1]$, $n\leq m^{\gamma_2}$ for some $\gamma_2>0$ and $\delta$ is bounded from above  such that $\log \delta=o(\log m)$ then BH and SD at FDR level $\alpha \propto n^{-1/2}$ are ABOS.
\end{cor}

\begin{cor}\label{cor2}
Suppose Assumptions (A) and (C) hold. If $p=m^{-\beta}$ with $\beta\in(0,1]$, $n\propto \log m$ and $\delta$ converges to zero such that $\log \delta=o(\log m)$ then BH and SD at a fixed FDR level $\alpha\in (0,1)$ are ABOS.
\end{cor}

\begin{rem}
{\rm Corollary (\ref{cor1}) states that under some mild restrictions on $\delta$ BH and SD at the FDR level $\alpha \propto n^{-1/2}$ are ABOS. Corollary (\ref{cor2}) says that in case when $n\propto \log m$ then under the additional requirement that $\delta\rightarrow 0$, BH and SD at the fixed FDR level $\alpha \in (0,1)$ are also ABOS. This result substantially extends the results of \citep{Bog2009} to the case where the prior on $\mu_i$ is fixed and not necessarily normal, while the sample size $n$ slowly increases to infinity. This additionally justifies the use of the fixed FDR level for BH in many applications, like e.g. in bioinformatics, where $n$ is much smaller than $m$. As discussed in \citep{Bog2009} the condition $\delta \rightarrow 0$ is quite reasonable in this context, since the cost of missing a true positive is usually large if $p$ is very small. }
\end{rem}


\section{ABOS in the context of multiple regression} \label{Sec:Regression}


\noindent
It is well known \citep{Bog2009, FG}  that there is a strong connection between model selection for multiple regression and multiple testing rules.  Under the simplified assumption of an orthogonal design matrix and known variance of the error term the two problems actually become identical. Consider a multiple linear regression model
$$
Y_{n\times 1}=X_{n\times (m+1)}\beta_{(m+1)\times
1}+\epsilon_{n\times 1}\;\;,
$$
where the first column in the design matrix consists of ones and
$\epsilon \sim N \left(0, \sigma^2
I_{n \times n}\right)$.
Let us additionally assume that
\begin{equation}\label{as1}
X'X=nI_{(m+1)\times (m+1)}\;\;,
\end{equation}
and that the regression coefficients
$\beta_{1},\ldots,\beta_{m}$ can be modelled as independent
random variables from the following mixture distribution
\begin{equation}\label{as2}
(1-p)d_{0}+p \nu \;.
\end{equation}
Under the assumptions (\ref{as1}) and (\ref{as2}) least squares
estimates $\hat \beta_i$, $1\leq i \leq m$, are independent random
variables from the mixture distribution
\begin{equation}\label{mixlin}
(1-p)N\left(0,\frac{\sigma^2}{n}\right)+p \left(\nu * N\left(0,\frac{\sigma^2}{n}\right)\right)\;\;.
\end{equation}
This is identical with (\ref{barX}) and thus the problem of detecting true regressors is equivalent to the
multiple testing problem. Therefore, in case
when each false positive (falsely detected regressor) induces the
cost $\delta_0$ and each false negative induces the cost $\delta_A$,
thresholds of the Bayes rule and the optimal Bayes risk are obtained just like in Lemma \ref{LemF2} and in Theorem \ref{TH1_general}.

As mentioned in the introduction we will focus here on the case of Assumption (B), where the loss ratio has no particular influence on the asymptotic results.   In this case $\frac{2 \log f \delta}{n}= -\frac{2\log p} {n}(1+o_n)$.  We also consider only sparsity parameters $p\rightarrow 0$ satisfying  assumption (\ref{pm}). Since under orthogonal designs $m < n$ one has  $-\log p=O(\log m) = O(\log n)$, and finally  $ -\frac{2\log p (1+o_n)}{n} \rightarrow 0$. Thus,   under orthogonal designs assumptions (B) and (\ref{pm}) imply Assumption (A) with $C=0$. Therefore we will  refrain from referring to  Assumption (A) in this section.

We  will first discuss a model selection criterion which is ABOS in case of extreme sparsity
(\ref{as_spars}), as in Corollary \ref{Cor_optcv}. However, it is easy to see that for $m \leq n$
that sparsity assumption reduces to
\beq\label{as_spars0}
m p \rightarrow s \in (0,\infty], \quad  \frac{\log (m p)}{\log n} \rightarrow 0\;.
\eeq

\subsection{ABOS of mBIC  when $\sigma$ is known} \label{Sec:mBIC}

It was shown in \citep{BGD} in the context of QTL mapping that for large $m$ classical model selection criteria like AIC or BIC tend to select too large models. Based on Bayesian ideas a modified version of BIC (mBIC) was proposed to take into account the number of available regressors.
When $\sigma$ is known the mBIC criterion suggests choosing the model $M$
for which
\beq\label{mBIC_pen}
\frac{RSS_M}{\sigma^2}+k(\log n + 2\log m+d)\;\;
\eeq
 obtains a minimum, where $RSS_{M}$ refers to the residual sum of squares for model $M$, $k = k(M)$ is the number of regressors in the model and $d$ is a certain constant. A comprehensive introduction into the ideas leading to mBIC is given in \cite{BGZ}. 

\begin{rem} \label{Choice_of_d}
{\rm
It follows from the derivation of mBIC that from a Bayesian perspective  $\exp(-d/2)$ is the a priori expected number of regressors. If there is no prior knowledge on the model size the recommended standard choice is $d=-2\log(4)$, which guarantees  control of FWER at level 0.1 for $n \geq 200$ and $m \geq 10$. For further details see \cite{BGZ}.
} 
\end{rem}

 Apart from ABOS we  want to show consistency of mBIC.

\vspace{2mm}
{\bf \noindent Definition.} A model selection rule is said to be consistent if the probability of selecting the true model converges to 1 as $m\rightarrow \infty$.

\begin{thr} \label{TH3}
  Consider the orthogonal regression model specified by the conditions (\ref{as1}) and (\ref{as2}) and let assumptions (B) and (C) (with C=0) hold. Under (\ref{as_spars0}) mBIC is  ABOS, while under the considerably weaker assumption \beq \label{as3}
m p \rightarrow s \in (0,\infty], \quad  mp \sqrt{\frac{\log n}{n}} \rightarrow 0
\eeq
mBIC is consistent.
\end{thr}
\begin{pro}
It is easy to check that under assumption (\ref{as1})  mBIC  suggests choosing
those regressors for which
$$ \frac{n\hat\beta_j^2} {\sigma^2}>\log n+2\log m+d\;\;.$$
From Corollary \ref{Cor_optcv} one immediately concludes that under the sparsity assumption (\ref{as_spars0})   this selection rule is ABOS.

To prove consistency of mBIC let the random variable $M_j$ be Bernoulli distributed  where a misclassification of predictor $X_j$ denotes a success. If $t_1$ and $t_2$ denote the probability of type I and type II error of mBIC, respectively, then for sufficiently large $n$
$$
P(M_j = 1) = (1-p)t_1+p t_2 \leq K p  \sqrt{\frac{\log n}{n} }
 $$
for some constant $K$, where the last inequality is shown in Appendix \ref{Sec:ApCons}.
Using Markov's inequality the probability of picking the wrong model (which is the probability of at least one wrong misclassification) can thus be bounded like
\beq \label{Markov}
 P(\sum\limits_{j=1}^m M_j  \geq 1) \leq  E\left( \sum\limits_{j=1}^m M_j \right) \leq  K mp \sqrt{\frac{\log n}{n}}\; ,
\eeq
which according to  (\ref{as3}) converges to 0.
\end{pro}

\begin{rem}
{\rm
Theorem \ref{TH3}   addresses the situation of sparsity, where the expected number of true signals remains constant or slowly increases with  $m$. The assumption $mp \rightarrow s < \infty$  was used when deriving the  mBIC penalty in \cite{BGD}.
 Theorem \ref{TH3} actually tells us that mBIC remains optimal when the number of true signals is mildly growing, for example $m p = \log m$ is still conceivable.
This scenario might be more realistic in many applications, where one would hope that by increasing the number of markers one could actually be able to detect more true signals.  However, the situation described is still very sparse, which is one motivation  to introduce in Section \ref{Sec:mBIC2}  criteria which are slightly less restrictive.

}\end{rem}

\begin{rem}
{\rm Note that  under the assumption $mp \rightarrow s < \infty$ the expected value of the number of false positives $EP$ produced by the standard BIC is
equal to
$EP=m(1-p)t_1=\frac{m}{\sqrt{n \log n}}(1+o_{n,m})$. Thus BIC is not consistent when $\lim_{n\rightarrow \infty}\frac{m}{\sqrt{n \log n}}>0$.
}
\end{rem}

\begin{rem}
{\rm Another interesting situation arises for distributions  $\nu$
for which there exists an open interval  including 0 such that $\nu(-l,r) = 0$
 (cf. \citep{JT}).  It can be shown that in this situation the mBIC rule is  not  optimal anymore, although its risk still converges to 0.
}\end{rem}

\subsection{Modifications of BIC controlling FDR}  \label{Sec:mBIC2}

As shown in  \citep{BGZ} there exists a close connection between mBIC penalty and the Bonferroni correction for multiple testing. In a recent paper \citep{A} Abramovich et al. have been discussing extensively penalized model selection schemes which control the false discovery rate. Their starting point is the close relationship between step-up and step-down FDR controlling procedures at level $\alpha$ and the following penalizing scheme:  For models of size $k$  define the selection criterion
\beq\label{FDR_proc}
\frac{RSS_M}{\sigma^2} + \sum\limits_{l=1}^k q_N^2(\alpha l/2 m) \;,
\eeq
where $q_N(\eta)$ is the $(1-\eta)$ - quantile of the standard normal distribution. It can be shown quite easily that the size of models selected by this procedure is larger or equal $k_G$ and smaller or equal $k_F$ (see \citep{A}). The procedure is therefore  nested between BH and SD, and from Theorem \ref{TH:BH} it immediately follows that it is also ABOS.

We will adopt approximations of the FDR penalization (\ref{FDR_proc}) to amend BIC. A simple argument involving the normal tail approximation shows that
$$
q_N^2(\alpha l/2 m) \sim 2\log(m/l) -  \log[ 2\log(m /\alpha l)] + \log(2/\pi) - 2\log \alpha\;.
$$
In view of Corollary \ref{Rem:sqrt_n} we are mainly interested in the case where $\alpha \propto n^{-1/2}$ which leads to the criterion
\beq \label{mBIC1_pen}
\mbox{mBIC1:} \quad \frac{RSS_M}{\sigma^2}+k(\log (n m^2) + d_1) - 2\log(k!) - \sum_{i=1}^k \log \log(nm^2/i^2) \;\;.
\eeq
Here the constant $d_1$ can be chosen appropriately  to control FDR at a given level.  
Neglecting the last term of the mBIC1 penalty, which is of a lower order than the two preceding terms,   leads to the following simplified form of (\ref{mBIC1_pen}),
\beq \label{mBIC2_pen}
\mbox{mBIC2:} \quad \frac{RSS_M}{\sigma^2}+k(\log (n m^2) + d_2) - 2\log(k!)  \;.
\eeq
This might be thought of as a first order approximation of the FDR penalization, whereas mBIC1 is a second order approximation.
Interestingly, the penalty in mBIC2 is very similar to a modification of RIC introduced in \citep{GF}, with additional penalty term
$$
 2 \sum_{i=1}^k \log(m/i) = k \log(m^2) -2 \log(k!) \;\;,
$$
which was motivated by an empirical Bayes approach. 

 Abramovich et al. consider  in \citep{A} the approximation $\sum_{l=1}^{k} q_N^2(\alpha l/2 m) \sim k q_N^2(\alpha k/2 m)$, which can be justified by using the Sterling approximation for $k!$. The resulting first order criterion has the form
\beq \label{mBIC3_pen}
\mbox{mBIC3:} \quad \frac{RSS_M}{\sigma^2}+k(\log (n m^2) + d_3) - 2k\log(k) \;.
\eeq
Compared with (\ref{mBIC2_pen}) this means essentially that  $\log(k!)$ is substituted by $\log(k^k)$.

\begin{rem}\label{Choice_of_d2}
{\rm
In the simulation study of Section \ref{Simulation}  the constant of mBIC1 is chosen as $d_1=0$, which guarantees control of FDR at a level below $0.06$ for sample size $n$ larger than 200.
For mBIC2 the constant $d_2 = -2\log(4)$ is used, which coincides with the recommended standard choice of $d$ for mBIC. For moderate $m$ and $n$  as in the simulation study  mBIC1 with $d_1 = 0$ and mBIC2 with $d_2 = -2\log(4)$ have rather similar penalties for small $k$.  In case of mBIC3 the Sterling approximation leads to $d_3 =  d_2 + 2$. 
}
\end{rem}

\begin{thr} \label{TH:modBIC}
Consider the orthogonal regression model specified by the conditions (\ref{as1}) and (\ref{as2}). Let Assumptions  (B) and (C) as well as (\ref{pm}) be true. For the denser case (\ref{dense}) the additional condition (\ref{ularge}) is assumed to hold. Then the
rules mBIC1, mBIC2 and mBIC3 are ABOS.   The rules are consistent under the additional assumption (\ref{as3}).
\end{thr}

The proof is given in  Appendix \ref{Sec:App5}.
\ \\

\begin{rem}
{\rm
The FDR controlling selection rules mBIC1 - mBIC3 are ABOS under much less restrictions on the sparsity levels than mBIC.   However, conditions for consistency are exactly the same. Actually given the other assumptions of Theorem \ref{TH:modBIC} it follows that (\ref{as3})  is also necessary for the Bayes rule to be consistent.
 }\end{rem}


\subsection{ABOS of mBIC when $\sigma$ is unknown}  \label{Sec:mBIC_su}


We have seen that for known $\s$ and under the simplifying assumption of an orthogonal design matrix, the problem of model selection using mBIC in multiple regression is equivalent to multiple testing, in the sense that a regressor is included in the model chosen by mBIC if and only if the corresponding square of the sample regression coefficient is larger than a fixed threshold. In case of unknown $\s$ the situation gets much more complicated and no such direct connection with multiple testing can be established. We are only interested in the comparison of models which include the intercept. In this case the Bayesian Information Criterion chooses that model which minimizes $BIC = n\log RSS_{M}+ k \log n$.  The corresponding version of mBIC becomes
\begin{equation}\label{mBIC}
mBIC = n\log RSS_{M}+k(\log n+2 \log m+d)\;\;.
\end{equation}
 Our main goal is to show that also in  case of unknown $\s$ mBIC is asymptotically optimal.

Some problem occurs when (\ref{mBIC}) is used as a selection criterion for very large models. To be able to estimate the parameters of a model $M$ we need the restriction that $k \leq n-2$. But if $k$ is getting close to $n$ then overfitting will lead to extremely small $\log RSS_{M}$,  and the global minimum of (\ref{mBIC}) is likely to be attained by models of maximum size $k = n-2$ (if that many regressors are available).  It can be ruled out that such models are correct under the assumption of sparsity.  To cope with this pathology  we will restrict $L$, the maximal number of regressors to be allowed in addition to the common intercept term, by
\begin{equation}\label{as4}
L=o\left(\frac{n}{(\log n+2\log m)^2 \log m}\right)  \mbox{ as }n \rightarrow \infty\;\;.
\end{equation}
On the other hand to bound the type II error it is necessary to search among sufficiently large models, and we require the lower bound
\beq \label{asL}
L \geq m p (\log n)^{1+\eta} \; \mbox{ for some } \eta > 0 \mbox{ and all sufficiently large }n \; .
\eeq

\begin{thr} \label{TH:mBIC_unknown_sigma}
Suppose as in Theorem \ref{TH3} that Assumptions  (B) and (C), (\ref{as1}), (\ref{as2})  hold.  Furthermore assume that  (\ref{as4}) and (\ref{asL})  are true.
Then  the mBIC criterion (\ref{mBIC}) is ABOS under (\ref{as_spars0}), and consistent under (\ref{as3}).
\end{thr}

The somewhat lengthy proof of this theorem is provided in Appendix \ref{ApSec:Th3.3}.

\begin{rem}
{\rm Note that except for the conditions (\ref{as4}) and (\ref{asL}) on the potential model size $L$ the assumptions for ABOS of mBIC in case of unknown $\sigma$ are exactly the same as in Theorem \ref{TH3} for known $\sigma$. We conjecture that similarly the results of Theorem \ref{TH:modBIC} concerning ABOS of the FDR controlling modifications of BIC should also hold in case of unknown $\sigma$. However, the techniques used for the proof of Theorem \ref{TH:mBIC_unknown_sigma} cannot easily be extended to mBIC1 - mBIC3.  We will come back to this point in the simulation study in the next section.
}\end{rem}

\section{Simulation results} \label{Simulation}

We employ computer simulations to investigate the performance of the proposed model selection rules for  multiple regression. For the sake of simple notation in this section $m$ denotes the number of regressors plus intercept.
We use orthogonal designs with $n = m$, where the design matrices $X_{m\times m}$ are chosen as  Hadamard matrices, whose elements are equal to 1 or -1.  For each of the simulation runs the number of nonzero regression coefficients $k^*$ was simulated from a binomial distribution $B\left(m,p\right)$.
Then the values of nonzero coefficients $\beta_1,\ldots,\beta_{k^*}$ were simulated  from a normal distribution $N(0,\tau^2)$, with  $\tau^2=0.9$. Finally the values of the response variable were simulated according to the multiple regression model
$$
Y_i=\sum_{j=1}^{k^*} \beta_j X_{ij}+\epsilon_j\;\;,
$$
where $\epsilon_j \sim N(0,1)$.
The specific value of the variance of regression coefficients $\tau^2=0.9$ is selected in such a way that  the power of the Bayes oracle for $m=256$ is in the range between 50\% and 60\%. This choice allows to assess differences in performance of the considered model selection rules.

In the first part of the simulation study we consider sparsity parameters  $p\in\{0.001, 0.005, 0.01, 0.02, 0.05, 0.1, 0.2\}$ and simulate for $m = 256$ as well as $m = 1024$. In the second part we will look at
a wider range of sample sizes $n = m \in \{128, 256, 512, 1024, 2048, 4096\}$, while the sparsity parameters are computed according to $p \propto m^{-\beta}$ for four different levels $\beta \in \{1, 1/2, 1/4, 1/8\}$.

 We compared the following model selection criteria:
\begin{enumerate}
\item The Bayes Oracle (\ref{JT1}) with $\delta_0=\delta_A$. This oracle is aimed at minimizing the expected number of wrongly classified regressors and in our setting includes those explanatory variables for which
\begin{equation}\label{break}
n\hat \beta_{i}^2>\frac{n\tau^2+1}{n\tau^2}\left(\log \left(n\tau^2+1\right)+2 \log \left( \frac{1-p}{p}\right)\right) \;\; .
\end{equation}
\item Modified versions of Bayesian information criterion:
\begin{enumerate}
\item mBIC:  (\ref{mBIC_pen})  with $d = -2\log 4$
\item mBIC1: (\ref{mBIC1_pen}) with $d_1 = 0$
\item mBIC2: (\ref{mBIC2_pen}) with $d_2 = -2\log 4$
\item mBIC3: (\ref{mBIC3_pen}) with $d_3 = -2\log 4 + 2$
\end{enumerate}
 The values of the constants are chosen according to Remark \ref{Choice_of_d} and Remark \ref{Choice_of_d2}.
\item Step up and step down FDR controlling procedures, (\ref{step-up}) and (\ref{step-down}) at FDR levels $\alpha=0.05$. These procedures test individually each of the regression coefficients based on simple regression models.
\end{enumerate}

Modified versions of BIC and FDR controlling procedures are investigated under two scenarios: when $\sigma$ is known and when it is unknown. In case when $\sigma$ is unknown modified versions of BIC are based on $n\log RSS_{M}$ instead of $\frac{RSS_M}{\sigma^2}$ (see (\ref{mBIC_pen}) and (\ref{mBIC})). For unknown $\sigma$ the FDR controlling procedures are based on t-tests instead of z-tests.

To identify the regression models,  which are  ``best'' with respect to our model selection criteria, we  start with ordering explanatory variables based on the results of simple regression t-tests. This procedure gives us the proper sequence of nested models,  since under the orthogonal design the estimate of a regression coefficient for a given explanatory variable does not depend on the other regressors included in the model. Then we compare  values of model selection criteria for these nested models, starting from the null model, with no explanatory variables, and ending with a model of dimension $k_{max}=0.3m$. The need for using the bound on the maximal number of components in  the considered models results from the fact that under our design the residual sum of squares for the full model is equal to 0. Therefore, in case of unknown $\sigma$, all modified versions of BIC are optimized by the full model (see the discussion before introducing assumption (\ref{as4})).
Despite of this, according to Theorem  \ref{TH:mBIC_unknown_sigma} and the results of \cite{CL} on the consistency of similar model selection rules, we expect that our model selection criteria are consistent if the true design is sparse and $k_{max}$ goes to infinity at a slower rate than $m$. The choice $k_{max}=0.3m$ corresponds to the expected upper bound of model sizes for the sparsity level $p = 0.2$.

 For all considered procedures we report several characteristics, which are calculated based on 10000  replicates. For each of these replications we compute the number of chosen variables that do not appear in the true model (false positives, FP) and the number of true regressors which were not detected (false negatives, FN).
These values are used to calculate the following statistics:
\begin{itemize}
     \item[1.] Misclassification probability:   $\quad \text{MP} = (\text{FP + FN})/(m-1)$.
   \item[2.] False discovery rate:  $\quad \text{FDR} = \frac{\text{FP}}{\text{FP} +\ k^* - \text{FN}}$, or $0$ in case of no discoveries.
    \item[3.] Power $= \frac{k^*-\text{FN}}{k^*}$ (cases for which $k^*=0$ are excluded from this analysis).
\end{itemize}
For each scenario the values of MP, FDR and Power are averaged over all 10000 simulations.

\subsection{First part of Simulation}
The results of this part of the simulation study are illustrated in Figure \ref{Fig3} and Figure \ref{Fig4} in Appendix \ref{Sec:ApFig}.
Figure \ref{Fig3} presents the graphs of our computed characteristics as functions of the sparsity parameter  $p$ in case of known $\sigma$. The two plots (a) and (b) of the first line show that, as expected, the Bayes oracle has the lowest misclassification probability MP. However, the differences  in MP between the Bayes oracle and FDR controlling procedures, as well as  mBIC1-mBIC3, are hardly observable.  For $p < 0.05$ also MP of mBIC is comparable to and sometimes even better than MP of other criteria. However, for $p=0.05$ differences become observable, and for $p > 0.05$ MP of mBIC substantially exceeds the values reported for other methods. Qualitatively there is no different behavior in the plots for $m=256$ and $m=1024$, though it is clear that MP gets smaller for larger sample size. These observations agree well with our results on the asymptotic optimality of mBIC in case of extreme sparsity, and of the FDR controlling procedures and mBIC1-mBIC3 in a wider range of sparsity levels. Apparently our asymptotic analysis  describes the situation already quite well for $m=256$.

Plots (c) and (d) of Figure \ref{Fig3} show the FDR of different procedures.
FDR of the Bayes oracle increases from 0 for $p=0$ to 0.08 for $p=0.2$ in case of $m=256$, and to 0.03 for  $m=1024$. As expected, FDR of both step up and step down multiple testing procedures slowly decreases from approximately 0.05 for $p=0$ to 0.04 for $p=0.2$ independently of the sample size.
The same pattern is observed for the first modified version of BIC aimed at controlling FDR, mBIC1. For $m=256$ its FDR behaves almost identical to BH, whereas for $m=1024$ FDR starts at 0.03 and decreases to 0.02. FDR of mBIC2 and mBIC3 behave quite differently in case of extreme sparsity. Due to the choice of constants $d_1$ and $d_2$, FDR of mBIC2 is close to FDR of mBIC1 for small $p$.  In contrast mBIC3 has extremely small FDR for $p$ close to 0, which is due to the fact that for small $k$ Sterling's approximation is not valid. For larger $p$ (resulting in the choice of larger models) mBIC2 and mBIC3  behave more and more similar, and their FDR stabilizes  at a level of approximately 0.05 for $m=256$ and at 0.025 for $m=1024$, being thus slightly larger than FDR of mBIC1.
Finally FDR of the modified version of BIC aimed at controlling the Family Wise Error Rate, mBIC, quickly decreases; for $m=256$ from approximately 0.043 for $p=0$ to 0.0015 for $p=0.2$, and for $m=1024$ from approximately 0.015 down to 0.001.

The pattern of the graphs (e) and (f) for power corresponds to the behavior of FDR. At $p = 0.001$ clearly the Bayes oracle has smallest power. In case of $m=256$ for $p\geq 0.01$ the power of the Bayes oracle exceeds the power of other model selection criteria, whereas for $m=1024$ BH and SD have largest power. However, the differences of power between all criteria apart from mBIC are very small and for $p>0.001$ do not exceed 4\%. Also, it is interesting to observe that the power of these criteria slowly increases with $p$. mBIC performs substantially different than the other methods. Its power is significantly smaller  and remains constant as a function of $p$.
Graphs (e) and (f) illustrate also that as expected power increases with sample size.

In Figure \ref{Fig4} the results for unknown $\sigma$ are reported. The most obvious difference between the case of known and unknown $\sigma$  is observed for the multiple testing procedures based on simple regression tests. FDR of these procedures is close to the nominal level of 0.05 only when $p$ is very close to 0. For larger values of $p$ other important regressors inflate the residual error in simple regression tests, which leads to a very low power, low FDR and large misclassification rate. As a consequence, when $\sigma$ is unknown simple regression tests perform substantially worse than other methods based on model selection strategies. This finding  has been discussed extensively in \cite{FRTB} in the context of genome wide association studies.

 Concerning modified versions of BIC  the performance of mBIC1-mBIC3 is only slightly affected by the fact that $\sigma$ is unknown when $p\leq 0.1$. However, for $p=0.2$ and $m=256$ one observes a significant increase of FDR and MP when compared to the known $\sigma$ case. In particular mBIC2 and mBIC3 have a sudden increase of FDR which results also in a significantly larger MP than that of the Bayes rule. mBIC1 suffers from the same problem, though to a lesser extent. Thus for larger $p$ the second order approximation in mBIC1 proves beneficial.

 While mBIC2 and mBIC3 are getting for larger $p$ too liberal, mBIC has the opposite tendency. Especially for $m=256$  the fact that $\sigma$ is unknown leads to a substantial decrease of power and FDR for large values of $p$.  For $m=1024$ the relative performance of mBIC substantially improves and is only slightly worse than for known $\sigma$. However, both in terms of power and MP mBIC is still performing much worse than mBIC1 - mBIC3.

\subsection{Second part of Simulation}
Here we want to assess numerically  the asymptotic behavior which was analyzed theoretically in Section \ref{Sec:Regression}. To this end we will perform similar computations as above, but consider the wider range of sample sizes  $n = m \in \{128, 256, 512, 1024, 2048, 4096\}$. The sparsity parameter is computed as $p = c_\beta m^{-\beta}$, where we analyze the extremely sparse case $\beta = 1$ as well as $\beta \in \{1/2, 1/4, 1/8\}$. For each scenario the  factor $c_\beta$ is chosen such that for $m = 128$ we always have $p = 0.125$. The misclassification probability for the four different scenarios and for the various methods are provided in Figure \ref{Fig:asymp}. We no longer consider SD, as it has been seen before to behave more or less identical with BH. We also present here only the case of unknown $\sigma$, which is of particular interest in view of the unproven conjecture that mBIC1 - mBIC3 will be ABOS for a wider range of sparsity levels than mBIC.

\begin{figure}[h]
\caption{Asymptotic behavior of the misclassification rate MP at sparsity $p  \propto  m^{-\beta}$ for different values of $\beta$.  } \label{Fig:asymp}
\vspace{3mm}
$\begin{array}{c@{\hspace{0cm}}c}
\multicolumn{1}{l}{\mbox{(a) } p \propto m^{-1} }

&    \multicolumn{1}{l}{\mbox{ (b) } p \propto m^{-1/2}}   \\
\epsfig{file=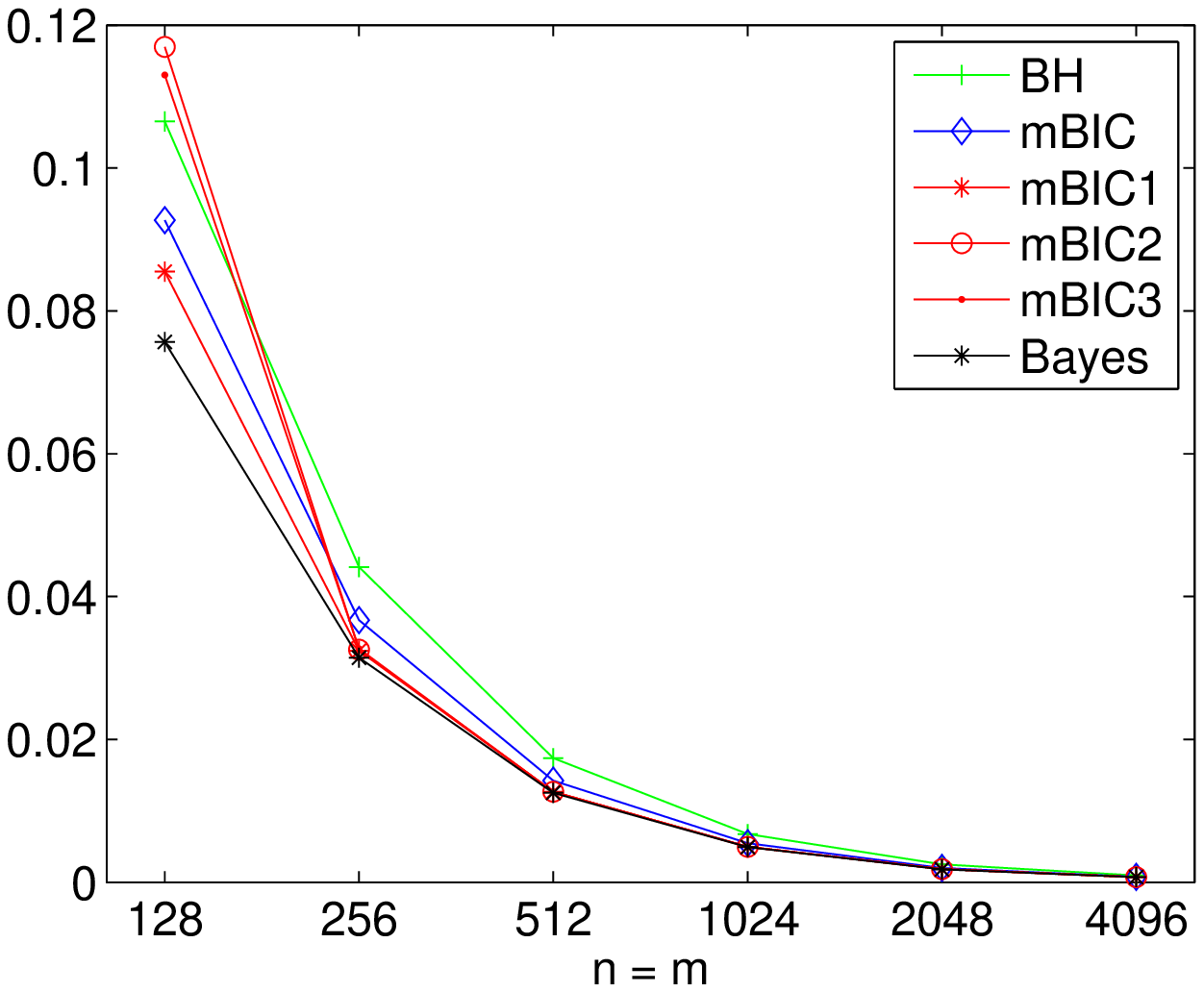,width=7.0cm} &
       \epsfig{file=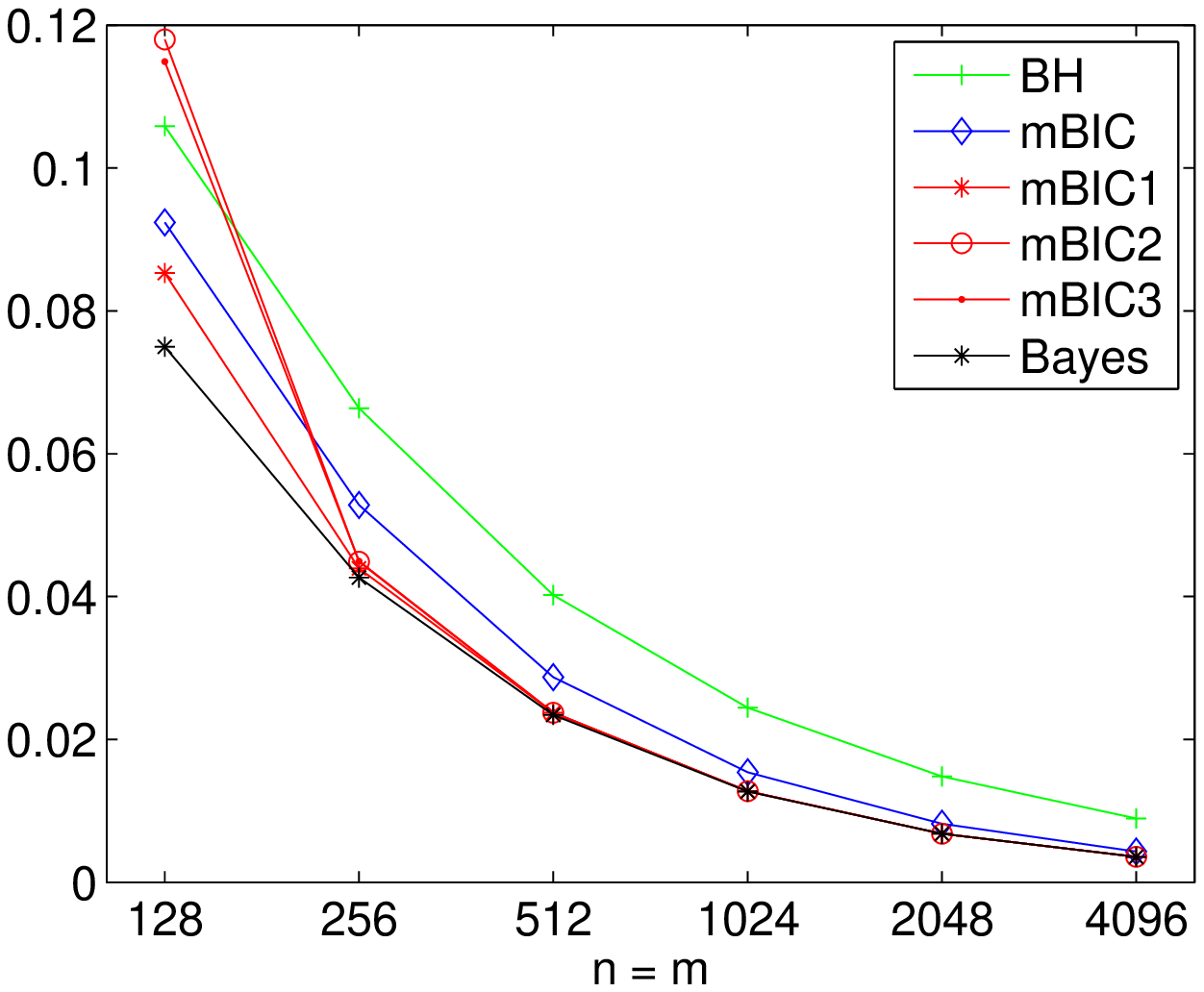,width=7.0cm}\\
       \multicolumn{1}{l}{\mbox{(c) } p \propto m^{-1/4} }
&    \multicolumn{1}{l}{\mbox{ (d) } p \propto m^{-1/8}  }  \\
\epsfig{file=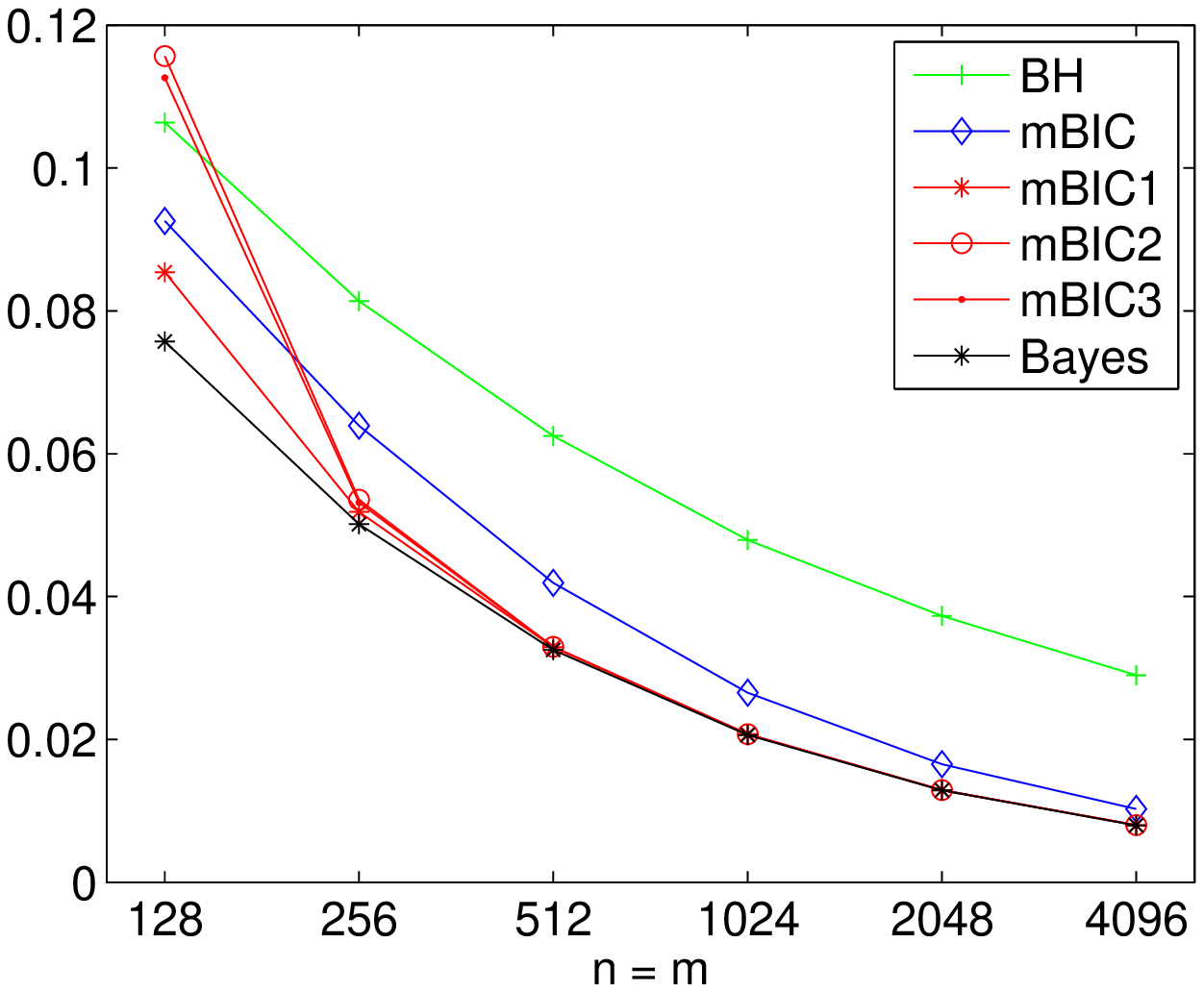,width=7.0cm} &
       \epsfig{file=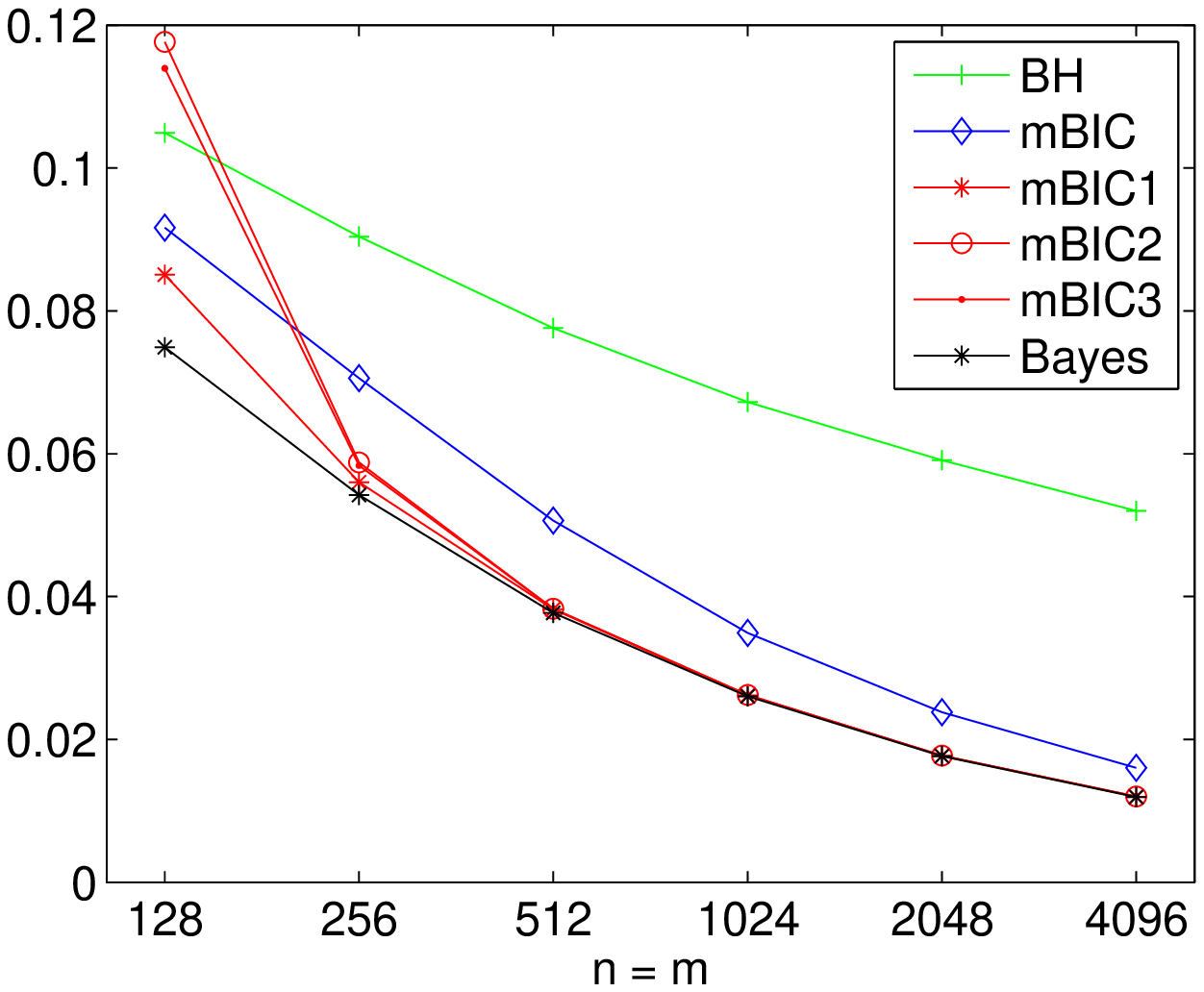,width=7.0cm}
    \end{array}$
\end{figure}

For $m= 128$ (and $p = 0.125$) mBIC1 has lower misclassification rate than all other criteria. mBIC2 and mBIC3 have relatively large misclassification rate, and are performing worse than mBIC. We had seen this behavior before already for $m = 256$ and $p = 0.2$. If there are relatively many true signals and $m$ is small then mBIC2 and mBIC3 tend to be slightly too liberal.

For $\beta = 1$ the misclassification rate of all procedures converges towards that of the optimal Bayes rule. In particular it is confirmed that mBIC is  ABOS  in case of extreme sparsity, although mBIC1 - mBIC3 perform even better. In case of extreme sparsity it seems that even BH behaves relatively well. For smaller $\beta$ a multiple testing approach is not suitable in case of unknown $\sigma$ as we discussed already above.

 The smaller $\beta$, the poorer becomes the performance of mBIC. Although it seems that its misclassification rate still converges towards that of the Bayes rule, this is only true in absolute terms. Already for $\beta = 1/2$ the ratio  of the misclassification rates between BH and the Bayes rule remains more or less constant at 1.2. For  $\beta = 1/8$ this ratio is actually growing, and mBIC is certainly not optimal. On the other hand MP of mBIC1 - mBIC3 rapidly converges towards MP of the Bayes rule in all four scenarios, which supports our conjecture that an analogue of Theorem \ref{TH:modBIC} should also hold in case of unknown $\sigma$.
Finally Figure \ref{Fig:asymp} suggests that regardless of the sparsity level $\beta$ all modifications of BIC are consistent selection rules in the asymptotic framework of Assumption B.

\section{Discussion} \label{Sec:Disc}

The first part of this paper generalizes optimality results of \cite{Bog2009} for multiple testing procedures. Instead of  scaled normal distributions we consider models of a larger class of distributions under the alternative. Only  point null hypotheses are considered and the measure under the alternative is kept fixed. The asymptotics is thus not driven by a scaling parameter which determines the effect size, but rather by the sample size $n$ which is assumed to become large.
 In that context we study two situations: The ``verge of detectability'' case as in \cite{Bog2009}, where the power of the Bayes oracle is positive but less than 1. In this article the notion of ``the verge of detectability'' is extended to the practically important case where the distribution of the effect size is fixed and the sample size $n$ slowly increases with the number of tests $m$. When sparsity is of the form $p \propto m^{-\beta}$ and the ratio of losses $\delta$ is bounded from above, then the ``verge of detectability'' is obtained when $n$ grows proportionally to $\log m$.  The second analyzed case is concerned with asymptotic power equal to 1, which is naturally associated with the situation where $n$  grows faster  than  $\log m$.  

In both cases all optimality results of \cite{Bog2009} could be proved for the considered general class of distributions, where in the second case the analysis is slightly more involved and some additional mild restrictions on the asymptotic behavior of the loss ratio $\delta$ are necessary. In particular it was shown that the Bonferroni selection rule is ABOS in case of extreme sparsity, whereas the Benjamini--Hochberg rule is ABOS under a much wider range of sparsity levels. Thus results  of \cite{Bog2009} have been extended to many practically important cases, where the distribution of the true effects is not symmetric. A new result is  that  the step down version of the FDR - controlling procedure is ABOS under almost the same conditions as BH.

Optimality results were then transferred into the context of linear regression. The simplest situation is concerned with orthogonal regressors  and known error variance $\sigma^2$, where optimality results from multiple testing can be directly applied. We analyzed the performance of mBIC, a modification of BIC which was previously introduced for model selection in high dimensional data \cite{BGD}, and which is known to control the family wise error rate under the given conditions \cite{BGZ}. It turns out that mBIC is ABOS in case of extreme sparsity, namely under the same conditions as the Bonferroni selection rule for multiple testing. Additionally three different FDR-controlling modifications of BIC were introduced. Optimality results for these selection rules, mBIC1 - mBIC3, entirely correspond to results for the step up and step down FDR controlling procedures  in multiple testing. Thus mBIC1 - mBIC3 are ABOS under a much wider range of sparsity levels than mBIC. All modified versions of BIC (including mBIC) are consistent under the same  assumption on sparsity levels which guarantee consistency of the Bayes oracle.

Next we showed ABOS of mBIC under extreme sparsity in case of unknown $\sigma$, a situation which is technically much more demanding than the previous case of known $\s$. We conjecture that in analogy to the known $\s$ case,  mBIC1 - mBIC3  should be ABOS  when removing the extreme sparsity restriction. While we were not able to give a formal proof, simulation results strongly support this conjecture.
Furthermore mBIC in case of unknown  $\sigma$ is consistent under the same conditions on sparsity levels under which the Bayes oracle is consistent. The same is expected to hold for mBIC1 - mBIC3.
 Apart from our simulation study, consistency of the modified versions of BIC   for unknown $\sigma$ can also be conjectured based on recent consistency results for the extended version of Bayesian Information Criterion, EBIC, reported in \cite{Chen1} and \cite{CL}. As discussed in \cite{ZB}, if the dimension of the maximal allowable model  $k_{max}$ satisfies $k_{max}/m \rightarrow \infty$ then mBIC2 is asymptotically equivalent to the standard version of EBIC, based on a uniform prior on the model dimension. It follows that mBIC2 can be interpreted as an approximation of the Bayesian rule, in which the prior  on the true number of regressors is uniform over the set $\{0,\ldots,k_{max}\}$, with $k_{max}=o(m)$.

The results presented in this article are important to understand optimality of model selection criteria under sparsity. However, they are somewhat preliminary as they are only considering the case of orthogonal regressors. In most applications where sparsity is an issue one is also dealing with $m > n$, that is the number of regressors exceeds the sample size. Our current analysis is explicitly not applicable to this situation. However,  we believe that the majority  of  results can be extended to the case $m>n$ if the design matrix satisfies  certain conditions for identifiability of small models, which are discussed for example in \cite{DS}, \cite{BR} or \cite{CL}. These expectations are confirmed by the successful application of mBIC2 to analyze  genome wide association study data, as reported in \cite{FRTB}. Theoretical analysis of  asymptotic optimality properties of modifications of BIC under non-orthogonal designs is the topic of further research.

\section*{Acknowledgment}
We want to thank Professor Jayanta K. Ghosh for many discussions and guidance.

This work is partially funded by by the WWTF grant MA09-007a for F. Frommlet and by grant 1 P03A 01430 of the Polish Ministry of Science and Higher Education for  M. Bogdan.

\section{Appendix}\label{Sec:App}

\subsection{Proof of Lemma \ref{LemF1}}   \label{Sec:App0}

The proof of Lemma \ref{LemF1} relies on the following technical result.

\begin{lem}\label{ApLem_1}
Let $a_n \rightarrow a$ be any convergent sequence. Define\\[2mm] $h_n(\mu) := \exp\left(a_n \frac{\mu}{\sigma^2} -
\frac{\mu^2}{2 \sigma^2}\right)$ and $h(\mu) := \exp\left(a \frac{\mu}{\sigma^2} -
\frac{\mu^2}{2 \sigma^2}\right)$. Then
\begin{equation} \label{double_limit}
\lim_{n\rightarrow \infty}  \left\|h_n \right\|_{L^n(\nu)} = \left\|h \right\|_{L^\infty(\nu)} \; .
\end{equation}
\end{lem}
\begin{pro}
First note that for all $n$ it holds that $h_n \in L^\infty(\nu)$, and therefore  also $h_n \in L^m(\nu), \forall m >0 $. It is easy to check that $\lim_n \left\|h_n -h\right\|_{L^\infty(\nu)} = 0$. Thus for any $\epsilon >0$ and sufficiently large $n$ we have $\left\|h_n -h\right\|_{L^n( \nu)}\leq \left\|h_n -h\right\|_{L^\infty( \nu)}<\epsilon$. Now   (\ref{double_limit}) easily follows by the triangle inequality and the fact that\\ $\lim_{n\rightarrow \infty} \left\|h\right\|_{L^n( \nu)}= \left\|h\right\|_{L^\infty( \nu)}$.
\end{pro}

\noindent Now we are ready to prove Lemma \ref{LemF1}.

\begin{pro}

Let $h_n(\mu) = \exp\left(a_n \frac{\mu}{\sigma^2} -
\frac{\mu^2}{2 \sigma^2}\right)$. Then $(\delta f)^{1/n}  = \left\|h_n \right\|_{L^n(\nu)}$ and due to Assumption (A)   $\lim_n (\delta f)^{1/n} = e^{C/2}$. Note that $a_n$ has to be bounded, otherwise the sequence  $\left\|h_n \right\|_{L^n(\nu)}$  could not be bounded. Let $a$ be an accumulation point of $a_n$.
By Lemma (\ref{ApLem_1})  for any subsequence $a_j \rightarrow a$ it holds
\begin{equation} \label{double_limit_1}
\lim_j  \left\|h_j \right\|_{L^j( \nu)} = \left\|\exp\left(a \frac{\mu}{\sigma^2} -
\frac{\mu^2}{2 \sigma^2}\right)\right\|_{L^\infty( \nu)}  \; .
\end{equation}
If $a \in S$ then   $\left\|h \right\|_{L^\infty( \nu)} = \exp\left(\frac{a^2}{2 \sigma^2}\right)$
and  taking logarithms yields $a = -\sqrt{C} \sigma$.  Thus the only potential accumulation point of $a_n$ within $S$ is $-T$.
To complete the proof of Lemma \ref{LemF1} we will show that  $a\notin S$ leads to a contradiction with Assumption (C).

If  $a\notin S$ then $ a \in (l_a,   r_a)$ where $l_a<r_a$ are the boundaries of $S$, closest to $a$.
It is immediately clear that either
\begin{equation}\label{one}
\left\|h \right\|_{L^\infty( \nu)}= h(l_a)
\end{equation}
or
\begin{equation}\label{two}
\left\|h \right\|_{L^\infty( \nu)}= h(r_a)\;\;.
\end{equation}
The maximum is taken on the right boundary (\ref{two})  when $a\in (\frac{l_a+ r_a}{2},r_a)$ and  for $r_a\neq 0$ we obtain that
$a= \frac{1}{2}\left(\frac{T^2}{r_a}+r_a\right)$. Now, since $a\leq 0$  these conditions imply
$$r_a<0\;\;,\;\;T^2 > r_a^2\;\;\mbox{and}\;\;  T^2 < l_a   r_a\;\; ,$$
and we conclude that $-T\in (l_a,  r_a)$. But according to Assumption (C) we have $-T \in S$, which contradicts  $(l_a,   r_a) \notin S$.

Similarly, one can  show that for any value $a\in (l_a,\frac{l_a+ r_a}{2})$  the case (\ref{one})  also leads to a contradiction with Assumption (C) .

Now consider the remaining  case (\ref{two}) and $r_a=0$. Then   (\ref{double_limit_1}) implies that $T=0$.  However,  due to Assumption (C) $\mu$ has a positive density in some neighborhood of 0, in contradiction with  $r_a$ lying on the boundary of the support of $\mu$.

The proof that $b_n\rightarrow T$ goes exactly along the same lines.

\end{pro}

\subsection{Proof of Lemma \ref{LemF2}} \label{Sec:App1}
\ \\

\begin{pro}
By Lemma  \ref{LemF1}  $a_n$ converges to 0.  Also, by Assumption (C) there
exists $\epsilon>0$ such that $\nu(\mu)$ has a density $\rho(\mu)$ on the interval $(-\epsilon, \epsilon)$. It is immediately clear that
\begin{equation}\label{eq1}
\int\limits_{(\e,\infty)} h^n_n(\mu)\ d\nu(\mu) \leq
 e^{-n \frac{\e^2}{2\sigma^2}} \nu(\e,\infty)\ .
\end{equation}
Also, there exists  $n_0$ such that for every  $\mu < -\e$ and $n>n_0$ it holds  $a_n \mu < \mu^2/4$
(because $a_n \rightarrow 0$). Thus for $n>n_0$
\begin{equation}\label{eq2}
\int\limits_{(-\infty, -\e)} h^n_n(\mu)\ d\nu(\mu) \leq e^{-n
\frac{\e^2}{4\sigma^2}} \nu(-\infty,-\e)\;\;.
\end{equation}

Concerning the integral over the interval $(-\epsilon,\epsilon)$,
by completion of squares one derives
\begin{eqnarray} \label{f1}
& \int\limits_{-\e}^\e  h^n_n(\mu)\  \rho(\mu) d\mu  =
\rho_n \exp\left(\frac{n a_n^2}{2\sigma^2}\right) \int\limits_{-\e}^\e
\exp\left(-n \frac{ (\mu - a_n)^2}{2 \sigma^2}\right)\ d\mu & \nonumber \\
& = \rho_n e^{ \frac{n a_n^2}{2\sigma^2}}   \frac{\sqrt{2 \pi} \sigma}{\sqrt{n}}
\left[ \Phi(\sqrt{n}(\e-a_n)/\sigma) - \Phi(\sqrt{n}(- \e- a_n)/\sigma)  \right] \; ,&
\end{eqnarray}
where $\rho_n \in [\inf\limits_{\mu \in (-\e,\e)} \rho(\mu) ,\sup\limits_{\mu \in(-\e,\e)}  \rho(\mu)]$,   
and  $0 < \inf\limits_{\mu \in (-\e,\e)} \rho(\mu) \leq \sup\limits_{\mu \in(-\e,\e)}  \rho(\mu) < \infty$ 
according to Assumption (C).

 Note that $\Phi(\sqrt{n}(\e-a_n)/\sigma) \rightarrow 1$ as well as $\Phi(\sqrt{n}(- \e-a_n)/\sigma) \rightarrow 0$  (because  $a_n \rightarrow 0$).
Comparing (\ref{eq1}), (\ref{eq2}) and (\ref{f1}) we observe that the integral over $(-\epsilon, \epsilon)$ dominates the two remaining terms and from (\ref{JT1}) it follows that
$$ 1=\sqrt{\frac{2\pi}{n}}\sigma(f \delta)^{-1}\rho_n\exp\left( \frac{n a_n^2}{2\sigma^2}\right)(1+o_n)\;\;. $$

Thus we may conclude that the sequence
$$ S_n := (f \delta)^{-1} n^{-1/2} \exp\left( \frac{n a_n^2}{2\sigma^2}\right) $$
is bounded and therefore for any convergent subsequence  it holds that
 $$
a_{n} \sim -\frac{\sigma \sqrt{\log n + 2 \log (\delta f)}}{\sqrt n} \ .
$$
  To get the exact behavior we further split the domain of the integral in
$(-\e,- g_n)$, $(- g_n ,0)$ and $(0, \e)$, where $g_n$ is a positive sequence such that $a_n = o(g_n)$, or more specifically
\begin{equation} \label{g_split}
g_n \rightarrow 0  \quad \mbox{ with } \quad \frac{\log n}{n g_n^2} \rightarrow 0, \quad  \frac{\log (\delta f)}{n g_n^2} \rightarrow 0 \; .
\end{equation}
For the first interval we get a bound by evaluating the integrand at $- g_n$, for the second and third interval we repeat the computations leading to (\ref{f1})  with the corresponding boundaries, and finally obtain
$$
\delta f =\int\limits_{-g_n}^0  h^n_n(\mu) \rho(\mu) d\mu (1 + o_n)  =
  \frac{  \rho(0^-) \sqrt{2\pi} \sigma} { \sqrt{n}}  \exp\left(\frac{n a_n^2}{2\sigma^2}\right) (1 + o_n)
$$
which yields (\ref{f0a}).
The proof for $b_n$ is exactly the same.
\end{pro}

\begin{rem}
{\rm
The proof of Lemma \ref{LemF2} relies upon choosing a suitable  sequence $g_n$.  The choice of the sequence $g_n$ strongly depends on the asymptotic behavior of $\delta f$.    If for example for sufficiently large $n$, $\delta f \leq  n^\alpha$, with $\alpha>0$,   one might use $g_n = \frac{\log n}{\sqrt n}$, the choice of \citep{JT}.   Another  situation occurs if $\delta f \sim  e^{n^{1-\gamma}}$ with $0<\gamma<1$,  where $g_n = n^{-\gamma/3}$ is a suitable choice.}
\end{rem}

\subsection{Proof of Theorem \ref{TH1_general}} \label{Sec:App2}
\ \\

\begin{pro}

Notice, that the type II error of the Bayes oracle is given by $t_2 = \int \Psi_n(\mu) \ d\nu(\mu)$ with
$$
\Psi_n(\mu) = \Phi\left(\frac{\sqrt n (b_n - \mu)}{\sigma}\right)  - \Phi\left(\frac{\sqrt n (a_n - \mu)}{\sigma}\right) \; .
$$

We will now calculate the asymptotic formula for the type II error in case when $C=0$. Consider first the integral over $\mu \in (-\infty,0)$. Remember that  $a_n\rightarrow 0$, thus for $n$ sufficiently large  $\nu$ has a density $\rho(\mu)$ on $(2a_n,0)$ and it holds that
$$
\int\limits_{2 a_n}^{0} \Psi_n \ d\nu   =  \int\limits_{2 a_n}^{0} \left[\Phi\left(\frac{\sqrt n (b_n - \mu)}{\sigma}\right)  - \Phi\left(\frac{\sqrt n (a_n - \mu)}{\sigma}\right)  \right]
\ \rho(\mu) \ d\mu \; .
$$
Applying the mean value theorem and substitution yields
\bea
\int\limits_{2 a_n}^{0} \Psi_n \ d\nu  & = &  \rho_n \frac{\s}{\sqrt n} \int\limits_{\frac{\sqrt n}{\s} a_n}^{-\frac{\sqrt n}{\s} a_n} \left[\Phi\left(\frac{\sqrt n (b_n - a_n)}{\sigma} - z \right)  - \Phi\left( -z \right)  \right]
 \ dz  
\eea
for some $\rho_n \in [\inf\limits_{\mu \in (2 a_n,0)} \rho(\mu) ,\sup\limits_{\mu \in(2 a_n,0)}  \rho(\mu)]$.  Using the facts that $\int\limits_{-x}^{x}\Phi(z) dz = x$ and $-\frac{\sqrt{n}b_n}{\sigma}\rightarrow -\infty$ we further obtain
\bea
\int\limits_{2 a_n}^{0} \Psi_n  \ d\nu &=& \rho_n \frac{\s}{\sqrt n} \left[\int\limits_{\frac{\sqrt n}{\s} a_n}^{-\frac{\sqrt n}{\s} a_n}[1-\Phi(-z)] \ dz - \int\limits_{\frac{\sqrt n}{\s} a_n}^{-\frac{\sqrt n}{\s} a_n} 
 \Phi\left(z + \frac{\sqrt n}{\s} (a_n -  b_n) \right) \ dz\right]  \\
&=& - \rho(0^-) a_n (1+o_n)= \sigma \rho(0^-) \sqrt{\frac{\log v}{n} }  (1 + o_n) \; .
\eea
where the last equality holds due to (\ref{cv_general}).

It remains to show that the integral over $(-\infty,2 a_n)$ is of lower order. 
It holds that
\bea 
\int\limits_{-\infty}^{2 a_n} \Psi_n \ d\nu &\leq&
\int\limits_{-\infty}^{2 a_n} \left( 1 - \Phi(\sqrt n (a_n - \mu)/\sigma))
\right)\ d\nu  \\
&\leq &    1 - \Phi(-a_n  \sqrt n/\sigma  )   = O\left( (v \log v)^{-1/2}\right) \; .
\eea
 Assumption (A) yields $f \delta \log v \rightarrow \infty$, and hence  $(v \log v)^{-1/2} = o\left(\sqrt{\frac{\log v}{n} }\right)$.
Similar computations for the  interval $(0, \infty)$ lead to
\begin{equation} \label{t2_general}
t_2 = \sigma \sqrt{\frac{\log v}{n} } \left(\rho(0^-) +
\rho(0^+)\right)  (1 + o_n(1)) \;.
\end{equation}

In case of $0 < C < \infty$  we know from Lemma \ref{LemF1} that $a_n \rightarrow - T$ and $b_n \rightarrow T$, where $T=\sigma \sqrt{C}>0$. 
For  $\mu \in (-T,T)$,  $\Psi_n(\mu) \rightarrow 1$, while for $\mu \in (-\infty,T) \cup (T,\infty)$,  $\Psi_n(\mu)\rightarrow 0$. Then by the  dominated convergence theorem,
\beq\label{t2Cg0}
t_2=\int_{-\infty}^{\infty} \Psi_n(\mu) d\nu(\mu) = \nu(-T,T)\ (1 + o_n)\;,
\eeq
and  $\nu(-T,T)>0$, since the distribution has a positive density in neighborhoods of $-T$ and $T$.

The Bayes risk can be written as
$$
R=mp\delta_At_2(1+f\delta t_1/t_2)\;.
$$
Thus by (\ref{t2_general}) and (\ref{t2Cg0}) to complete the proof of Theorem \ref{TH1_general} it is enough to show that
\begin{equation}\label{conpr}
f\delta t_1/t_2\rightarrow 0\;\;.
\end{equation}
In case of $C=0$,  (\ref{cv_general}) and the normal tail approximation yield
$t_1 \propto (v\log v)^{-1/2}\;.$
Thus from  (\ref{t2_general}) we easily obtain
$$f\delta \frac{t_1}{t_2}\propto \frac{f\delta \sqrt{n}}{\sqrt{v}\log v}=\frac{1}{\log v}\rightarrow 0\;\;.$$ 

In case of $C>0$ we write
$t_1= t_{1a}+t_{1b}\;\;,
$
where
$t_{1a}=\Phi\left(\sqrt{n} a_n/\sigma\right)$
and
$t_{1b}=1-\Phi\left(\sqrt{n} b_n/\sigma\right)\;.$
Using the fundamental equality (\ref{JT1}) for $a_n$ yields
$$
\delta f t_{1a} \sim    \frac{\sigma}{T \sqrt n}   \ \frac{1}{\sqrt{2 \pi}}
\int\limits_{\mathcal{R}  } \exp\left( -\frac{n}{2\sigma^2}(a_n - \mu)^2 \right)\ d\nu(\mu) \;.
$$
Because $a_n \rightarrow -T$ similar considerations as in (\ref{eq1}) show that the integral vanishes rapidly for $\mu \notin (-T - \e, -T +\e)$. 
Now observe that 
$$\frac{1}{\sqrt{2 \pi}}
\int_{-T-\e}^{-T+\e} \exp\left( -\frac{n}{2\sigma^2}(a_n - \mu)^2 \right) \rho(\mu) \ d \mu \leq
 M_{\rho} \frac{1}{\sqrt{2 \pi}} \int_{\mathbb{R}} \exp\left( -\frac{n}{2\sigma^2}(a_n - \mu)^2 \right) \ d \mu \;\;,$$
where $M_{\rho}= \sup_{\mu\in(-T-\e,-T+\e)} \rho(\mu)<\infty$. 
Moreover, 
$$\frac{1}{\sqrt{2 \pi}} \int_{\mathbb{R}} \exp\left( -\frac{n}{2\sigma^2}(a_n - \mu)^2 \right) \ d \mu=
\frac{\sigma}{ \sqrt n} \;.
$$
Thus we finally obtain $\delta f t_{1a} = O\left( \frac{1}{n} \right)$.
Analogous considerations for  $t_{1b}$  finish the proof.

\end{pro}

\subsection{Proof of Theorem \ref{TH2_general}} \label{Sec:App3}
\ \\

\begin{pro}
First consider the case $C = 0$. To prove  sufficiency of (\ref{optcv1_general}) and (\ref{optcv2_general})
for ABOS of a fixed threshold rule note that computing type II error for rules of the form (\ref{optcv_general}) involves similar computations to those leading to (\ref{t2_general}), but using $\tilde a_n$, and $\tilde b_n$ instead of $a_n$ and $b_n$. Taking into account (\ref{optcv_general}) and (\ref{optcv1_general}) one thus obtains 
$$
\int\limits_{2 \tilde a_n}^{0} \Psi_n \ d\nu   =  -\rho(0^-) \tilde a_n(1 + o_n) = \rho(0^-) \sigma \sqrt{\frac{\log v}{n}} (1+o_n) \;,
$$
which is asymptotically equivalent to the first contribution of the type II error  of the Bayes Oracle. On the other hand
$$
\int\limits_{-\infty}^{2 \tilde a_n} \Psi_n \ d\nu   \leq  1 - \Phi(-\tilde a_n  \sqrt n/\sigma  ) \sim
\frac{\exp(-z_a/2)}{\sqrt{2 \pi v [\log (v) + z_a]}} = o\left(\sqrt{\frac{\log v}{n} }\right) \; ,
$$
where the last equality follows from the first part of Assumption (A) and (\ref{optcv2_general}). Similar calculations on the interval $[0,\infty]$ yield 
$$
\int\limits_{0}^{\infty} \Psi_n \ d\nu  = \rho(0^+) \sigma \sqrt{\frac{\log v}{n}} (1+o_n)  \; .
$$
Thus the type II error component of the risk $R_2=mp\delta_At_2$ satisfies $R_2=R^B(1+o_n)$.

Now, using (\ref{t2_general}) and  the tail approximation for the type I error we obtain
\beq \label{HlpR1R2}
R_1/R^B = \frac{m(1-p)\delta_0 t_1}{R^B} =C_{\sigma \rho} \frac{\exp(-z_a/2)+\exp(-z_b/2)}{\log v}(1+o_n)\;\;,
\eeq
where 
$C_{\sigma \rho}=\frac{1}{\sqrt{2 \pi}\sigma(\rho(0^-)+\rho(0^+))}$.
Thus under assumption (\ref{optcv2_general}) $R_1=o(R^B)$, which completes the proof of sufficiency  for $C=0$. 

In case of $C> 0$ due to (\ref{optcv1_general})  it holds that
$\tilde a_n \rightarrow -T$ and $\tilde b_n \rightarrow T$, and hence thresholds specified by (\ref{optcv_general}) also have  type II error of the form (\ref{t2Cg0}). For sufficiency it remains to establish (\ref{conpr}). To this end note that the type I error can be written approximately as
$$
t_1 \sim \frac{1}{\sqrt{2\pi}}\frac{\exp(-z_a/2) + \exp(-z_b/2)}{\sqrt{v \log v}} \; .
$$
 Hence
\beq \label{ntwo}
R_1/R^B = C_{\nu} \frac{\exp(-z_a/2)+\exp(-z_b/2)}{\log v}(1+o_n) \; ,  
\eeq
where $C_{\nu}=\frac{\sqrt{C}}{\nu(-T,T)}$. Thus, under assumption (\ref{optcv2_general}) again $R_1=o(R^B)$, and the proof of sufficiency  is completed.

Concerning necessity, similar arguments as in the proof of Theorem 3.2 of \citep{Bog2009} show that  (\ref{optcv1_general})  is  necessary for ABOS. In that case the computations leading to (\ref{HlpR1R2}) and (\ref{ntwo}) are still valid and imply the necessity of  (\ref{optcv2_general}). 

\end{pro}

\subsection{Lemma on the existence of the exact BFDR controlling rule} \label{Sec:ApBFDR}
\ \\

We first prove the following result

\begin{lem} \label{Lem:pointmass}
For any fixed $s \neq 0$ the function   
$$
f(c) := \frac{2 - \Phi(c-s) - \Phi(c+s)}{2(1 - \Phi(c))} \; 
$$
satisfies
\begin{itemize}
\item[a)] $f(0)=1$,
\item[b)]  $\lim\limits_{c\rightarrow \infty} f(c)=\infty$,
\item[c)] $f(c)$ is increasing in $c$  for $c \geq 0$.
\end{itemize}
\end{lem}
 \begin{pro}
 
Points a) and b) easily follow by elementary algebra. To prove point c) 
  let us define
$$
g(c) := (1 - \Phi(c)) \phi(c-s) -  (1 - \Phi(c - s)) \phi(c) \;.
$$
Then straight forward calculations yield
$$
f'(c) > 0  \quad  \Leftrightarrow  \quad   g(c+s) > g(c)
$$
Let us consider at first the case of $s > 0$. In this situation it is enough to show that g(c) is increasing. We find
$$
g'(c) = (1 - \Phi(c)) \phi'(c-s)  - (1 - \Phi(c-s)) \phi'(c)
$$
and define   
$h(c) = \frac{\phi'(c)}{1 - \Phi(c)}$. Then clearly
$$
g'(c) > 0 \quad  \Leftrightarrow  \quad   h(c-s) > h(c).
$$

To show that $h(c)$ is a decreasing function observe that
$$
h'(c)=\frac{1}{2\pi(1-\Phi(c))^2}e^{-c^2/2}\left(\sqrt{2\pi}(c^2-1)(1-\Phi(c))-ce^{-c^2/2}\right) \; .
$$
Now, the standard bound on the tail of the normal distribution yields
$$
\sqrt{2\pi}c^2(1-\Phi(c))<ce^{-c^2/2}\;\;,
$$
which implies that $h'(c)<0$.

 The proof for $s<0$ goes analogously. In that case $g(c)$ has to be decreasing, which yields  $h(c) > h(c - s)$, and again $h(c)$ is a decreasing function. 
 
\end{pro} 
 
The following Lemma \ref{WS_L2} easily follows from Lemma \ref{Lem:pointmass}.
\begin{lem}\label{WS_L2}
Let $\nu(\cdot)$ be any probability measure such that $\nu({0})<1$. Let us define 
\begin{equation}\label{Hc}
H(c):= \frac{\int_R (2 - \Phi(c-\sqrt{n}\mu/\s) - \Phi(c+\sqrt{n}\mu/\s)) d\nu(\mu)}{2(1 - \Phi(c))}\;\;.
\end{equation}
 Then it holds
\begin{itemize}
\item[a)] $H(0)=1$,
\item[b)] $\lim_{c\rightarrow \infty} H(c)=\infty$,
\item[c)] $H(c)$ is  increasing on $[0, \infty]$.
\end{itemize}
\end{lem}

\begin{lem}\label{final}

Let $\nu(\cdot)$ be any probability measure such that $\nu({0})<1$. 
Let 
$$
BFDR(c)=\frac{(1-p)t_{1}(c)}{(1-p)t_{1}(c)+ p\ (1-t_{2}(c))}\;\;,
$$
where 
$$
t_1(c) = 2(1-\Phi(c))
$$
and
$$
t_2(c) = 1-\int\limits_{\mathbb R}(\Phi(-c -\sqrt n \mu/\sigma)+1-\Phi(c - \sqrt n \mu/\sigma))d \nu \; .
$$

Then $BFDR(c)$ is continously decreasing from $1-p$ for $c=0$ to 0 for $c\rightarrow \infty$.
\end{lem}
\begin{pro}
Observe that 
$$BFDR(c)=\frac{1}{1+\frac{p}{1-p}H(c)}\;\;,$$
with $H(c)$ as in (\ref{Hc}).
Thus Lemma \ref{final} is a direct consequence of Lemma \ref{WS_L2}.
\end{pro}

\subsection{Proof of Theorem \ref{TH_BFDR}} \label{Sec:App4}
\ \\

\begin{pro}

Let us define $u^B_n=c_B \sigma/ \sqrt{n}$. First we want to show that $u^B_n$ is bounded.  Assume on the contrary that for some subsequence $u^B_j \rightarrow \infty$. It holds for any constant $K>0$, that
\bea
& & \int\limits_{\mathbb R} \Phi(\sqrt j (-u^B_j - \mu)/\sigma)d \nu+ 1 - \int\limits_{\mathbb R} \Phi(\sqrt j (u^B_j - \mu)/\sigma)d \nu    \\
 &\geq &(\nu(-\infty,-K)+\nu(K,\infty)) (1 - \Phi(\sqrt j (u^B_j - K)/\sigma)) \; .  
\eea
If $u^B_j \rightarrow \infty$ we can apply  the tail approximation for the normal distribution and obtain from (\ref{BFDR_rule})
$$
\frac{\alpha_j}{f} \leq \frac{(1 -\alpha_j)(u^B_j - K)}{(\nu(-\infty,-K)+\nu(K,\infty))u^B_j}
 \exp\left(- \frac{j(u^B_j K - K^2)}{2 \sigma^2}  \right)
(1 + o_j) \; .
$$
But on the other hand the second assumption of (\ref{con3_general}) yields $\left(\frac{\alpha_j}{f}\right)^{1/j} \rightarrow \exp(- C_0)$ which contradicts $u^B_j \rightarrow \infty$.

If $u_j: = u_j^B \rightarrow u < \infty$   then the denominator of (\ref{BFDR_rule}) converges  to a constant $C_{\nu,u}= 1- \nu(-u,u)$.  Under the first assumption of (\ref{con3_general}) equation (\ref{BFDR_rule}) can only hold if $\sqrt j u_j \rightarrow \infty$. Thus we can apply again the tail approximation to obtain
$$
\frac{\alpha_j}{f} = \sqrt{\frac{2}{\pi}}\frac{1 - \alpha_j}{c_B C_{\nu,u}} \exp(- \frac{c_B^2}{2})(1 + o_j) \; .
$$
Taking logarithms and some simple calculations yield
$$
c^2_B =  2\log \left(\frac{f}{\alpha}\right)
- \log \left(  2 \log \left(\frac{f}{\alpha}\right)\right) +
\log \left(\frac{2}{\pi}\right)+2\log\left(\frac{1 - \alpha_\infty}{C_{\nu,u}}\right) + o_j \;\;.
$$
Now, the second condition  in (\ref{con3_general}) implies that $u=\sigma \sqrt{2 C_0}$, which completes the proof of (\ref{BFDR2_general}).

The critical value has exactly the same form as in the case of normal distributions and the result on ABOS follows exactly the same way as in \citep{Bog2009}. Define $s_n := \frac{\log(f \delta \sqrt n)}{\log(f/\alpha)} - 1$, then necessary and sufficient conditions for optimality are $s_n \rightarrow 0$ and $2s_n \log(f/n) - \log \log(f/\alpha) \rightarrow -\infty$ which immediately provides (\ref{BFDR_opt}). From the first equation in (\ref{BFDR_opt}) it follows that in case of ABOS $C_0 = C/2$, where $C$ is the constant from Assumption (A).
\end{pro}

\subsection{Lemmas needed for Theorem \ref{TH:BH}}\label{ApSec:Th2.4}

To prove optimality of the type II risk component of SD in the denser case 
we first  show that with  large probability the random threshold of SD is bounded from above by the asymptotically optimal threshold $\tilde c_{1n}$.

\begin{lem}\label{LSD2}
Let $c_{SD}$ be the random threshold SD threshold at the level $\alpha_n$ and let $\tilde c_1 = \tilde c_{1n}$ be the GW threshold (\ref{GW2}) at the level
$\alpha_{1n}=\alpha_n \xi_m$, where $\xi_m = (\log m)^{-s}$ with $s > 1$. Suppose that Assumptions (A) and (C), (\ref{pm}), (\ref{alpha}) and (\ref{ularge})  hold with $\alpha=\alpha_n$.   
Then $\tilde c_1$ is ABOS. Moreover, for every $\gamma_u > 0$ it holds for sufficiently large $m = m_n$ that
\beq \label{bound:cSD}
P(c_{SD}\geq \tilde c_1 )\leq m^{- \gamma_u}\;\;\;.
\eeq
\end{lem}
\begin{proof}
Based on the second condition in (\ref{ularge}) it is easy to show that $\alpha_{1n}$ satisfies the asymptotic optimality assumptions provided in  Theorem \label{TH_BFDR}. Thus, Theorem   \ref{TH:GW} immediately yields that $\tilde c_{1}$ is ABOS.

To prove the second assertion of the Lemma we first note that by Lemma \label{WS_L2} the function $\tilde H(c) := \frac{2(1-\Phi(c))}{1 - F(c)}$ is  decreasing. Therefore according to the definition of $\tilde c_{1}$,
\beq \label{bound:cSD1}
\left\{c_{SD}\geq\tilde c_{1}\right\} = \left\{\tilde H(c_{SD}) \leq \alpha_{1n}\right\}\;\;.
\eeq
On the other hand the definition of $c_{SD}$ actually gives
$$
\frac{2(1-\Phi(c_{SD}))}{1-\hat F_m(c_{SD}) + 1/m} = \alpha_n
$$
and thus
$$
\left\{c_{SD}\geq\tilde c_{1}\right\} =
\left\{
\frac{1-\hat F_m(c_{SD})+ 1/m}{1- F(c_{SD})} \leq \xi_m
\right\} \; .
$$
Taking another intersection of the right hand side with $\left\{c_{SD}\geq\tilde c_{1}\right\} $ we can conclude that
\beq \label{bound:cSD2}
P(c_{SD}\geq\tilde c_{1}) \leq P\left(\inf\limits_{c \geq \tilde c_{1}} \frac{1-\hat F_m(c)+ 1/m}{1- F(c)} \leq \xi_m \right) \; .
\eeq
Using the standard transformation $U_i = F(|Z_i|)$ one obtains
$$
P(c_{SD}\geq\tilde c_{1}) \leq P\left(\inf\limits_{t \in [z_{1m}, 1]} \frac{1-\hat G_m(t)+ 1/m}{1- t} \leq \xi_m \right)
$$
where  $z_{1m}=F(\tilde c_1)$, and $\hat G_m(t)$ is the empirical cdf of $U_1,\ldots,U_m$. Now, using the transformation $u = 1-t$ and observing that $V_i=1-U_i$ also has a uniform distribution  we obtain
$$
P(c_{SD}\geq\tilde c_{1}) \leq P\left(\inf\limits_{u \in [0, 1 - z_{1m}]} \frac{\hat G_m(u) + 1/m}{ u} \leq \xi_m \right)
$$
This is equivalent to computing the probability that the empirical process $\hat G_m(u)$ intersects the line $L = -\frac{1}{m} + u \xi_m$ within the interval $[\frac{1}{m \xi_m},1 - z_{1m}]$.
 For this type of problem Proposition 9.1.1 of \cite{SW}
  can be applied. Define the event
$$
B_i = \{\hat G_m(u) \mbox{ intersects the line }  y = (u - a)/(bm)  \mbox{ at height }
i/m \mbox{ but not below}\}
$$
Then
$$
P(B_i) = {m \choose i} a(a + ib)^{i-1} (1 - a - ib)^{m-i} \;.
$$
 In our case  $a = b = \frac{1}{m \xi_m}$  and thus
$$
P(B_i) = {m \choose i} \frac{1}{m \xi_m} \left(\frac{1 + i}{m \xi_m}  \right)^{i-1} \left(1 - \frac{1 + i}{m \xi_m}  \right)^{m-i}
\quad \mbox{ for }  \quad  i < m \xi_m - 1 \; 
$$
and $P(B_i)=0$ for $i \geq m \xi_m - 1$.

Now, similar to Lemma 10.3.1 of \cite{SW} (page 414) we can apply Stirling's formula, which  for $i < m \xi_m - 1$ yields
\bea
P(B_i) &<&
\frac{m !}{(i+1)!(m-i)!}  \left(\frac{1 + i}{m \xi_m}  \right)^{i} \left(1 - \frac{1 + i}{m \xi_m}  \right)^{m-i} \\
&<&
\frac{m^{m + 1/2} e^{-m} \sqrt{2\pi} \exp(1/12 m)}
{(i+1)^{i + 3/2} e^{-(i+1)} \sqrt{2\pi} (m-i-1)^{m - i + 1/2} e^{-(m-i)} \sqrt{2\pi}}
\left(\frac{1 + i}{m \xi_m}  \right)^{i} \left(1 - \frac{1 + i}{m \xi_m}  \right)^{m-i}
 \\
&<&
\frac{ \exp(1/12 m+1)}{\sqrt{2\pi}}
\frac{1}{(i+1)^{3/2} \sqrt{1 - i/m}}\ \xi_m^{-i}  \left(\frac{m - (1 + i)/\xi_m }{m -i}  \right)^{m-i}
\\
&<&
\frac{ \exp(1/12 m+1)}{\sqrt{2\pi}}
\frac{1}{(i+1)^{3/2} \sqrt{1 - i/m}}\ \xi_m^{-i}
\exp\left( -i \left(\frac{1+i}{i\xi_m}-1\right) 
\right)
\eea
In the last step we adapted the inequality $
\left(1 - \frac{i(\lambda - 1)}{n - i}  \right)^{n-i} <
e^{-i(\lambda - 1)}
$ used by Shorack and Wellner in the proof of Lemma 10.3.1. In summary we find that
$$
 P(B_i) < K \xi_m^{-i}   \exp(-(i+1)/\xi_m)  \; .
$$
for some constant $K$ which can be chosen such that it does not depend on $m$ or $i$.
As long as $\frac{1}{\xi_m}\exp(-1/\xi_m) < 1$ we then have
$$
P(c_{SD}\geq\tilde c_{1})  \leq K \sum\limits_{i=0}^\infty
  \xi_m^{-i}   \exp(-(i+1)/\xi_m)
  = K \ \frac{\exp(- 1/\xi_m)}{1 - \frac{1}{\xi_m}\exp(-1/\xi_m)} \; .
$$
Remembering that  $\xi_m = (\log m)^{-s}$ with $s > 1$ finally yields (\ref{bound:cSD}).

\end{proof}

The next lemma discusses the type II error component of the risk of SD.
\begin{lem} \label{Lem:type2_dense}
Under the assumptions of Theorem \ref{TH:BH}  the type II error component of the risk of SD satisfies
\begin{equation}\label{t2risk}
R_A \leq R_{B}(1+o_{m})\;\;.
\end{equation}
\end{lem}
\begin{proof}
For the extremely sparse case (\ref{as_spars}) we have seen already that the result follows by comparing with the Bonferroni rule which is ABOS according to Lemma \ref{L1}. It remains to show the result for the denser case (\ref{dense}) and to note that both cases overlap.

Denote by $L_A$ the number of false negatives under the SD rule and let $\tilde c_1$ be defined as in Lemma \ref{LSD2}. Clearly
$$
E(L_A)\leq E(L_A| c_{SD}\leq \tilde c_1)P(c_{SD}\leq \tilde c_1)+ m P(c_{SD} >  \tilde c_1)\;\;,
$$
and furthermore
$$E(L_A| c_{SD}\leq \tilde c_1)P(c_{SD}\leq \tilde c_1)\leq  E L_1\;\;,$$
where $L_1$ is the number of false negatives produced by the rule based on the threshold $\tilde c_1$.
Since by Lemma \ref{LSD2} the rule based on $\tilde c_1$ is asymptotically optimal, it follows that $\delta_A EL_1=R_{opt}(1+o_{m})$. On the other hand 
 $P(c_{SD} >  \tilde c_1) \leq m^{-\gamma_u}$ for any $\gamma_u > 0$ if only $m$ is sufficiently large, and therefore
$$
R_A=\delta_A E L_A  \leq R_{B}(1+o_m)+\delta_A m^{1-\gamma_u}\;\;.
$$
Now by using assumptions  (\ref{dense}) and (\ref{ularge}), and choosing e. g. $\gamma_u = \gamma_2/2+1$, we conclude that $\delta_A m^{1-\gamma_u}=o(R_{opt})$, and the proof is thus complete.
\end{proof}

\subsection{Lemma needed for Theorem \ref{TH3}} \label{Sec:ApCons}
\ \\
\begin{lem} \label{Lem:consistency}
Assume that (\ref{as1}), (\ref{as2}), (\ref{as3}) as well as assumptions (B) and (C) hold. 
Then the following bounds are valid for the type I and  type II error rates of mBIC: 
\beq \label{t1t2_mBIC_bounds}
\frac{t_1}{p} = O\left( (n \log n)^{-1/2} \right),\quad t_2 = O\left(\sqrt{\frac{\log n}{n}} \right) \; .
\eeq
\end{lem}
\begin{proof}
Let $h_{n,m} : = \log n + 2 \log m + d$. From the tail approximation of the standard normal distribution we immediately obtain
$$
t_1 \sim \sqrt{\frac{2}{\pi h_{n,m}}} e^{- h_{n,m}/2} \leq \frac{c}{m} (n \log n)^{-1/2}
$$
for some constant $c$. Using the fact that $mp \rightarrow s > 0$ from assumption (\ref{as3}) gives the first bound of (\ref{t1t2_mBIC_bounds}).

To bound type II error we proceed similarly as in the proof of Theorem \ref{TH1_general}. We have
$t_2 = \int \Psi_n(\mu) \ d\nu(\mu)$ with
$$
\Psi_n(\mu) = \Phi\left(\sqrt{h_{n,m}} - \frac{\sqrt n  \mu}{\sigma}\right)
- \Phi\left(-\sqrt{h_{n,m}} - \frac{\sqrt n  \mu}{\sigma}\right) \; .
$$

The asymptotic behavior of this integral is obtained by similar analysis like that leading to (\ref{t2_general}),
resulting in
\begin{equation}\label{t_2_app}
t_2  = \s (\rho(0^-)+\rho(0^+)) \frac{\sqrt{\log n + 2 \log m}}{\sqrt n} (1 + o_{n,m}),
\end{equation}
which completes the proof of Lemma \ref{Lem:consistency} (since $m\leq n$).

\end{proof}

\subsection{Proof of Theorem \ref{TH:modBIC}} \label{Sec:App5}
\ \\

\begin{pro}
In \citep{A} it was shown that the step-up procedure BH corresponds to the smallest local minimum of the selection criterion (\ref{FDR_proc}), whereas SD corresponds to the largest local minimum of (\ref{FDR_proc}). Now mBIC1 is searching for the global minimum of (\ref{mBIC1_pen}), but we can again consider the smallest local minimum as well as the largest local minimum of (\ref{mBIC1_pen}). These will correspond to step-up and step-down procedures based on the comparison
$$
\frac{n \hat\beta_{[k]}^2}{\sigma^2} \geq \log{nm^2} + d - 2 \log(k) - \log\log(n m^2/k^2) \;.
$$
Translating this comparison to the level of p-values when applying the usual tail approximation for the standard normal distribution yields
\beq \label{mBIC1_SU_SD}
p_{[k]} \leq \frac{A k}{m} \quad \mbox{with} \quad
A^2 = \frac{2e^{-d}}{\pi n  z(k,m,n)},
\eeq
where $z(k,m,n)= 1 + \frac{d - \log \log (n m^2 / k^2)}{\log (n m^2 / k^2)}$. Since for sufficiently large $n$
$$
1-\frac{2 \log \log n}{\log n}=z_1(n)<z(k,m,n)<z_2(n)=1-\frac{\log \log n}{6\log n}\;\;,
$$
it holds that mBIC1 can be sandwiched between the step-up and step-down BH procedures with the FDR levels    $\alpha_i= \sqrt{\frac{2e^{-d}}{\pi n  z_i(n)}},\  i=1,2$,
correspondingly. Since both $\alpha_i$ satisfy (\ref{BFDR_opt}) the conditions of Theorem  \ref{TH:BH}  are fulfilled and mBIC1 is itself ABOS.

Similar considerations  give the result for mBIC2, for which we obtain $z(k,m,n) = \log (n m^2 / k^2 + d)$ in (\ref{mBIC1_SU_SD}). Using the inequalities
$$
\log n < z(k,m,n) < 3 \log n
$$
for $n$ large enough to sandwich mBIC2 between step-up and step-down procedures, ABOS of mBIC2 follows immediately from the fact that  $\alpha \propto \frac{1}{\sqrt{n \log n}}$ fulfills (\ref{BFDR_opt}).

Finally for mBIC3 we get
$$
z(k,m,n) =
e^2  (1 - 1/k)^{2(k-1)}  \left(\log (n m^2 / k^2) + d + 2 + 2(k-1)\log(1 - 1/k)\right) \; ,
$$
and we find again
$
\log n < z(k,m,n) < 3 e^2 \log n
$ which yields ABOS as above.

The consistency result is obtained the following way. From ABOS and the Markov inequality  in (\ref{Markov}) one easily concludes that all three criteria are consistent exactly when the Bayes oracle is consistent. Consider the asymptotic formulas concerning  type II error  (\ref{t2_general}) and type I error (\ref{conpr}) for the special case $\delta = 1$. Then it immediately follows that the Bayes oracle is consistent under assumption (\ref{as3}).

\end{pro}

\subsection{Proof of Theorem 3.3}\label{ApSec:Th3.3}
\ \\

Recall that in our setting (two groups model,  orthogonality) it is reasonable to think in terms of type I error  (misclassification of a regressor under $H_0$) and type II error (misclassification of a true signal) for model selection procedures.  To prove Theorem 3.3 we first bound the type I and the type II errors in Lemma \ref{Lem:t1_suk} and Lemma \ref{Lem:t2_suk} respectively.   Both these results will be proved assuming minimal conditions under which the individual bounds on type I and type II errors hold. The conditions in Theorem 3.3 ensure that both lemmas hold and additionally that the overall upper bound on the total risk of mBIC is asymptotically equivalent to that of the Bayes Oracle. \\

To bound the type I error we will make use of the following corollary  given after Theorem 2 of Section 16.7, Vol.2 of \cite{Fell}.
\begin{cor}\label{Feller}
 Let $F$ be the common distribution function of i.i.d. random variables $X_1, \ldots, X_n$ with $E(X_i) = 0, \mbox{Var }(X_i) = \s^2$ and let $F_n$ be the distribution function of the normalized sum  $(X_1 + \cdots + X_n)/( \sqrt{n}\s)$. If $1< x = o(\sqrt{n})$, then for any $\e > 0$, for all sufficiently large $n$,
\beq \label{Cor_Feller}
\frac{\exp(-(1+\e)x^2/2)}{\sqrt{x}} < 1 - F_n(x) < \frac{\exp(-(1-\e)x^2/2)}{\sqrt{x}}
\eeq
\end{cor}

\begin{lem} \label{Lem:t1_suk}
Assume $n \rightarrow \infty$, $m=m(n) \rightarrow \infty$  and that (\ref{as1}), (\ref{as2}) and (\ref{as4}) hold.
Then the type I error probability of the decision rule based on mBIC criterion (\ref{mBIC}) is bounded by
\begin{equation}\label{ttypeI}
 t_1\leq \frac{C_1}{\sqrt{n}{m}}(1+o_{n,m})\;\;,
 \end{equation}
 with $C_1 = \frac{\sqrt{2}}{\sqrt{\pi}} \exp(-d/2)$.
\end{lem}
\begin{pro}
Let $1 \leq i \leq m$. Then the probability of type I error correspoding to $\beta_i$ is given by
$$
t_{1i} = P(A_i|B_i),
$$
where $B_i$ denotes the event that $\beta_i=0$ and $A_i$ denotes the event that the corresponding regressor  is included in the model chosen by mBIC. Through exchangeability, it follows that $t_{1i} = t_1$, for each $1 \leq i \leq m$. Let us compare two models $M$ and $M\bigcup X_i$ where $ X_i \notin M$.  mBIC considers supplementing model $M$ with the variable $X_i$ only in the case
\begin{equation}\label{neq1}
\log \frac{RSS_M}{RSS_{M} - n \hat\beta_i^2} \geq  \frac{\log n + 2 \log m + d}{n}
\end{equation}
where $RSS_M = Y'Y - \sum\limits_{j \in M} n \hat\beta_j^2$ denotes the residual sum of squares of the model $M$ (we have used the orthogonality assumption (\ref{as1}) here). Henceforth we will use the abbreviations $Z_{i,M}:= \log \frac{RSS_M}{RSS_{M} - n \hat\beta_i^2}$ and $u_{n,m}:=\frac{\log n+2\log m + d}{n}\;$. Note that if $M_1 \subset M_2$ then $RSS_{M_1} \geq RSS_{M_2}$. Therefore $RSS_{M_1}/(RSS_{M_1} - n \hat\beta_i^2)\leq RSS_{M_2}/(RSS_{M_2} - n \hat\beta_i^2)$ and the event that a given false positive is added to the model $M_1$ is contained in the event that it is added to the model $M_2$. According to these considerations we obtain an upper bound for the type I error
$$t_1 \leq P(\bigcup_{M\in \Omega_L}
\{Z_{i,M} \geq u_{n,m}\}|B_i),$$
where  $\Omega_L$ is the set of all models with $L-1$ regressors in addition to the the common intercept term, such that $X_i \notin \Omega_L$. 
We bound the probability of the right hand side above in three intermediate steps.

\noindent {\bf Step 1:} Let  $T_i = \frac{\hat\beta_i^2}{\s^2}$ and $\epsilon_{n,m}=\frac{\log(\log n+2\log m)}{n}$. Then
$$\{Z_{i,M}  \geq u_{n,m}\}  \subset
\{ Z_{i,M} -  T_i  \geq \epsilon_{n,m}\} \cup
\{ T_i  \geq
u_{n,m}-\epsilon_{n,m}\}\;\;,$$
and therefore
\begin{eqnarray*}
 &&P(\bigcup_{M\in \Omega_L} \{Z_{i,M}  \geq u_{n,m}\}|B_i)\leq\\
&& P(\bigcup_{M\in \Omega_L} \{Z_{i,M} -  T_i \geq \epsilon_{n,m}\}|B_i)
+P(T_i\geq u_{n,m}-\epsilon_{n,m}|B_i)\;\;.\nonumber
\end{eqnarray*}
The second term on the right hand side of the above inequality can be expressed as
$$
P(T_i\geq u_{n,m}-\epsilon_{n,m}|B_i)=P\left(\frac{n\hat
\beta^2_i}{\sigma^2}\geq \log n + 2\log m + d - \log(\log n+2 \log
m)|B_i\right)\;\;,$$
and from the normal tail approximation we obtain
$$
P(T_i\geq
u_{n,m}-\epsilon_{n,m}|B_i)=C_1\frac{1}{\sqrt{n}{m}}(1+o_{n,m})\;\;,
$$
where $C_1=\sqrt{\frac{2}{\pi}} \exp(-d/2)$. Now to establish inequality (\ref{ttypeI}) it remains to be shown that  $P(\bigcup_{M\in \Omega_L} \{Z_{i,M} -  T_i\geq \epsilon_{n,m}\}|B_i)$ is of  lower order.

\vspace{3mm}

\noindent {\bf Step 2:}  Let $\delta_n = \frac{1}{n}$. Similar arguments as in Step 1 yield
$$\{Z_{i,M} -T_i \geq \epsilon_{n,m}\} \subset
\{Z_{i,M}>-\log(1-(T_i+\delta_{n}))\} \cup \{-T_i -\log(1-(T_i+\delta_{n}))\geq
\epsilon_{n,m}\}\;\;.$$
The first set on the right hand side can be rewritten as $\{\frac{n \hat\beta_i^2}{RSS_M}>T_i+\delta_{n}\}$ and therefore
\begin{eqnarray}
&&P(\bigcup_{M\in \Omega_L} \{Z_{i,M} -  T_i \geq
\epsilon_{n,m}\} |B_i)\nonumber\\
&&\leq P(\bigcup_{M\in \Omega_L}
\{\frac{n \hat\beta_i^2}{RSS_M} >T_i+\delta_{n}\}|B_i)+P(-T_i -\log(1-(T_i+\delta_{n}))\geq
\epsilon_{n,m}|B_i)\;\;.\nonumber
\end{eqnarray}
To bound the second term  note that $-\log (1-x) \leq x + 2x^2$ for $0 \leq x \leq 1/2$. Hence
$$P(-T_i -\log(1-(T_i+\delta_{n})|)\geq
\epsilon_{n,m} |B_i)\leq
P\left((T_i+\delta_{n})^2>\frac{\epsilon_{n,m}-\delta_n}{2}|B_i\right)+P(T_i +\delta_n \geq 1/2 |B_i)\;\;.
$$
First note that
$$
P\left((T_i+\delta_{n})^2>\frac{\epsilon_{n,m}-\delta_n}{2}|B_i \right)=P\left(nT_i>\frac{\sqrt{n\log(\log
n+2\log m)}}{2}(1+o_{n,m})|B_i\right)\;\;
$$
and for sufficiently large $n$ the normal tail approximation yields
$$
P\left((T_i+\delta_{n})^2>\frac{\epsilon_{n,m}-\delta_n}{4}|B_i\right)< \exp\left(-\frac{\sqrt{n}}{2}\right)\;\;.
$$
Second we have
$$
P(T_i + \delta_n \geq 1/2|B_i)=P(nT_i \geq n/2-1|B_i)\leq \sqrt{\frac{2}{\pi}} \exp(-(n/4-1/2)).
$$
Combining the two bounds obtained above, it follows that
$$
P(-T_i -\log(1-(T_i+\delta_{n}))\geq
\epsilon_{n,m}|B_i) = o\left(\frac{1}{\sqrt{n}{m}}\right),
$$
since $m < n$. \\
\noindent {\bf Step 3:} We will now bound the remaining term
$$
P\left(\bigcup_{M\in \Omega_L}
\left\{\frac{n \hat\beta_i^2}{RSS_M}>T_i+\frac{1}{n}\right\}|B_i\right)\;\;.
$$
Observing that
$$\left\{\frac{n\hat \beta_i^2}{\sigma^2}\left(\frac{n\sigma^2}{RSS_M}-1\right) > 1\right\}\subset\left\{\frac{n\hat \beta_i^2}{\sigma^2} > c_{n,m}\right\}\cup \left\{\frac{n\sigma^2}{RSS_M}-1 >  \frac{1}{c_{n,m}}\right\}\;,$$
where we choose $c_{n,m}=\log n+2\log m$, we conclude that
$$
P\left(\bigcup_{M\in \Omega_L} \left\{\frac{n \hat\beta_i^2}{RSS_M}>T_i+\frac{1}{n}\right\}|B_i\right)\leq P\left(\frac{n\hat \beta_i^2}{\sigma^2}\geq c_{n,m}|B_i\right)+ \sum_{M\in \Omega_L} P \left( \frac{n\sigma^2}{RSS_M}-1\geq \frac{1}{c_{n,m}}|B_i\right)\;\;.
$$
By the tail approximation of the standard normal distribution we obtain that for sufficiently large $n$
$$
P\left(\frac{n\hat \beta_i^2}{\sigma^2}\geq c_{n,m}|B_i\right)=\frac{1}{\sqrt{n}m}o_{n,m}\;\;.
$$
We now have to bound the remaining series. Fix any model $M \in \Omega_L$ and assume that $k$ true signals are not included in $M$. Under assumptions (\ref{as1}) and (\ref{as2}) it immediately follows that 
$RSS_M =  W_k + Z_k$, where $Z_k = n \sum \limits_{r=1}^k \hat{\beta}_{j_r}^2$  refers to the $k$ true signals which were not detected, and  $ W_k \sim \sigma^2 \chi^2_{(n-L-k)}$ is the remainder term. 
$Z_k$ and $W_k$ are independent and $Z_k$ is stochastically larger than a $\sigma^2 \chi^2_{k}$ distributed random variable. Therefore  $RSS_M$  is stochastically larger than  $\sigma^2 \chi^2_{(n-L)}$. But this argument holds for any $k$, and we conclude that
$$
P\left( \frac{n\sigma^2}{RSS_M}-1\geq \frac{1}{c_{n,m}}|B_i \right)\leq P\left( \frac{n}{\chi^2_{n-L}}-1 \geq \frac{1}{c_{n,m}}\right) \; .
$$
Now
$$
P\left( \frac{n}{\chi^2_{n-L}}-1\geq \frac{1}{c_{n,m}}\right) =
P\left(\frac{\chi^2_{n-L}-(n-L)}{\sqrt{2(n-L)}}\leq \frac{L-\frac{n}{1+c_{n,m}}}{\sqrt{2(n-L)}}\right)\;\;
$$
will be bounded using a normal tail approximation argument.
>From assumption (\ref{as4}) it follows that $L = o(n/c_{n,m})$ and therefore
$$
\frac{L-\frac{n}{1+c_{n,m}}}{\sqrt{2(n-L)}} = - \frac{\sqrt{n}}{\sqrt{2}(\log n +2 \log m)}(1+o_{n,m})\;\;.
$$
Applying Corollary \ref{Feller} with $x = \frac{\sqrt{n}}{\sqrt{2}(\log n +2 \log m)} = o(\sqrt{n-L})$  yields that for every $\epsilon>0$ and $n$ large enough (dependent on $\epsilon$) it holds
$$
P\left(\frac{\chi^2_{n-L}-(n-L)}{\sqrt{2(n-L)}} \leq
\frac{L-\frac{n}{1+c_{n,m}}}{\sqrt{2(n-L)}}\right)
\leq \exp\left(-\frac{(1-\e) n^2}{2 (\log n +2 \log m)^2}\right)\;\;.$$

 \vspace{3mm}

The number of models with $L-1$ regressors is
 ${m \choose L-1} < m^{L}$. Thus, for sufficiently large $n$
 $$\sum_{M\in \Omega_L} P\left( \frac{n\sigma^2}{RSS_M}-1\geq \frac{1}{c_{n,m}}\right)\leq m^{L} \exp\left(-\frac{(1-\e) n}{4 (\log n +2 \log m)^2}\right) = o\left(\frac{1}{\sqrt{n}{m}}\right)\;\;,$$
 which finishes the proof.
  \end{pro}


Next we compute a bound for the type II error:
 \begin{lem} \label{Lem:t2_suk}
Assume $n \rightarrow \infty$, $m=m(n) \rightarrow \infty$  and that (\ref{as1}), (\ref{as2}), (\ref{as4}), (\ref{asL}), (\ref{as3})  and Assumption (C) hold.
 Then the type II error of the decision rule based on mBIC criterion (\ref{mBIC}) is bounded by
\beq \label{t2_mBIC}
t_2 \leq  \s (\rho(0^-)+\rho(0^+)) \frac{\sqrt{\log n + 2 \log m}}{\sqrt n} (1 + o_{n,m})
\eeq
\end{lem}
\begin{pro}
Let $1 \leq i\leq m$ and suppose $\tilde B_i$ is the event that $\beta_i \neq 0$ and let $\tilde A_i$ be the event that the corresponding regressor is not detected. Then we have type II error $t_{2i} = P(\tilde A_i|\tilde B_i)$, and by exchangeability it follows that $t_{2i} = t_2$ is independent of $i$, for each $ 1 \leq i \leq m$.

Let us introduce the symbol $D$ to denote the event that  none of the $X_j$'s corresponding to the null hypothesis are included in the model chosen by mBIC. Similar to the proof of Lemma \ref{Lem:t1_suk} one can show  that for every $i\neq j$, $P(A_j|B_j,\tilde B_i) = O\left(\frac{1}{\sqrt{n}{m}}\right)$. Then $P(D^c|\tilde B_i) = O(\frac{1}{\sqrt n})$ and thus 
$$
t_2 = P(\tilde A_i \cap D|\tilde B_i) + O(n^{-1/2}) \;.
$$
 To shorten the notation we now define $\tilde A^D_i = \tilde A_i \cap D$.

Note that
 $$
 t_{2}=\sum_{k=1}^{m} P(\tilde A^D_i  |K=k,\tilde B_i) P(K=k|\tilde B_i)\;\;,
 $$
 where $K$ is the number of nonzero $\beta$'s among $\beta_1, \ldots \beta_m$.
Under assumption (\ref{as2}),  given $\tilde B_i$, $K-1$ has a binomial distribution $B(m-1,p_m)$. Define $L^{\prime} = \left \lfloor mp_m (\log n)^{1+\eta}\right \rfloor$, where $\lfloor z\rfloor$ denotes the largest integer less than or equal to $z$. Using the assumption (\ref{as3}) and Bennett's inequality, it is easy to show that
\begin{equation}\label{benett}
P(K > L^{\prime}|\tilde B_i) = o(n^{-1/2}).
\end{equation}
 Thus
 \begin{equation}\label{delta}
 t_2 \leq \sum_{k=1}^{L^{\prime}} P(\tilde A^D_i  |K=k,\tilde B_i) P(K=k|\tilde B_i) + O(n^{-1/2}) \;\;.
 \end{equation}
Note that here we made use of assumption (\ref{asL}).

 Given $K=k$, let the ordered values of the squares of the estimates of the regression coefficients corresponding to the true regressors among $X_1, \ldots, X_m$ be denoted by $\hat \beta_{(1)}^2\leq \hat \beta_{(2)}^2\leq \ldots \leq \hat \beta_{(k)}^2$.
Clearly
$$
P(\tilde A^D_i  |K=k,\tilde B_i)=\sum_{r=1}^{k} P(\tilde A^D_i  | \hat \beta_i=\hat \beta_{(r)}, K=k,\tilde B_i) P(\hat \beta_i=\hat \beta_{(r)}| K=k,\tilde B_i)\;\;,
$$
and using the fact that $\hat{\beta_i}$'s corresponding to true signals are i.i.d. continuous random variables, the above equation can be rewritten as
$$
P(\tilde A^D_i  |K=k,\tilde B_i)=\frac{1}{k}\sum_{r=1}^{k} P(\tilde A^D_{(r)}  | K=k)\;\;,
$$
where $\tilde A^D_{(r)}$ is the generic event that neither the regressor corresponding to $\hat \beta_{(r)}$ nor any false positives are  detected by mBIC. 
Note that the events $\tilde A^D_{(r)}$'s are nested, i.e, $\tilde A^D_{(r+1)} \subset \tilde A^D_{(r)}$, and thus
\beq \label{nested}
\frac{1}{k}\sum_{r=1}^{k} P(\tilde A^D_{(r)}  | K=k) =
\sum_{r=1}^{k} \frac{r}{k}   P(\tilde A^D_{(r)} \cap (\tilde A^D_{(r+1)})^c    | K=k)  \;\;,
\eeq
where we define $\{\tilde A^D_{(k+1)}| K=k\} = \emptyset$. Thus we can write
$$
t_2 \leq \sum \limits_{k=1}^{L^{\prime}} \sum \limits_{r=1}^k \frac{r}{k} P(\tilde A^D_{(r)} \cap (\tilde A^D_{(r+1)})^c    | K=k) P(K=k|\tilde B_i) + O(n^{-1/2}).
$$
The event $\tilde A^D_{(r)} \cap \tilde (A^D_{(r+1)})^c$ implies that the model chosen by mBIC includes the $(k-r)$  true regressors having the largest absolute values of estimated regression coefficients, denoted by $X_{(r+1)},\dots, X_{(k)}$, while $X_{(j)}$ for $1 \leq j \leq r$ are not included. This event also corresponds to the situation when no false positives are included. Hence the model includes $k-r < L^{\prime} \leq L $ regressors, and in any case we have not yet exhausted our maximum model size. So, since $X_{(r)}$ is not included in the model we can infer that mBIC criterion is getting larger by adding $X_{(r)}$. Denoting by $RSS_{k-r}$  the residual sum of squares of the model including $X_{(r+1)},\dots, X_{(k)}$ (or only the intercept in case $r=k$), we have
$$
P(\tilde A^D_{(r)} \cap (\tilde A^D_{(r+1)})^c  | K=k) \leq  P\left(\log\left(\frac{RSS_{k-r}}{RSS_{k-r} - n\hat \beta_{(r)}^2} \right)\leq u_{n,m} | K=k\right)\;\;.
$$
Since for $x\in(0,1)$, $\log(1/(1-x))\geq x$,
\beq \label{A_est}
P(\tilde A^D_{(r)} \cap (\tilde A^D_{(r+1)})^c  | K=k) \leq
 P\left(\frac{n\hat \beta_{(r)}^2}{RSS_{k-r}}\leq u_{n,m}|K=k\right) \;\; .
\eeq
Under $K=k$, $RSS_{k-r}$ is the sum of two independent random variables, the first being a $\sigma^2 \chi^2_{n-k-1}$ random variable ($\chi^2_{n-k-1}$ being a central chi-square with $(n-k-1)$ degrees of freedom), while the second is $\sum\limits_{j=1}^r n \hat \beta_{(j)}^2$. Therefore
$$
P\left(\frac{n\hat \beta_{(r)}^2}{RSS_{k-r}}\leq u_{n,m}|K=k\right) = P\left(\frac{n\hat \beta_{(r)}^2}{\sigma^2 \chi^2_{(n-k-1)} +\sum\limits_{j=1}^r n \hat \beta_{(j)}^2} \leq u_{n,m}\right) \;,
$$
and because of $\hat \beta_{(r)}^2 > \hat \beta_{(j)}^2 $, for $r > j$ one obtains
$$
P\left(\frac{n\hat \beta_{(r)}^2}{RSS_{k-r}}\leq u_{n,m}|K=k\right) \leq P(n\hat \beta_{(r)}^2\leq \sigma^2 \chi^2_{n-k-1} u_{n,m}(1+o_{n,m}))\;\;,
$$
where $(1+o_{n,m})=\frac{1}{1-r u_{n,m}}$. (Note that $r \leq k \leq L^{\prime} = mp_m (\log n)^{1+\eta} $, and under assumption \ref{as4} $r u_{n,m} \rightarrow 0$ as $n \rightarrow \infty$).

Define $b_n = \frac{\log(\log n)}{4 \log n}$. Then
$$
\left(n-k-1+\sqrt{2}(n-k-1)^{1 - b_n}\right)/n = 1+o_{n,m},
$$ 
and therefore
\begin{eqnarray*}
P(n\hat \beta_{(r)}^2\leq \sigma^2 \chi^2_{n-k-1} u_{n,m}(1+o_{n,m})) \; \leq  \;
P(\hat \beta_{(r)}^2\leq \sigma^2 u_{n,m}(1+o_{n,m})) \\ +
 P\left(\chi^2_{n-k-1}>n-k-1+\sqrt{2}(n-k-1)^{1 - b_n}\right)\;\;.
\end{eqnarray*}
A simple application of Chebyshev's inequality yields 
$$
P\left(\chi^2_{n-k-1}>n-k-1+\sqrt{2}(n-k-1)^{1 - b_n}\right)\leq (n-k-1)^{-1 + 2 b_n} \;.
$$
Now, observe that
\begin{eqnarray*}
&\sum\limits_{k=1}^{L^{\prime}} \sum\limits_{r=1}^k   \frac{r}{k}
\; P\left(\chi^2_{n-k-1}>n-k-1+\sqrt{2}(n-k-1)^{1 - b_n}\right)
\; P(K=k |B_i) & \\
&\leq  
(n-k-1)^{-1 + 2 b_n} E((K+1)/2)  \leq  mp\ (n-k-1)^{-1 + 2 b_n} = o(n^{-1/2})\;\;,&
\end{eqnarray*}
where the last equality follows after some calculations from (\ref{as3}).
Therefore
\beq \label{t2:bound}
t_{2} \leq \sum_{k=1}^{L^{\prime}} \sum_{r=1}^k   \frac{r}{k}
\; P(\hat \beta_{(r)}^2\leq \sigma^2 u_{n,m}(1+o_{n,m}))
\; P(K=k |\tilde B_i) + O(n^{-1/2}) \;\;.
\eeq

Let us define
\beq \label{pa}
q=q_{n,m} := P\left(\hat \beta_{i}^2 \leq \sigma^2 u_{n,m}(1+o_{n,m})\right) \;\;.
\eeq
Given that $\hat\beta_{i} \sim \nu * \NN(0,\s^2/n)$  straight forward computations yield
$$
q  = \int\limits_{\mu \in \R} \left[\Phi\left(\sqrt{n u_{n,m}}(1+o_{n,m}) - \frac{\sqrt n\mu}{\s}  \right) -
\Phi\left(-\sqrt{n u_{n,m}}(1+o_{n,m}) - \frac{\sqrt n \mu}{\s}  \right)
\right] \ d\nu(\mu).
$$
The asymptotic behavior of this integral is obtained by similar analysis like that leading to (\ref{t2_general}),
resulting in
$$
q  = \s (\rho(0^-)+\rho(0^+)) \frac{\sqrt{\log n + 2 \log m}}{\sqrt n} (1 + o_{n,m}).
$$
Define the contribution for fixed $r$ in the sum on the right hand side of (\ref{t2:bound}) as
$$
\Psi_r := \sum_{k=r}^{L^{\prime}} \frac{r}{k}  P(\hat \beta_{(r)}^2\leq \sigma^2 u_{n,m}(1+o_{n,m})) \ P(K=k|\tilde B_i).
$$
For $r = 1$ we observe that
$$
P(\hat \beta_{(1)}^2\leq \sigma^2 u_{n,m}(1+o_{n,m})) = 1 - (1-q)^k = kq - \sum\limits_{j=2}^{k} {k \choose j} (-q)^j \;\;,
$$
with the convention that ${1 \choose 2}$=0.

Thus
\bea
\Psi_1&\leq& q+ 
\sum_{k=2}^{L^{\prime}} \frac{1}{k} \sum\limits_{j=2}^{k} {k \choose j} q^j P(K=k|\tilde B_i)\\
 & = &q+
\sum\limits_{j=2}^{L^{\prime}} \frac{q^j}{j}  \sum_{k=j}^{L^{\prime}} {k -1 \choose j -1}  P(K=k|\tilde B_i)
\\ & \leq &
q+\sum\limits_{j=2}^{L^{\prime}}   \frac{q^j}{j \ (j-1)!} (mp)^{j-1}\\
&\leq& q + q^2 mp e^{mp q} = q(1+o(q))\;\;.
\eea
as long as $mp q  \rightarrow 0$ which is guaranteed by (\ref{as3}).
The first inequality follows from the fact that under $\tilde B_i$, $K \sim \mbox{Bin}(m-1, p)$ and thus
\beq \label{frac_moment}
E((K-1) (K-2) \cdots (K-j+1))
=  (m-1)(m-2) \dots (m - j + 1) p^{j-1}
\leq (mp)^{j-1}.
\eeq

Finally we have to bound the contribution $\Psi_r$ in the sum on the the right hand side of (\ref{t2:bound}) stemming from $r>1$. Note that
$$
P(\hat \beta_{(r)}^2\leq \sigma^2 u_{n,m}(1+o_{n,m})) = \sum_{j=r}^k {k \choose j}q^j (1-q)^{k-j} \;\;.
$$
Similar computations as above using (\ref{frac_moment}) yield
$$
\Psi_r = \sum_{k=r}^{L^{\prime}} \frac{r}{k} \sum\limits_{j=r}^{k} {k \choose j} q^j (1-q)^{k-j} P(K=k|\tilde B_i)
\leq
\sum\limits_{j=r}^{L^{\prime}}   \frac{q^j}{ (j-1)!} s^{j-1}
\leq q^r s^{r-1}  e^{q s}.
$$
Summing over all possible values of $r > 1$ finally gives
$$
\sum\limits_{r=2}^{L^{\prime}} \Psi_r \leq \frac{q^2 s}{1 - q s}  e^{q s} = o(q).
$$
Thus we have  shown that
$$
t_{2} \leq \sum\limits_{r=1}^{L^{\prime}} \Psi_r + O(n^{-1/2}) \leq q(1+o_{n,m}),
$$
since $O(n^{-1/2}) = o(q)$.
This completes the proof of the lemma.
\end{pro}

{\bf Proof of Theorem 3.3} \ \\

\begin{pro}
First note that the assumption on $mp$ in (\ref{as_spars0}) is stronger than that in assumption (\ref{as3}). 
Given (\ref{as_spars0}) it is easy to see that the type II error estimate in Lemma \ref{Lem:t2_suk} is asymptotically of the same form as the type II error of the Bayes oracle for $C = 0$ in (\ref{t2_general}). To show ABOS it is therefore sufficient that the risk component of the type I error is of smaller order than the Bayes risk. From Lemma \ref{Lem:t1_suk} we conclude 
$$
\frac{R_1}{R_{BO}} = O\left(\frac{\delta}{m p \sqrt{\log v}}\right) \;.
$$
Under Assumption (B) $\delta$ is bounded from above and ABOS follows.

Consistency follows exactly the same way as in Theorem \ref{TH3}.
\end{pro}

\newpage

\subsection{Figures of the first part of the simulation study} \label{Sec:ApFig} \ \\
\nopagebreak[4]

\begin{figure}[h!]
\caption{Simulation runs for known $\sigma$. Misclassification probability (MP), False Discovery Rate (FDR) and Power for different selection rules  and sparsity parameter $p$ at values of $p\in\{0.001, 0.005, 0.01, 0.02, 0.05, 0.1, 0.2\}$.} \label{Fig3}
\vspace{9mm}
$\begin{array}{c@{\hspace{-5mm}}c@{\hspace{0mm}}c} 
\multicolumn{1}{l}{\mbox{(a) MP }, n=m=256}
&  \multicolumn{1}{l}{\mbox{(b) MP }, n=m=1024} & \\
\epsfig{file=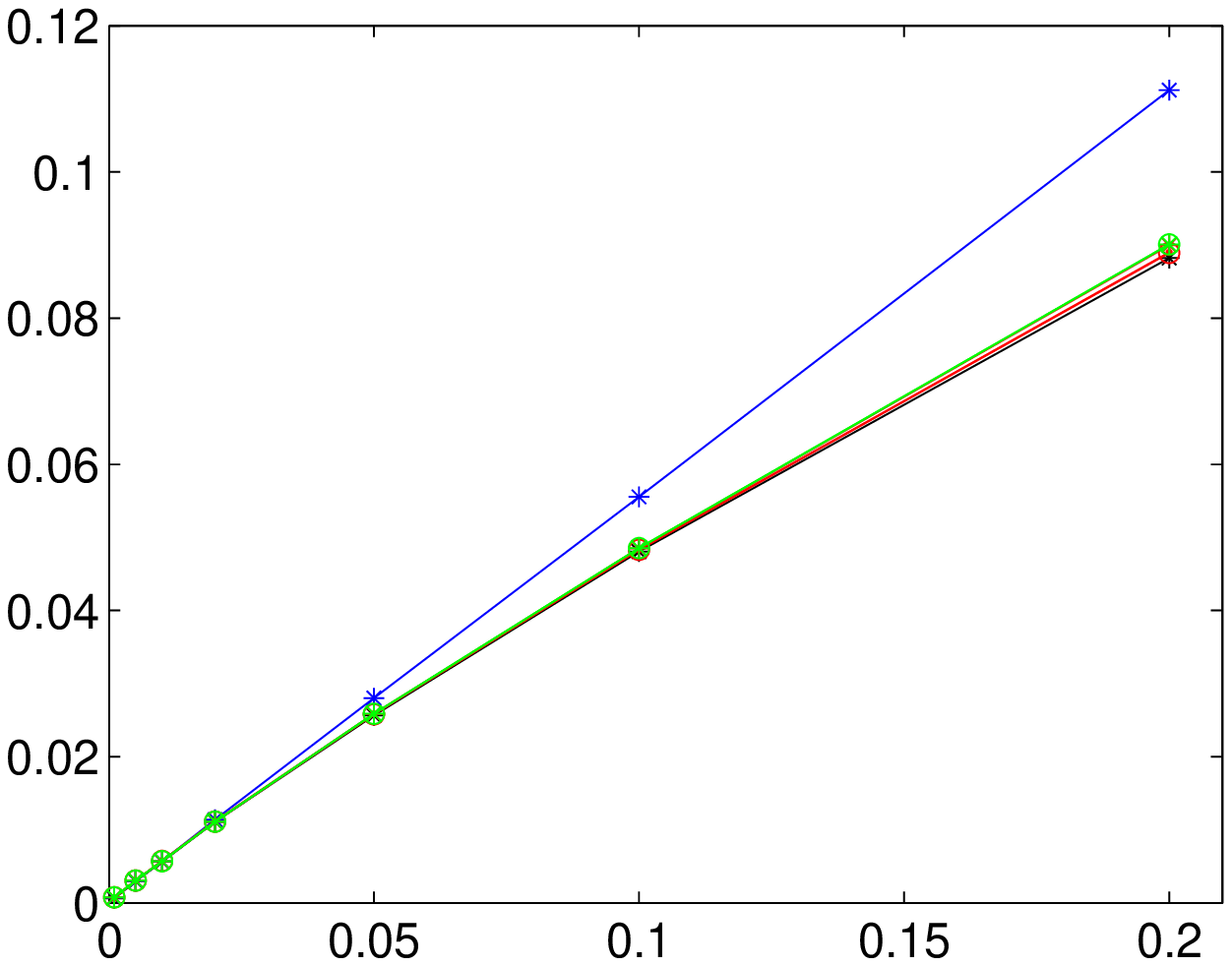,width=7.0cm} &
       \epsfig{file=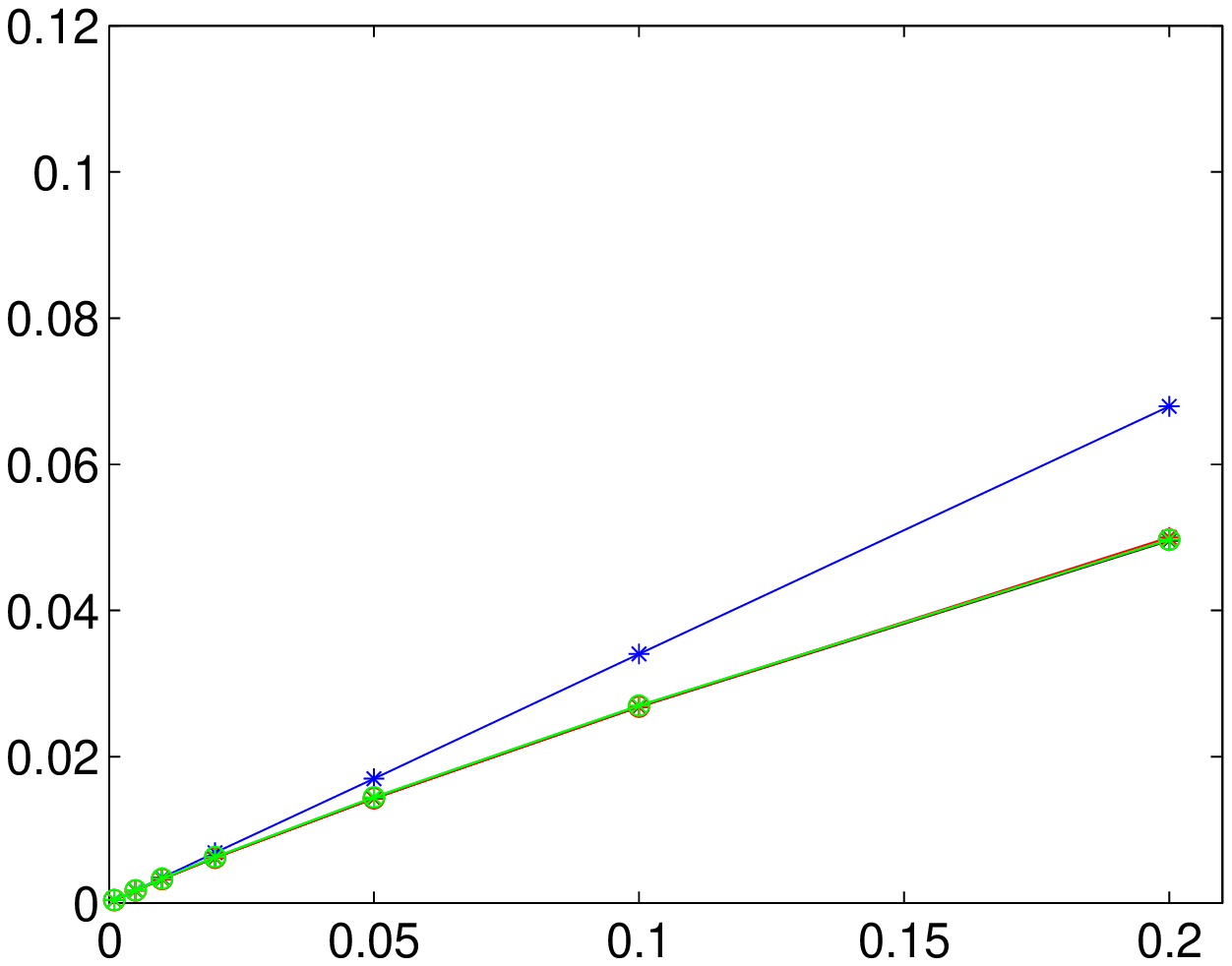,width=7.0cm} &
\\
\multicolumn{1}{l}{\mbox{(c) FDR }, n=m=256}
&  \multicolumn{1}{l}{\mbox{(d) FDR }, n=m=1024} & \\
\epsfig{file=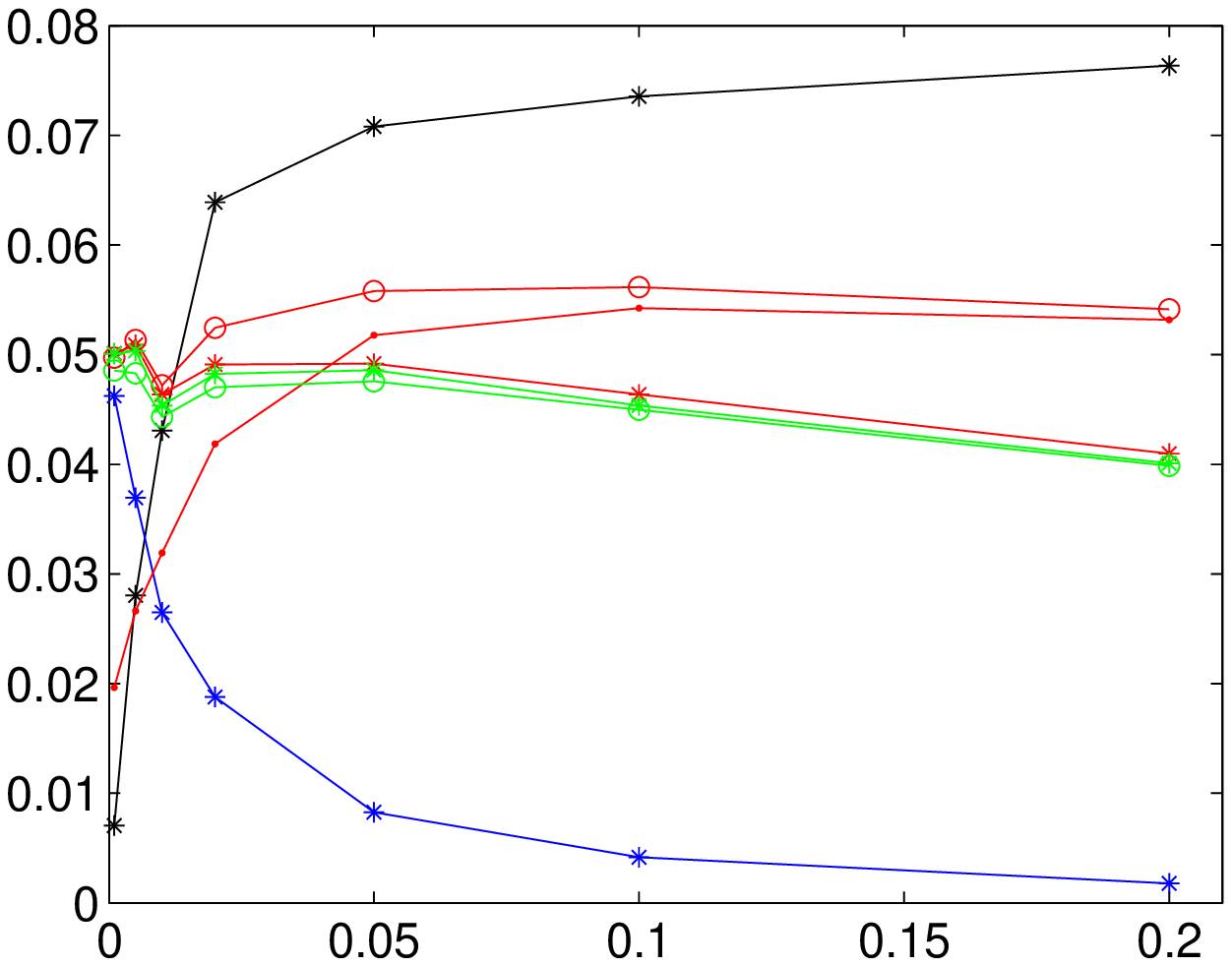,width=7.0cm} &
       \epsfig{file=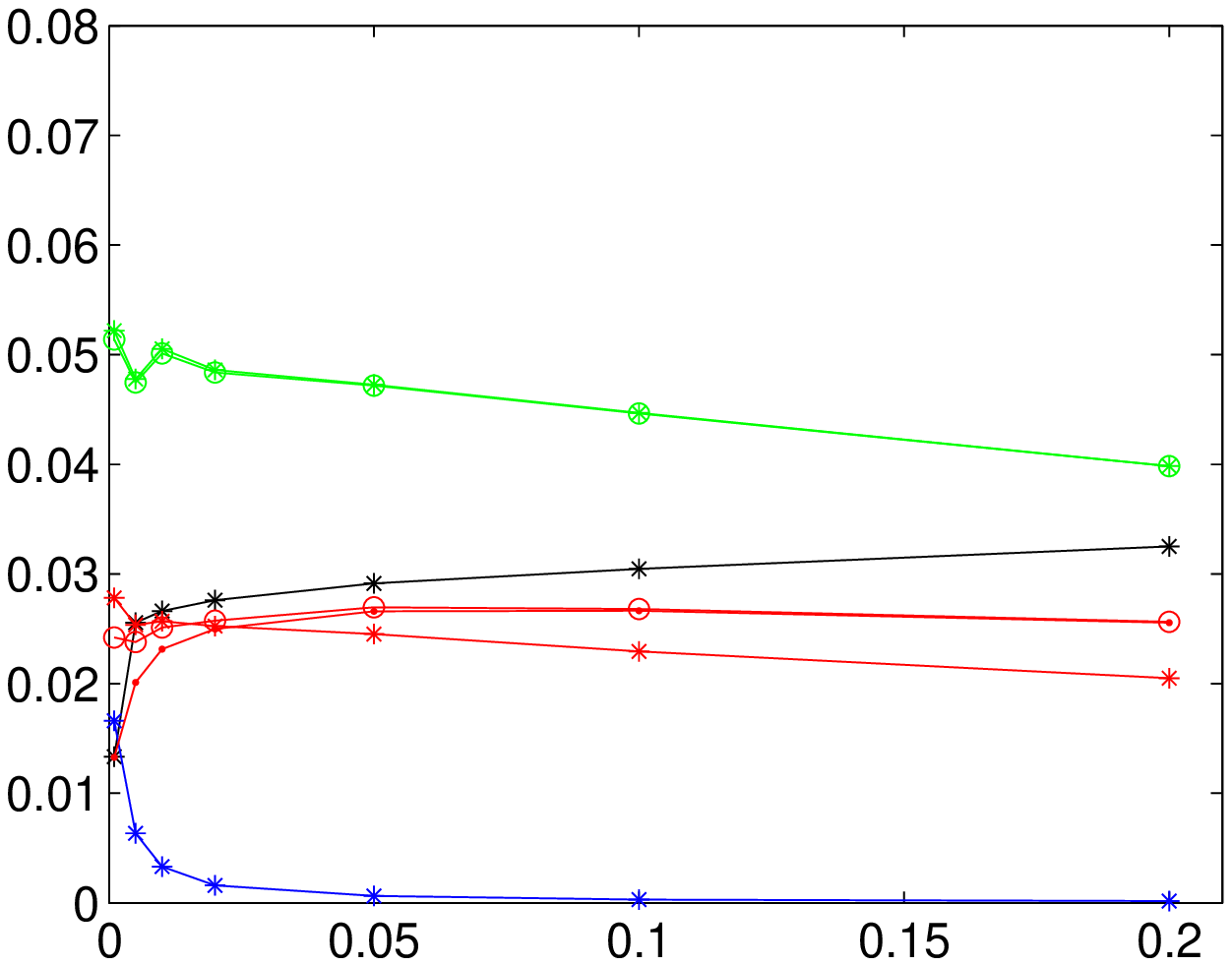,width=7.0cm} &
\\
\multicolumn{1}{l}{\mbox{(e) Power }, n=m=256}
&  \multicolumn{1}{l}{\mbox{(f) Power }, n=m=1024} & \\
 \epsfig{file=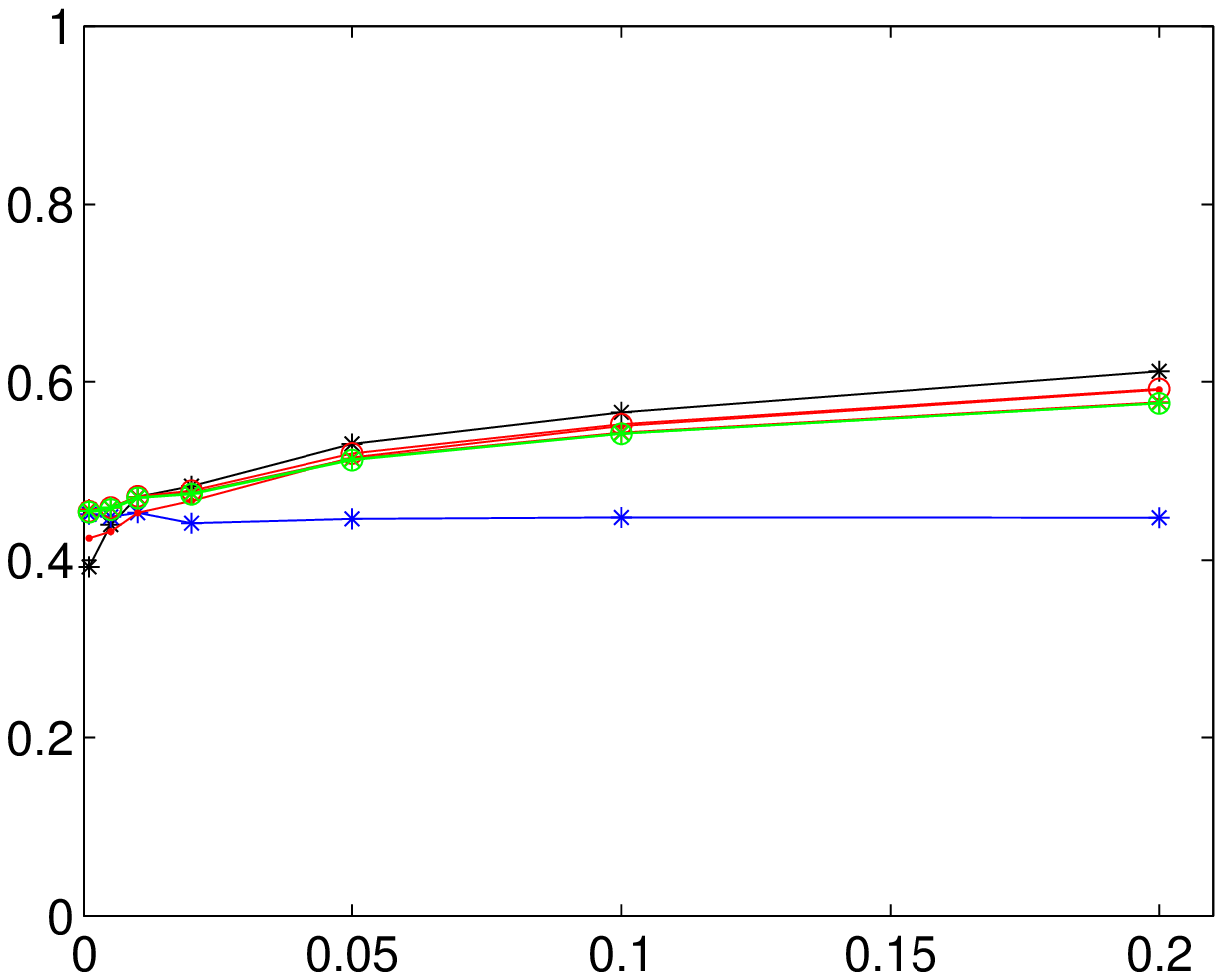,width=7.0cm} &
       \epsfig{file=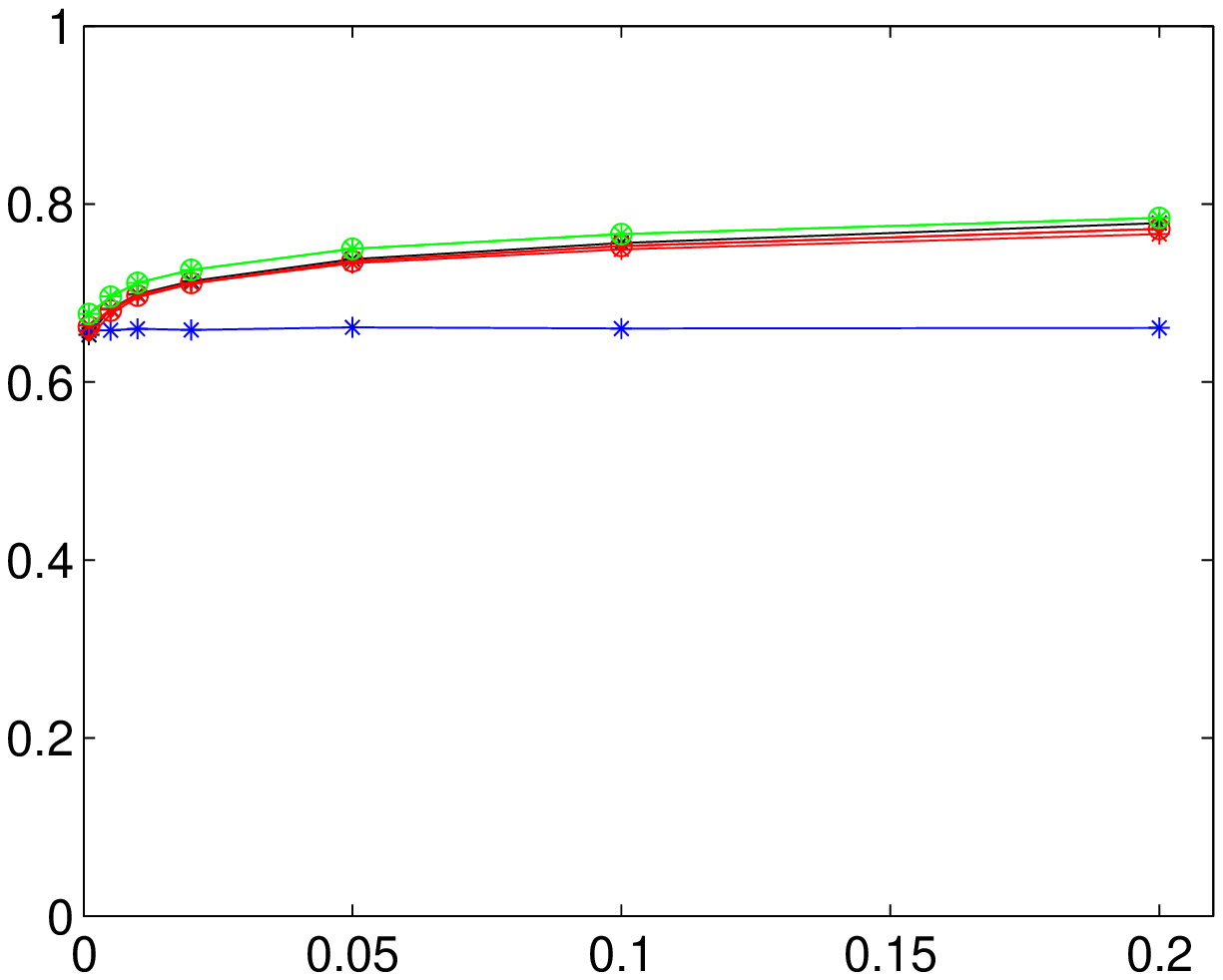,width=7.0cm} \hspace{-5mm} 
       &\epsfig{file=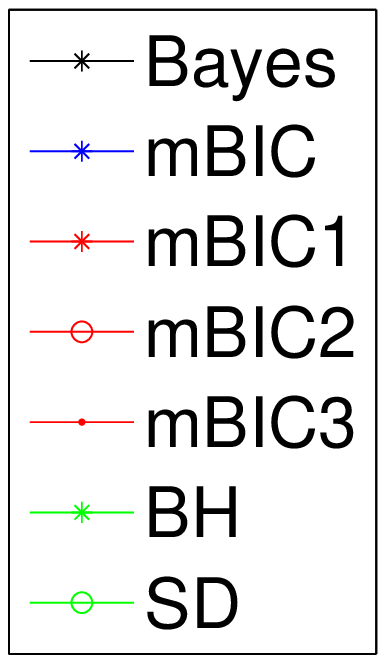,width=7.0cm}
\end{array}$
\end{figure}

\newpage

\begin{figure} [t!]   
\caption{Simulation runs for unknown $\sigma$. Misclassification probability (MP), False Discovery Rate (FDR) and Power for different selection rules and sparsity parameter $p$ at values of $p\in\{0.001, 0.005, 0.01, 0.02, 0.05, 0.1, 0.2\}$.} \label{Fig4}
\vspace{9mm}
$\begin{array}{c@{\hspace{-5mm}}c@{\hspace{0mm}}c} 
\multicolumn{1}{l}{\mbox{(a) MP }, n=m=256}
&  \multicolumn{1}{l}{\mbox{(b) MP }, n=m=1024} & \\
\epsfig{file=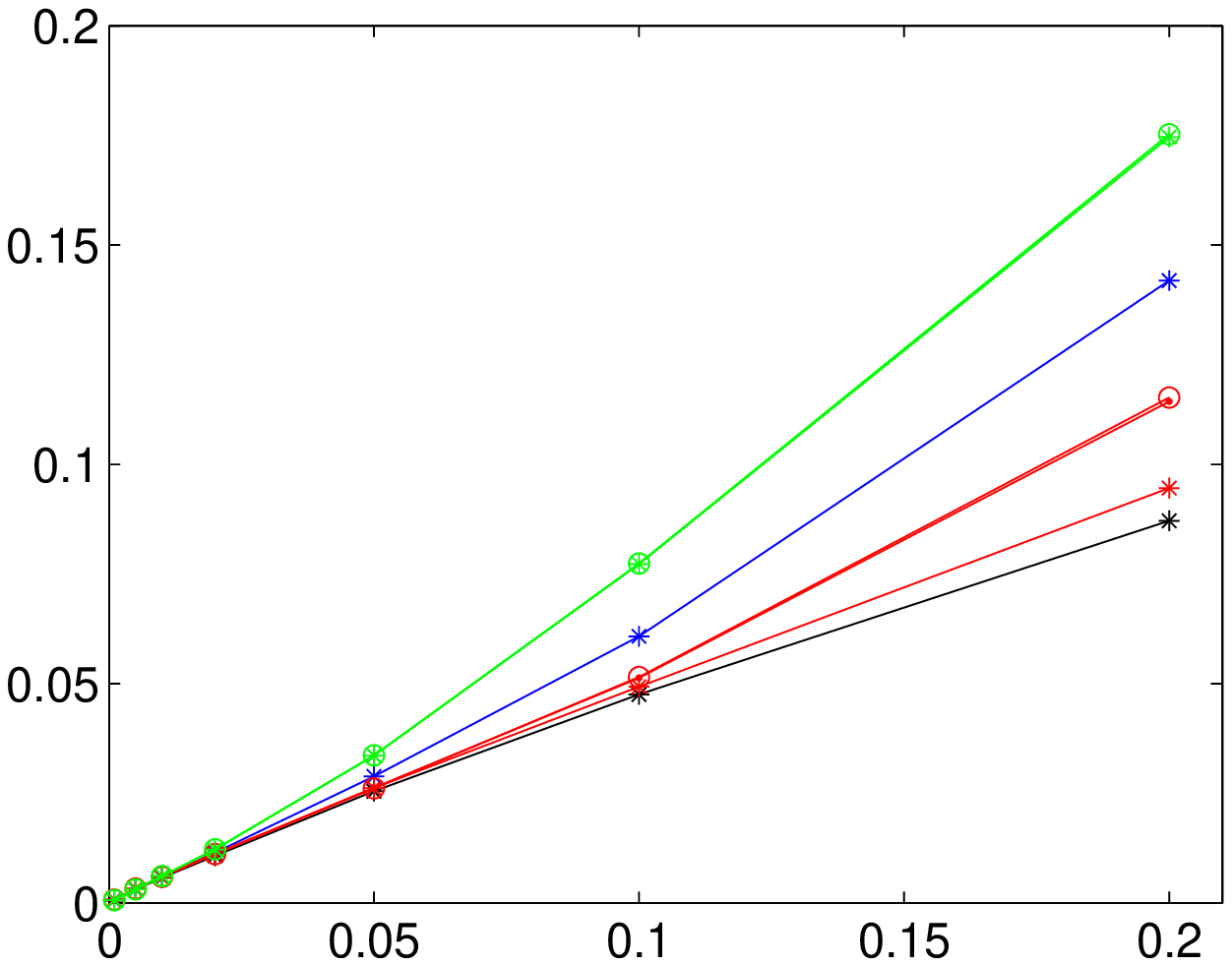,width=7.0cm} &
       \epsfig{file=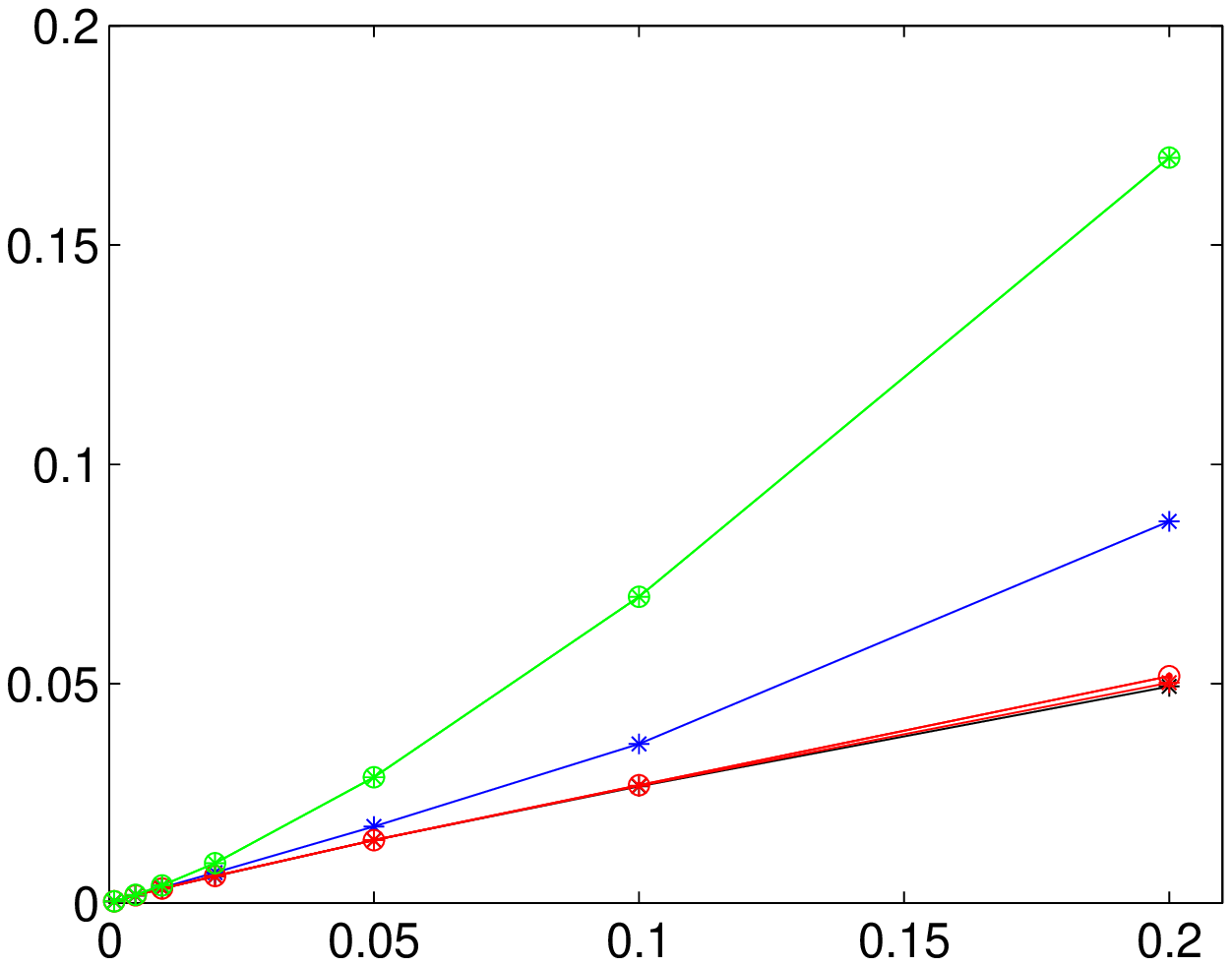,width=7.0cm} &
\\
\multicolumn{1}{l}{\mbox{(c) FDR }, n=m=256}
&  \multicolumn{1}{l}{\mbox{(d) FDR }, n=m=1024} & \\
\epsfig{file=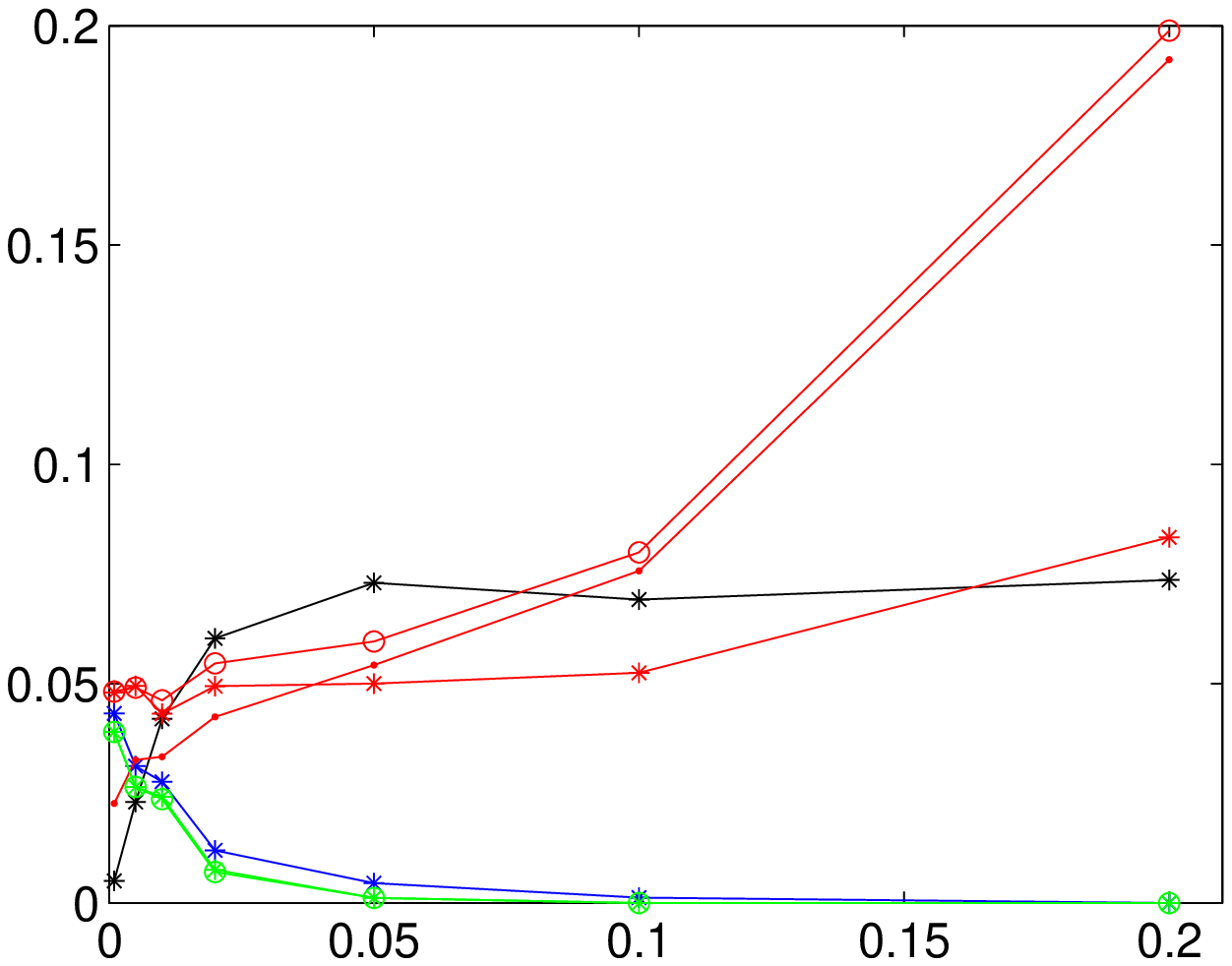,width=7.0cm} &
       \epsfig{file=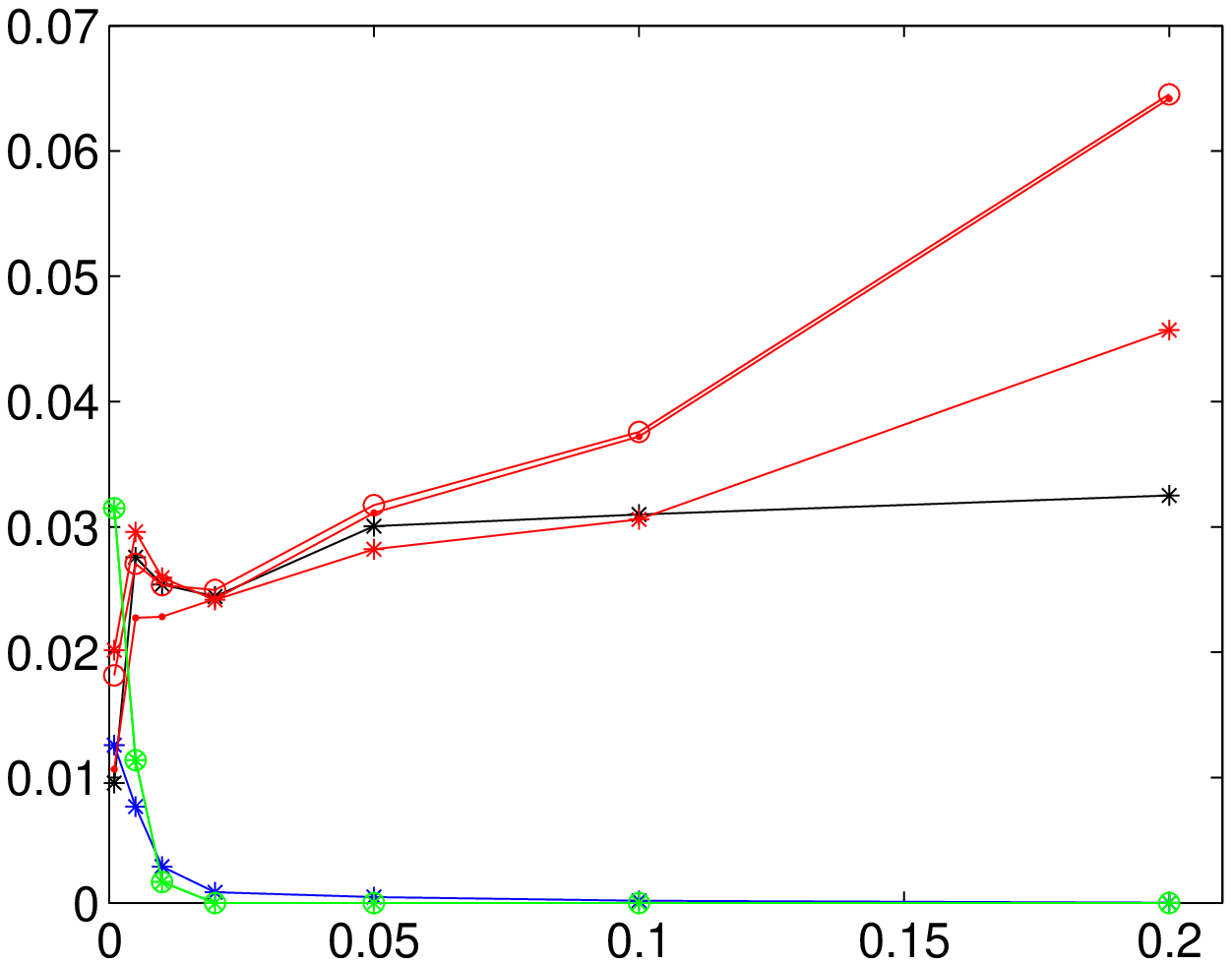,width=7.0cm} &
 \\
\multicolumn{1}{l}{\mbox{(e) Power }, n=m=256}
&  \multicolumn{1}{l}{\mbox{(f) Power }, n=m=1024} & \\
 \epsfig{file=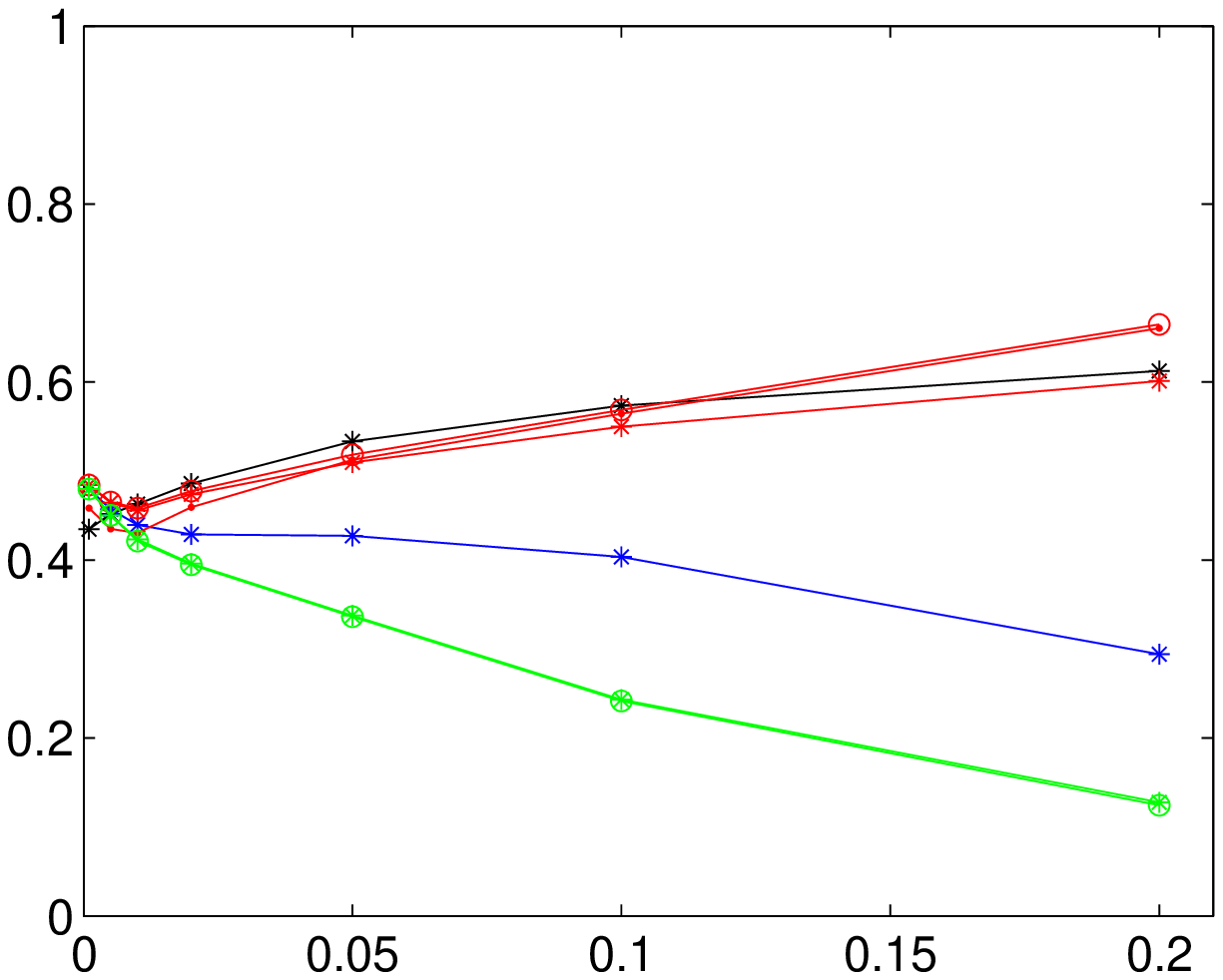,width=7.0cm} &
       \epsfig{file=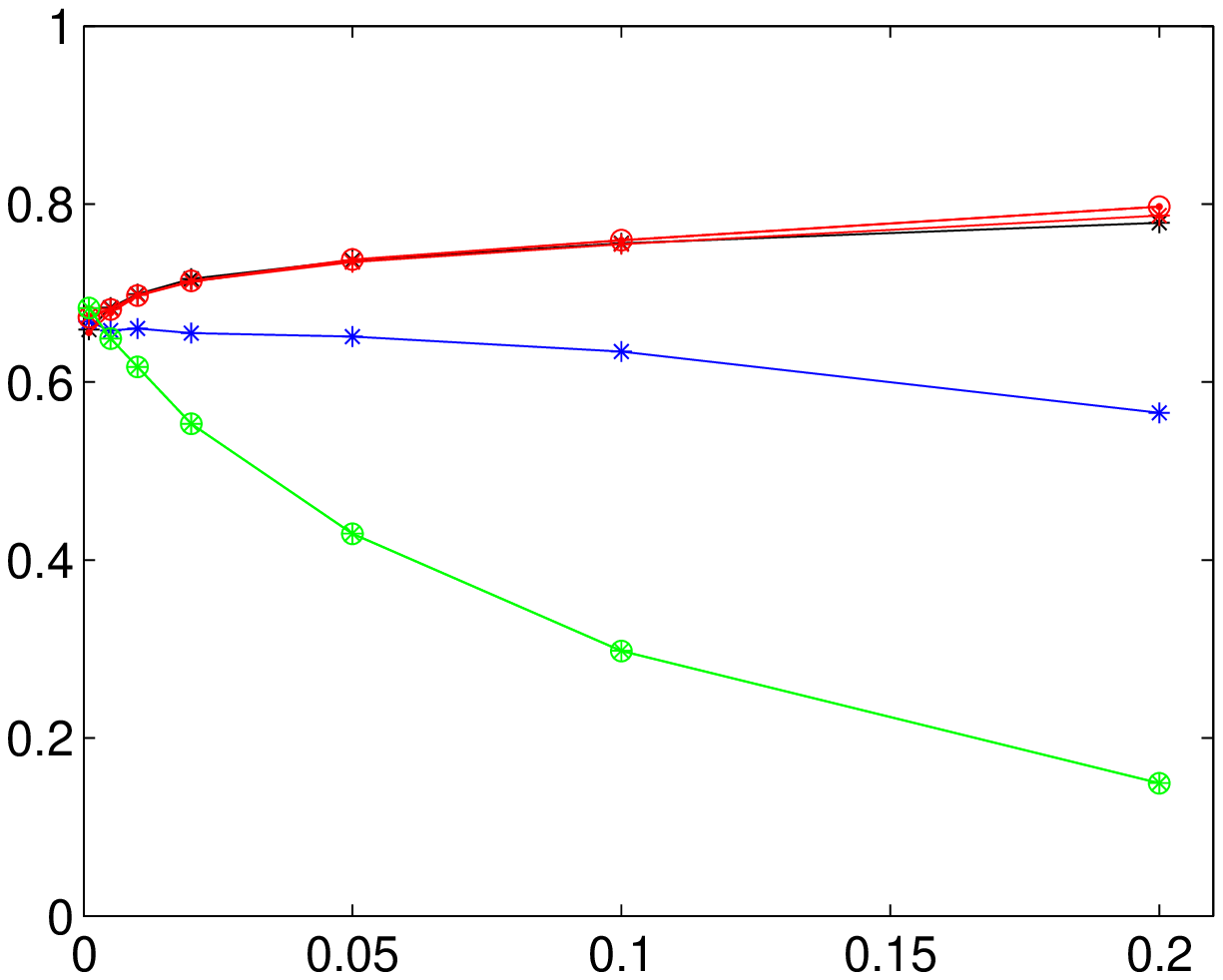,width=7.0cm} \hspace{-5mm} 
       &\epsfig{file=legend.eps,width=7.0cm}
\end{array}$
\end{figure}

\newpage

\bibliographystyle{plain}

\end{document}